\documentclass[11pt, reqno]{amsart}
\usepackage{geometry}                % See geometry.pdf to learn the layout options. There are lots.
\usepackage{graphicx}
\usepackage{amssymb}
\usepackage{epstopdf}
\usepackage{color}
\usepackage[utf8]{inputenc} % set input encoding (not needed with XeLaTeX)
\usepackage{amscd}      % to be able to use "CD" environment
\usepackage[all]{xy}

\usepackage[colorinlistoftodos]{todonotes}
\let\oldtocsection=\tocsection

\let\oldtocsubsection=\tocsubsection

\let\oldtocsubsubsection=\tocsubsubsection

\renewcommand{\tocsection}[2]{\hspace{0em}\oldtocsection{#1}{#2}}
\renewcommand{\tocsubsection}[2]{\hspace{2em}\oldtocsubsection{#1}{#2}}
\renewcommand{\tocsubsubsection}[2]{\hspace{4.5em}\oldtocsubsubsection{#1}{#2}}

\makeatletter
\def\subsection{\@startsection{subsection}{2}
 \z@{.5\linespacing\@plus.7\linespacing}{-.5em}
 {\normalfont\bfseries}}
\def\subsubsection{\@startsection{subsubsection}{3}
 \z@{.5\linespacing\@plus.7\linespacing}{-.5em}
 {\normalfont\bfseries}}
\makeatother

\numberwithin{equation}{section}

\newtheorem{prop}{Proposition}[section]
\newtheorem{lem}[prop]{Lemma}

\newtheorem{cor}[prop]{Corollary}
\newtheorem{them}[prop]{Theorem}

\newtheorem{rem}[prop]{Remark}
\newtheorem{defn}[prop]{Definition}
\newtheorem{numex}[prop]{Example}
\newtheorem{assump}[prop]{Assumption}
\newtheorem{nota}[prop]{Notation}
\newtheorem{lem-de}[prop]{Lemma-Definition}
\newtheorem{de-lem}[prop]{Definition-Lemma}

\newenvironment{pf}{\begin{trivlist}\item[]{\sc Proof.}}%
           {\nolinebreak \hfill $\Box$ \end{trivlist}}

\newcommand{\G}{{\mathcal G}}
\newcommand{\s}{{\bf s}}             % source map
\renewcommand{\t}{{\bf t}}           % target map

\newcommand{\ttt}{{\mathbb T}}

\def \K {{\mathcal{K}}}

\def \dif {\mathrm{d}}
\def \l {\mathfrak{l}}

\def \k{\mathfrak{k}}
\def \g{\mathfrak{g}}
\def \h{\mathfrak{h}}
\def \b{\mathfrak{b}}
\def \d{\mathfrak{d}}

\def \p{\mathfrak{p}}
\def \fs{\mathfrak{s}}
\def \ft{\mathfrak{t}}

\def \q{\mathfrak{q}}
\def \n{\mathfrak{n}}

\def \Ad { {\mathrm{Ad}} }

\DeclareMathOperator \Ker { {\mathrm{Ker}} }
\def\lcf{\lbrack\! \lbrack}
\def\rcf{\rbrack\! \rbrack}
\newcommand{\ra}{\rangle}
\newcommand{\la}{\langle}

\newcommand{\rt}{\mathsf{r}}
\newcommand{\lt}{\mathsf{l}}

\def \lara {\la \cdot , \cdot \ra}           %maps
\newcommand{\Lie}{\mathcal{L}}
\newcommand{\pr}{{{\rm p}}}

\def \B {{\mathcal{B}}}

\def \C {{\mathbb{C}}}

\def \hs {\hspace{.2in}}

\def \O {{\mathcal{O}}}

\def \calL {{\mathcal{L}}}

\def \sB {{\scriptscriptstyle{B}}}
\def \sG {{\scriptscriptstyle{G}}}
\def \sE {{\scriptscriptstyle{E}}}
\def \sA {{\scriptscriptstyle{A}}}
\def \sK {{\scriptscriptstyle{K}}}
\def \sH {{\scriptscriptstyle{H}}}
\def \sD {{\scriptscriptstyle{D}}}
\def \sQ {{\scriptscriptstyle{Q}}}
\def \sP {{\scriptscriptstyle{P}}}
\def \sM {{\scriptscriptstyle{M}}}

\def \sS {{\scriptscriptstyle{S}}}

\def \sR {{\scriptscriptstyle{R}}}
\def \sO {{\scriptscriptstyle{\O}}}
\def \sX {{\scriptscriptstyle{X}}}
\def \sY {{\scriptscriptstyle{Y}}}
\def \sZ {{\scriptscriptstyle{Z}}}
\def \sT {{\scriptscriptstyle{T}}}

\def \sGo {{\scriptscriptstyle{G,1}}}
\def \sGt {{\scriptscriptstyle{G,2}}}

\def \sAN {{\scriptscriptstyle{AN}}}
\def \sBD {{\scriptscriptstyle{BD}}}
\def \sI {{\scriptscriptstyle{I}}}

\def \sL {{\scriptscriptstyle{L}}}
\def \pist {{\pi_{\rm st}}}

\title{Dirac geometry and integration of Poisson homogeneous spaces}
\author{Henrique Bursztyn}
\address{IMPA, Estrada Dona Castorina 110, Rio de Janeiro, 22460-320, Brazil.}
\email{henrique@impa.br}

\author{David Iglesias-Ponte}
\address{Departamento de Matem\'aticas, Estad\'{\i}stica e Investigaci\'on Operativa, Universidad de La Laguna, Spain}
\email{diglesia@ull.edu.es}

\author{Jiang-Hua Lu}
\address{Department of Mathematics, The University of Hong Kong, Pokfulam Road, Hong Kong}
\email{jhlu@maths.hku.hk}

\begin{document}

\begin{abstract}
\noindent Using tools from Dirac geometry and through an explicit construction, we show
that every Poisson homogeneous space of any Poisson Lie group admits an integration to a symplectic groupoid. Our theorem
follows from a more general result which relates, for a principal bundle $M\to M/H$, integrations of a Dirac structure on $M/H$ to
$H$-admissible integrations
of its pullback Dirac structure on $M$ by pre-symplectic groupoids. Our construction gives a distinguished class of explicit real or holomorphic
pre-symplectic and symplectic groupoids over semi-simple Lie groups and some of their homogeneous spaces, including their
 symmetric spaces, conjugacy classes, and flag varieties. In a more general framework,  we also show integrability of all
homogeneous spaces of
${\mathcal{LA}}^\vee$-Lie groups in the sense of E. Meinrenken.
\end{abstract}

\maketitle

%%%%%%%%%%%%%%%%%%%%%%%%%%%%%

\tableofcontents
\addtocontents{toc}{\protect\setcounter{tocdepth}{1}}
\section{Introduction}\label{s:intro}
Generalizing the association of Lie algebras to Lie groups, each Lie groupoid
over a manifold $M$ defines a Lie algebroid over $M$ (see e.g. \cite{Mac,MM}). A Lie algebroid is said to be {\it integrable}
if it is the Lie algebroid of some Lie groupoid. In such a case there always exists a {\it source-simply-connected
integration}, i.e., a Lie groupoid integrating the given Lie algebroid whose source-fibers are simply connected\footnote{Throughout the paper, a
 topological space
$X$ is said to be simply connected if $\pi_0(X) = 0$ and $\pi_1(X) = 0$.},
and such an integration is unique up to isomorphisms \cite{CF1, MM2}.
In contrast with Lie algebras, not every Lie algebroid is integrable.
Obstructions to integrability are explained in \cite{CF1}.

A Poisson structure $\pi$ on a manifold $M$ makes its co-tangent bundle $T^*M$ into a Lie algebroid, which we denote by $T_\pi^* M$,
and the Poisson manifold $(M, \pi)$ is said to be {\it integrable}
if so is the Lie algebroid $T_\pi^*M$.
%is integrable.
Obstructions to integrability of Poisson manifolds are studied in \cite{CF2} (see also \cite{CaFe}).

When a Poisson manifold $(M, \pi)$ is integrable, the
source-simply-connected integration $\G$ of the Lie algebroid $T_\pi^*M$ has the additional structure of a
{\it symplectic groupoid over $(M, \pi)$}: by definition, a symplectic groupoid \cite{CDW, Mac-Xu, We87}
over a Poisson manifold
$(M, \pi)$ is a pair $(\G, \omega)$, where $\G$ is a Lie groupoid (not necessarily
source-simply-connected)
and $\omega$ is a multiplicative symplectic form (see $\S$\ref{ss:kernel}) such that the target map
$(\G, \omega) \to (M, \pi)$ is Poisson.
In this context, we also say that
the symplectic groupoid $(\G, \omega)$ {\it is an integration of} (or {\it integrates}) the Poisson manifold $(M, \pi)$
(or the Poisson structure $\pi$).

Symplectic groupoids are thus global objects
corresponding to Poisson manifolds, {extending the way} Lie
groups integrate Lie algebras. Besides
their key role in a quantization scheme for
Poisson manifolds (see e.g. \cite{BatesWein,Hawk}), symplectic groupoids
have become an indispensable tool in the modern study of
Poisson structures (see e.g. \cite{CFM}). A central problem in Poisson geometry is to identify classes of Poisson manifolds
 which are integrable and devise methods to construct their symplectic groupoids.

An important class of Poisson manifolds consists of Poisson homogeneous spaces of Poisson Lie groups, originally introduced by
Drinfeld \cite{Drinfeld, Drinfeld2} as semi-classical limits of homogeneous spaces of quantum groups.
Different methods have been employed in the literature to tackle their integrability, producing partial results e.g. in
 \cite{BCST,FI, Lu3, Lu-note,S,X} (see also \cite{BCST2,BCQT} for applications to quantization).  In this paper, we prove
integrability of Poisson homogeneous spaces in its full generality.

\smallskip
\noindent
{\bf Theorem.} {\it
Every Poisson homogeneous space of any Poisson Lie group is integrable.
}

\smallskip
In addition to proving integrability, we construct explicit symplectic groupoids
integrating Poisson homogeneous spaces.
Compared to previous works, the novelty of our approach lies in its generality and
in its use of Dirac geometry.

We in fact prove
integrability of a much wider class of Lie algebroids over homogeneous spaces of Lie groups, namely
 {\it homogeneous spaces of ${\mathcal{LA}}^\vee$-Lie groups}  in the sense of E. Meinrenken \cite{Mein}.
To achieve this, we first establish in $\S$\ref{s:PP}  an equivalence,  for any
principal bundle $q: M \to M/H$ and via  a quotient operation,
between integrability of Lie algebroids over $M/H$  and {\it $H$-integrability} (as in \cite{FS19}) of
their pullback Lie algebroids over $M$ by $q$.
The Lie algebroids over a homogeneous space $G/H$ of an ${\mathcal{LA}}^\vee$-Lie group $G$ have the distinguished property that
their pullback Lie algebroids are action Lie algebroids over $G$, which, under some mild
assumptions on $G/H$, are shown to be $H$-integrable through an explicit construction of $H$-integrations.
These assumptions are circumvented in Theorem \ref{th:inte-BHl-all} (see also Theorem \ref{th:inte-LAvee})
to prove integrability of
homogeneous spaces of ${\mathcal{LA}}^\vee$-Lie groups in full generality. That integrability of Poisson homogeneous spaces follows
as a special case is explained in $\S$\ref{ss:main-inte}.

To construct explicit symplectic groupoids for Poisson homogeneous spaces, we turn to Dirac geometry.
Recall (see $\S$\ref{ss:kernel}) that a Dirac structure on a manifold $M$ is a vector sub-bundle $E$ of $TM \oplus T^*M$
satisfying suitable compatibility conditions.
Dirac structures are generalizations of Poisson structures, and our paper relies heavily on two of their key features.
First, any Dirac structure $E$ is a Lie algebroid, and if $E$ is integrable as a such, its source-simply-connected integration
carries the structure of a {\it pre-symplectic groupoid} \cite{BCWZ}, generalizing integrations of Poisson manifolds by
symplectic groupoids.
The second feature is the existence of pullback and pushforward operations in Dirac geometry, see e.g. \cite{villa}.
We show in $\S$\ref{ss:PP-Dirac} that for
a principal bundle $q: M \to M/H$, integrations of a Dirac structure on $M/H$ can be obtained as
quotients of {\it $H$-admissible integrations} (Definition \ref{defn:admi1}) of its pullback Dirac structure on $M$.
When applied to a Poisson structure $\pi$ on $M/H$, we obtain symplectic groupoids of $(M/H, \pi)$ as quotients of
$H$-admissible pre-symplectic groupoids integrating the pullback of $\pi$ to $M$ as a Dirac structure (Corollary \ref{co:pi-inte}).

In $\S$\ref{s:homog-G} and $\S$\ref{s:explicit}, we apply the general results in $\S$\ref{s:PP} - $\S$\ref{s:Dirac}
to Poisson homogeneous spaces of a Poisson Lie group $(G, \pi_\sG)$.
The pullback Dirac structure on $G$ of the Poisson structure on any
Poisson homogeneous space of $(G, \pi_\sG)$ is $(G, \pi_\sG)$-affine, in the sense that
the group multiplication
$m:  (G, \pi_\sG) \times (G, E) \rightarrow (G, E)$
of $G$ is a (forward strong) Dirac map.
 Under the assumption that $(G, \pi_\sG)$ has a Drinfeld double (Definition \ref{defn:D-double}),
we construct in Theorem \ref{thm:presymp}
%$\S$\ref{ss:explicit-affine}
explicit pre-symplectic
groupoids for every $(G, \pi_\sG)$-affine Dirac structures on $G$. Under an additional assumption on a Poisson homogeneous space $(G/H, \pi)$ of $(G, \pi_\sG)$, we apply the
quotient construction in $\S$\ref{ss:PP-Dirac} and
obtain in Theorem \ref{thm:GH} explicit symplectic groupoids integrating $(G/H, \pi)$.

We remark that techniques from Dirac geometry have been used to solve or shed light on other problems from Poisson geometry, notably in
\cite{MeinLectures} by E. Meinrenken. In fact, in this paper we make crucial use of the reformulation in \cite{Mein, MeinLectures}
of Drinfeld's classification of Poisson homogeneous spaces in terms of equivariant Manin triples and Harish-Chandra sub-pairs
(see Proposition \ref{prop:Hl}).

Results in $\S$\ref{s:PP} - $\S$\ref{ss:explicit-GH} are presented in the category of real smooth manifolds, but we observe in
$\S$\ref{s:hol} that they also carry through in the holomorphic category.
$\S$\ref{s:examples} is devoted to examples, including those previously obtained in \cite{BCST, Lu3,Lu-note} by different methods, and
abundant new examples of
symplectic and pre-symplectic groupoids over semi-simple Lie groups and some of their most
important homogeneous spaces such as symmetric spaces, conjugacy classes, and orbits in flag varieties.

After the first version of our paper was posted in the arXiv, some of the results there were generalized in \cite{Alv}. See Remarks \ref{re:Daniel-1} and
\ref{re:Daniel-2}  for more details.

\medskip

\noindent{\bf Acknowledgments}: Bursztyn thanks CNPq, Faperj and the Fulbright Foundation for financial support.
Lu is partially supported by the RGC of the Hong Kong  (GRF HKU 703712 and 17304415).
Iglesias thanks partial support by European Union (Feder) grant PGC2018- 098265-B-C32.
We are grateful to several institutions for hosting us during various stages of this project, including
EPFL, ETH (Poisson 2016), U. La Laguna, ICMAT, BIRS, Fields Institute (Poisson 2018) and U. C. Berkeley.
We thank D. Alvarez, A. Cabrera, R. Fernandes,  E. Meinrenken and A. Weinstein for stimulating discussions.

\medskip

\noindent
{\bf Notation and conventions}:
%\subsection{Notation and conventions}
Let $\G\rightrightarrows M$ be a Lie groupoid. We denote by
$\s$ and $\t$ the source and target maps, by $m \colon
\G^{(2)} = \{(g, h) \in \G \times \G|\;\s(g)=\t(h)\}\to\G$ the multiplication,
by $\mathrm{inv}:\G\to\G$ the inversion map, and by $\iota:M\to \G$ the
identity section. We write $g h:=m (g,h)$ for the product,
$1_x:=\iota(x)$ for the unit element over $x\in M$, and $g^{-1}:=\mathrm{inv}(g)$. Given $g\in \G$, we have the right
and left translations
\[
\rt_g: \s^{-1}(\t(g))\to
\s^{-1}(\s(g)), \hspace{.15in} \lt_g: \t^{-1}(\s(g))\to \t^{-1}(\t(g)).
\]
When there is no risk of confusion, we simplify our notation and write e.g. $\rt_g (X)$ instead of $\dif \rt_g|_h (X)$,
where $X$ is a tangent (multi-)vector at $h$.

The Lie algebroid of $\G$ is the vector bundle $A
= \Ker(\dif \s)|_\sM$ with  anchor $\rho= \dif \t |_\sA$ and bracket
$[\cdot,\cdot]$ on $\Gamma(A)$ through its identification with {\em right-invariant vector fields} on $\G$ via
\[
\Gamma(A) \ni u \mapsto u^\rt, \;\; u^\rt|_g=  \rt_g (u|_{\t(g)}).
\]
We will often use the notation $A={\rm Lie}({\mathcal{G}})$.

Similarly, for $u \in \Gamma(A)$, define the left invariant vector field $u^\lt$ on $\G$ via
$u^\lt|_g= \lt_g(u|_{\s(g)})$.
The ``opposite'' Lie algebroid $A^{\scriptscriptstyle{op}}$ has underlying vector bundle $\Ker(\dif\t)|_\sM$, anchor $\rho^{\scriptscriptstyle{op}}
= \dif \s|_{A^{\scriptscriptstyle{op}}}$ and bracket
$[\cdot,\cdot]^{\scriptscriptstyle{op}}$ on $\Gamma(A^{\scriptscriptstyle{op}})$ via its identification with
left-invariant vector fields on $\G$.
% $v\mapsto v^\lt$, $v^\lt|_g=\dif \lt_g|_{\s(g)}(v)$.
%The decompositions $T\G|_M = TM\oplus \Ker(\dif\s)|_\sM =  TM\oplus \Ker(\dif\t)|_\sM $ induce a natural vector bundle
%isomorphism from $A$ to $A^{\scriptscriptstyle{op}}$
%(obtained from the identification of both with the normal bundle to $\iota: M\to \G$)
%given by $ u \mapsto u-\rho(u)$, which is a Lie algebroid anti-homomorphism.
The differential of $\mathrm{inv}$ gives the Lie algebroid isomorphism $\mathrm{inv}: A
\to A^{\scriptscriptstyle{op}}$.
% $u\mapsto \rho(u)-u$.

We keep the same convention for Lie groups and Lie algebras. In particular,  for a Lie group $G$ with $\g = T_eG$,
the Lie bracket on $\g$ satisfies
\begin{equation}\label{eq:uuu}
[u_1, u_2]^\rt = [u_1^\rt, \, u_2^\rt], \;\;\; [u_1, u_2]^\lt =- [u_1^\lt, \, u_2^\lt]\hspace{.2in} u_1, u_2 \in \g.
\end{equation}
Note in particular that our convention on the Lie bracket on $\g$ implies that
\begin{equation}\label{eq:exp}
\frac{d}{d\epsilon}\Big |_{\epsilon=0}{\rm Ad}_{\exp(\epsilon u_1)} u_2 = -[u_1, u_2], \hs  u_1, u_2 \in \g.
\end{equation}
Moreover, if $\theta^\lt$ and $\theta^\rt$ are respectively the left-invariant and right invariant Maurer-Cartan $1$-forms on $G$, then
$d\theta^\lt -[\theta^\lt, \theta^\lt] = 0$ and $d\theta^\rt +[\theta^\rt, \theta^\rt] = 0$. Here for a $\g$-valued $1$-form $\theta$ on a manifold
$M$,  $[\theta, \theta]$ denotes the $\g$-valued $2$-form on $M$ given by $[\theta, \theta](X, Y) = [\theta(X), \theta(Y)]$ for $X, Y \in {\mathfrak{X}}^1(M)$.
%For $\xi \in \g^*$, let $\xi^\rt$ and $\xi^\lt$ be the respective right and left invariant $1$-forms on $G$ with values $\xi$ at $e$.

We say that an action of a Lie group $K$ on a manifold $P$ is {\em principal} if it is free and the quotient space has a smooth structure such that the quotient map $P \to P/K$ is a submersion. In this context, we make the following observation to be used in $\S$\ref{ss:PP-groupoids}.

\begin{lem}\label{lem:principal}
Let $P$ and $Q$ be manifolds with actions by a Lie group $K$, and let $f:P \to Q$ be a $K$-equivariant map.
If the $K$-action on $Q$ is principal, so is the $K$-action on $P$.
\end{lem}

\begin{proof}
Suppose first that the principal bundle $Q\to Q/K$ admits a section $\kappa$; denote its image by $S\subset Q$.
Since $f$ takes orbits to orbits, it is necessarily transverse to $\kappa$. So the fibered product $P\times_Q (Q/K) = f^{-1}(S)$
is a smooth manifold, which can be naturally identified with $P/K$. This smooth structure on $P/K$ is independent of the
choice of section (since any two sections are related by a map $Q\to K$, and multiplication by this map gives a diffeomorphism
between the corresponding ``slices'' in $P$). In general, one uses local sections of $Q\to Q/K$ to define an atlas on $P/K$
such that $P\to P/K$ is a principal $K$-bundle.
\end{proof}

\begin{rem}\label{rem:nonhausdorff} \em
It is common that the definition of a Lie groupoid allows the space of arrows to be a non-Hausdorff smooth manifold (though Hausdorffness is assumed on the space of units). We remark that if $P$ is such a manifold in the previous lemma, the argument in the proof still shows that, if the $K$-action on $Q$ is principal, the quotient $P/K$ has a (possibly non-Hausdorff) smooth structure for which $P\to P/K$ is a submersion.
\hfill $\diamond$
\end{rem}

%%%%%%%%%%%%%%%%%%%%%%%%%%%%%%%%%%%%%%%%%%%%%%%%
\section{Integrations of Lie algebroids via pullbacks}\label{s:PP}
Let $q: M \to M/H$ be a principal bundle. In this section, we establish
an equivalence between integrability of Lie algebroids over $M/H$ and $H$-integrability
of their pullbacks to $M$  in the sense of  \cite{FS19}.

The equivalence mentioned above
is a consequence of an equivalence between isomorphism classes of Lie algebroids over $M/H$ and isomorphism classes
 of {\it $H$-Lie algebroids pairs} over $M$ (see Definition \ref{defn:H-pair}), and a similar equivalence for Lie groupoids.
Such equivalences are hinted, at a deeper level and in the setting of algebraic groups and schemes,
in \cite[$\S$1.8.9]{BB}, where it is shown that the pullback operation by $q$ establishes  an equivalence between {\it $D$-algebras on $M/H$} and
{\it Harish-Chandra $D$-algebras on $M$}. The corresponding statements for Lie algebroids and Lie groupoids have appeared in the literature in
certain forms,   for
example, in \cite[$\S$2.4.3 and $\S$2.4.4]{Rob} for Lie algebroids and partially in \cite[$\S$7.3]{FS19} for Lie groupoids,
but we think it is of interest to have a review of the
 precise equivalences in one place, as we do in this section. For
pullbacks and quotients of Lie algebroids and Lie groupoids, see e.g. \cite{Higgins-Mackenzie, Mac}.

%%%%%%%%%%%%%%%%%%%%%%%%%%%%%
\subsection{$H$-Lie algebroid pairs}\label{ss:HCpairs}
Let $H$ be a Lie group with Lie algebra $\h$, and let $M$ be a manifold with a left $H$-action denoted as
$H \times M \to M,  (h, x) \mapsto h\cdot x$. For $u \in \h$, let $u_\sM \in {\mathfrak{X}}^1(M)$ be given by
\begin{equation}\label{eq:uM}
u_\sM|_x := \frac{d}{d\epsilon}\Big|_{\epsilon = 0} \exp (\epsilon u)\cdot x, \;\;\;\; \; x\in M.
\end{equation}
In this subsection, we do not make any additional assumptions on the $H$-action.

By an {\it $H$-Lie algebroid over $M$} we mean a Lie algebroid $A$ over $M$ with an $H$-action by Lie algebroid automorphisms covering the
$H$-action on $M$. For an $H$-Lie algebroid $A$, one has the induced action of $\h$ on $\Gamma(A)$ by
\begin{equation}\label{eq:h-GaA-0}
(u\cdot a)|_x = - \frac{d}{d\epsilon}\Big |_{\epsilon = 0} \exp(\epsilon u)\cdot \left(a|_{\exp(-\epsilon u)\cdot x}\right), \hs x \in M.
\end{equation}
Let $\h \ltimes M$ be the action Lie algebroid over $M$ with anchor
\[
\rho_\sM: \; \h \ltimes  M \rightarrow TM, \ (u, x) \mapsto u_\sM|_x.
\]
Then $\h \ltimes M$ is an $H$-Lie algebroid over $M$ with the $H$-action given by
\begin{equation}\label{eq:adjact}
 h\cdot (u, x)= (\Ad_h(u), \, h\cdot x).
\end{equation}
%The following definition, formulated slightly differently,  can be found in \cite[$\S$1.8.4]{BB}.

\begin{defn}\label{defn:H-pair}
{\em 1) An {\it $H$-Lie algebroid pair over $M$}
 is a pair $(A, \psi)$, where $A$ is an $H$-Lie algebroid over $M$, and
$\psi: \h \ltimes M \rightarrow A$ is an $H$-equivariant Lie algebroid morphism covering the identity map of $M$,
such that
\begin{equation}\label{eq:ua-0}
u \cdot a = [\widetilde{\psi}(u),\, a], \hs u \in \h, \, a \in \Gamma(A),
\end{equation}
where for $u \in \h$,  $\widetilde{\psi}(u) \in \Gamma(A)$ is given by $\widetilde{\psi}(u)|_x = \psi(u, x)$ for $x \in M$.

2) Two $H$-Lie algebroid pairs $(A, \psi)$ and $(A', \psi')$ over $M$ are said to be isomorphic if there is an $H$-equivariant Lie algebroid isomorphism
$I: A \to A'$ covering the identity map of $M$ such that
$I \circ \psi = \psi': \;\h \ltimes M \to A'$.
}
\end{defn}

%\hen{I wonder if it is worth introducing the terminology ``Harish-Chandra Lie algebroids (pairs?)'' if we immediately simplify it and don't seem to use it again...  %There are also ``Harish-Chandra pairs'' later which may make the terminology a bit confusing... It might be simpler to just use ``H-Lie algebroid pairs'' from the %outset, and mention the parallel story of Harish-Chandra $D$-algebras (now in the intro to the section) in a remark, the one below could be a good place (or a %separate one)... }

\begin{rem}\label{rem:H-principal}
{\rm An $H$-Lie algebroid pair $(A, \psi)$ as in Definition \ref{defn:H-pair} is called a
 {\it Harish-Chandra Lie algebroid} in \cite[$\S$1.8.4]{BB}, but our choice of the terminology is due to our emphasis on the role
played by $\psi$ and the need for a corresponding term for Lie groupoids (see Definition \ref{defn:pairs-G}).
When $\psi: \h\ltimes M \to A$ is fiber-wise injective, the pair $(A, \psi)$ is also called
an {\it $H$-principal Lie algebroid} and $\psi$ an {\it action morphism} in \cite{FS19}.  Without the name,
$H$-Lie algebroid pairs  are also considered
in \cite[$\S$2.4.3]{Rob}, where $\psi$ is called a {\it generator} for the
$H$ action on $A$. See also \cite{MPO}.  A similar definition was introduced for Courant algebroids in \cite[$\S$2.8]{Mein}.
We also refer to
\cite[Appendix]{BF:coupling} for a detailed discussion on inner Lie algebra and
Lie group actions on Lie algebroids.
}\hfill $\diamond$
\end{rem}
%2) When $M$ is a one point space,
%an $H$-Lie algebroid pair over $M$ is just a pair $(\l, \psi)$, where $\l$ is a Lie algebra with an action of $H$ by Lie algebra automorphisms,
%and $\psi: \h \to \l$ is an $H$-equivariant Lie algebra homomorphism such that the $H$-action on $\l$ integrates the adjoint action of
%$\h$ via $\psi$. In such a setting, one typically requires $\psi: \h \to \l$ to be injective, and $(H, \l)$ is called a {\it Harish-Chandra pair}.
%}
%\end{rem}

\subsection{Pullbacks and pushforward of Lie algebroids}\label{ss:PP-algebroids}
In this section, we assume that $H$ is a Lie group acting on a manifold $M$  so that
\[
q: \; M \rightarrow M/H
\]
is an $H$-principal bundle\footnote{All group actions in this paper are left actions unless otherwise specified,
although the quotient space for such an action $H\times M \to M$ is
denoted by $M/H$ instead of $H\backslash M$.}. If $(A, \psi)$ is an $H$-Lie algebroid pair over $M$, then it follows from
$\psi$ being
 a Lie algebroid morphism that
\begin{equation}\label{eq:rho-rho-psi}
\rho_\sA \circ \psi = \rho_\sM: \, \h \ltimes M \rightarrow TM,
\end{equation}
where $\rho_\sA: A \to TM$ is the anchor of $A$. Thus
$\psi$ is automatically fiber-wise injective, and $(A, \psi)$ is an $H$-principal Lie algebroid in the sense of
\cite{FS19}.

%In what follows, isomorphisms of Lie algebroids over $M/H$ are understood to be
%Lie algebroid isomorphisms covering the identity map of $M/H$.

Let $B$ be any Lie algebroid over $M/H$ with anchor $\rho_{{\sB}}$. By \cite[Section 1]{Higgins-Mackenzie}, the {\it pullback Lie algebroid} of $B$ by $q$ is the vector bundle
$q^!B =TM \times_{T(M/H)} B$ over $M$, i.e.,
\[
(q^!B)|_x = \{(X, b)\, |\,  X \in T_xM, \, b \in  B|_{q(x)}, \,  \dif q(X) = \rho_{{\sB}}(b)\},\hs x \in M,
\]
with the unique Lie algebroid structure such that $q^!B \to TM \times B$, defined by the fiber-wise inclusions, is a Lie algebroid morphism covering
$M \to M \times (M/H), x \mapsto (x, q(x))$.
%and the Lie bracket on $\Gamma(q^!B)$ determined by
%\[
%[(X_1, f_1(b_1\circ q)), (X_2, f_2(b_2 \circ q)] = ([X_1, X_2], \, X_1(f_2)(b_2\circ q)-X_2(f_1)(b_1\circ q) + f_1f_2 ([b_1, b_2]\circ q))
%\]
%where for $i = 1, 2$, $X_i \in \mathfrak{X}^1(M)$, $f_i \in C^\infty(M)$, and $b_i \in \Gamma(B)$.
In particular, the anchor of $q^!B$ is $q^!B \to TM, (X, b) \mapsto X$, and one has
${\rm rank}(q^!B) = \dim H + {\rm rank}(B)$.
The next result
%Lemma \ref{lem:AA}
is a special case of \cite[Theorem 4.3.6]{Mac}.

\begin{lem}\label{lem:AA}
For a Lie algebroid $A$ over $M$ and a Lie algebroid $B$ over $M/H$, a vector bundle morphism $\phi: A \to B$ covering
$q: M \to M/H$ is a Lie algebroid morphism if and only if the vector bundle morphism
\[
A \rightarrow q^!B, \; a \mapsto (\rho_\sA(a), \, \phi(a)),
\]
is a Lie algebroid morphism covering the identity map of $M$. In particular,
\begin{equation}\label{eq:chi}
\chi: \; q^!B \rightarrow B, \,(X, \, b) \mapsto b, \hs x \in M, \, (X, b) \in (q^!B)|_x,
\end{equation}
is a Lie algebroid morphism covering $q: M \to M/H$.
\end{lem}

The following lemma can be proved directly from the definitions.

\begin{lem}\label{lem:qB-pair}
Every Lie algebroid $B$ over $M/H$ gives rise to the $H$-Lie algebroid pair $(q^!B, \psi_\bullet)$ over $M$, where $H$ acts on $q^!B$ by
\[
h\cdot (X, \, b) = (\dif h(X), \, b), \hs \mbox{for} \;x \in M, \,  (X, b) \in (q^!B)|_x,
\]
and the Lie algebroid morphism $\psi_\bullet: h \ltimes M \rightarrow q^!B$ is given by
$\psi_\bullet (u, x) = (u_\sM|_x, 0)$.
\end{lem}

We now show that,  through
a {\it pushforward} operation, every $H$-Lie algebroid pair $(A, \psi)$
over $M$ is isomorphic to $(q^!B, \psi_\bullet)$ for some Lie algebroid $B$ over $M/H$.

Let $(A, \psi)$ be an $H$-Lie algebroid pair over $M$,  and let $\rho_\sA: A \to TM$ be the anchor of $A$.
Since
$\psi$ is fiber-wise injective, ${\rm Im}(\psi)$ is an $H$-invariant vector sub-bundle of $A$ of rank equal to $\dim \h$.
Consider the
 quotient vector bundle $A/{\rm Im}(\psi)$ with the induced $H$-action, and
define the vector bundle $\overline{A}^\psi$ over $M/H$ by
\begin{equation}\label{eq:OA}
\overline{A}^\psi = (A/{\rm Im}(\psi))/H.
\end{equation}
Denote the image of $a+{\rm Im}(\psi) \in A/{\rm Im}(\psi)$ in $\overline{A}^\psi$ by $[a+{\rm Im}(\psi)]$.
In the terminology of \cite[Definition 4.4.2]{Mac}, the injective Lie algebroid morphism $\psi: \h \ltimes M \to A$
 together with the $H$-action on $A/{\rm Im}(\psi)$
form an {\it ideal system} of $A$. Lemma \ref{lem:A-bar} below,  also proved in \cite[Theorem 3.6]{MPO} and
 \cite[Theorem 2.4.5]{Rob}, follows from
\cite[Theorem 4.4.3]{Mac}.

\begin{lem}\label{lem:A-bar}
For any $H$-Lie algebroid pair $(A, \psi)$, $\overline{A}^\psi$ has the structure of a Lie algebroid over $M/H$, uniquely determined by the
property that
\[
A \rightarrow \overline{A}^\psi,\; a \mapsto [a + {\rm Im}(\psi)]
\]
  is a Lie algebroid morphism covering  $q: M \to M/H$. The anchor of $\overline{A}^\psi$ is given by
\begin{equation}\label{eq:rho-Abar}
\rho_{\overline{\sA}^\psi}: \; \overline{A}^\psi \rightarrow T(M/H), \; [a+{\rm Im}(\psi)] \mapsto \dif q (\rho_\sA(a)).
\end{equation}
\end{lem}

For an $H$-Lie algebroid pair $(A, \psi)$ over $M$, we call  $\overline{A}^\psi$
the {\it pushforward} of  $(A, \psi)$ by $q$. (In \cite[$\S$2.4.3]{Rob}, $\overline{A}^\psi$ is called the {\it reduction} of $A$.)

We can now state the precise relations between the pullback and pushforward operations on Lie algebroids by $q: M \to M/H$.
Part (1) of the following proposition has also been proved in \cite[Proposition 2.4.6]{Rob}.
%We now show that the pullback and the pushforward operations on Lie algebroids via $q$ .

\begin{prop}\label{prop:bijection}
(1) For any $H$-Lie algebroid pair $(A, \psi)$, the map
\[
I: \; A \rightarrow q^!(\overline{A}^\psi), \; a \mapsto (\rho_\sA(a), \, [a + {\rm Im}(\psi)]),
\]
gives an isomorphism of $H$-Lie algebroid pairs from $(A, \psi)$ to $(q^!(\overline{A}^\psi), \psi_\bullet)$;

(2) For any Lie algebroid $B$ over $M/H$, the map
\begin{equation}\label{eq:BB}
\overline{\left(q^!B\right)}^{\psi_\bullet} \rightarrow B, \; [(X, b) + {\rm Im}(\psi_\bullet)] \mapsto b,
\end{equation}
is an isomorphism of Lie algebroids over $M/H$ covering the identity map of $M/H$.
\end{prop}

\begin{proof}
As $A \to \overline{A}^\psi, a \mapsto [a + {\rm Im}(\psi)]$, is a Lie algebroid morphism, it follows from Lemma \ref{lem:AA} that
$I: A \rightarrow q^!(\overline{A}^\psi)$ is a Lie algebroid morphism.  Note that as vector bundles over $M$, $A$ and
$q^! \overline{A}^\psi$ have the same rank, and $I$ is fiber-wise injective by \eqref{eq:rho-rho-psi}. Thus $I$ is an isomorphism of Lie algebroids.  It also follows from the definitions that
\[
I \circ \psi = \psi_\bullet:\; \h \ltimes M \rightarrow q^!\overline{A}^\psi.
\]
 Thus $I: (A, \psi) \to (q^!\overline{A}^\psi)$ is an isomorphism of $H$-Lie algebroid pairs over $M$.

In the terminology of \cite[Corollary 4.4.4]{Mac}, the {\it kernel system} of the fiber-wise surjective Lie algebroid morphism
$\chi: q^!B \to B$ in \eqref{eq:chi} is precisely the ideal system of $q^!B$ used in forming the quotient $\overline{\left(q^!B\right)}^{\psi_\bullet}$.
By \cite[Corollary 4.4.4]{Mac}, the  map in \eqref{eq:BB} is a well-defined isomorphism of Lie algebroids. This proves (2).
\end{proof}

%Recall that the action of $H$ on $M$ gives rise to the action Lie groupoid $H \ltimes M$ over $M$ with source and target respectively given by
%$(h, x) \mapsto x$ and $(h, x) \mapsto h\cdot x$.

Proposition \ref{prop:bijection} can be reformulated as follows.

\begin{them}\label{thm:PP-algebroids}
The assignment $(A, \psi) \mapsto \overline{A}^\psi$
 gives a one-to-one correspondence
from isomorphism classes of $H$-Lie algebroid pairs over $M$ to isomorphism classes of
Lie algebroids over $M/H$ covering the identity map of $M/H$, with the
inverse correspondence defined via
$B \mapsto (q^!B, \psi_\bullet).$
\end{them}

\begin{proof} One first checks from the definitions that the assignments
$(A, \psi) \mapsto \overline{A}^\psi$ and $B \mapsto (q^!B, \psi_\bullet)$
 indeed give well-defined maps between the corresponding sets of isomorphism classes, which,
by Proposition \ref{prop:bijection}, are inverses of each other.
\end{proof}

%\begin{rem}\label{re:PPH}
%{\rm
%Theorem \ref{thm:PP-algebroids} should be regarded as the analog for Lie algebroids of the following statement on vector bundles:
%the assignment $V \to V/H$ gives a one-to-one correspondence from isomorphism classes of $H$-equivariant vector bundles over $M$
%to isomorphism classes of vector bundles over $M/H$, with the inverse correspondence given by pullbacks of vector bundles over $M/H$ by $q$.
%\hfill $\diamond$
%}
%\end{rem}

\subsection{$H$-Lie groupoid pairs}
We now introduce the global counterparts of $H$-Lie algebroid pairs. In this subsection, we will consider a Lie group $H$ acting on a manifold $M$, with no further assumption.
%As in the case of Lie algebroids, the pullback Lie groupoid $q^!\K$ has additional $H$-equivariant structures which we now formulate.
Let $H \ltimes M$ be the action groupoid over $M$ with source $(h, x) \mapsto x$ and target $(h, x) \mapsto h\cdot x$.
%By an $H$-Lie groupoid on $M$ we mean a Lie groupoid $\G$ over $M$ with an $H$-action by Lie groupoid auto

\begin{defn}\label{defn:pairs-G}
{\em
(1) By an {\it $H$-Lie groupoid pair over $M$} we mean a pair $(\G, \Psi)$, where $\G$ is a Lie groupoid over $M$, and $\Psi$ is a
Lie groupoid morphism from $H \ltimes M$ to $\G$ covering the identity map of $M$.

(2) Two $H$-Lie groupoid pairs $(\G, \Psi)$ and $(\G', \Psi')$ are said to be isomorphic if there exists a Lie groupoid isomorphism
$I: \G \to \G'$ covering the identity map of $M$ such that $I \circ \Psi = \Psi': H \times M \to \G'$.
}
\end{defn}

With suitable conditions on the Lie groupoid morphism $\Psi$, an $H$-Lie groupoid pair $(\G, \Psi)$ is called an
{\it $H$-principal Lie groupoid over $M$} in \cite{FS19} (see Definition 2.2 and Proposition 2.4 therein).

Given an $H$-Lie groupoid pair $(\G, \Psi)$ over $M$, each $h \in H$ defines a bisection
\[
b_h: \; M \to \G, \; x \mapsto \Psi(h, x),
\]
 of $\G$ covering the action of $h$ on $M$, i.e.,
$\s (b_h(x))=x$ and $\t(b_h(x)) = h\cdot x$ for $x \in M$. One then has an $(H \times H)$-action on
$\G $ via
\begin{equation}\label{eq:double-action}
(h_1,h_2)\cdot g = \Psi(h_1,\t (g))g \Psi(h_2,\s(g))^{-1}.
\end{equation}
%Note that the induced action of $H \cong H \times \{e\} \subset H \times H$ on $\G$ gives a principal action of $H$ on each source fiber.
%The Lie groupoid $\G$,  equipped with the action of $H \cong H \times \{e\}$, is called an
%{\it $H$-principal Lie groupoid over $M$} in \cite[Definition 2.2]{FS19}.

%\hen{The last paragraph seems too long, just explaining parallel terminology... I would shorten it, or even skip it?
%perhaps a short remark like ``The Lie groupoid $\G$ together with the induced action of $H \cong H \times \{e\}$ is
%an {\it $H$-principal Lie groupoid over $M$} as in \cite[Definition 2.2]{FS19}.''
%Consider the restriction of the $(H \times H)$-action on $\G$ to $H_{\rm diag} =\{(h, h): h \in H\}\subset H \times H$
%(referred to as an {\it inner action} of $H$ in \cite[$\S$A.2]{BF:coupling}).}

\begin{lem}\label{lem:G-A} \cite[Proposition 3.2]{FS19}
Let $(\G, \Psi)$ be an $H$-Lie groupoid pair over $M$, let $A = {\rm Lie}(\G)$, and let
$\psi: \h \ltimes M \to A$ be the Lie algebroid morphism induced by $\Psi$. Then $(A, \psi)$ is an $H$-Lie algebroid pair over $M$,
where $H$ acts on $A$  by differentiating the $H \cong H_{\rm diag}$-action on $\G$.
\end{lem}

\begin{defn}\label{defn:Lie-pairs}
{\rm In the setting of Lemma \ref{lem:G-A}, we call $(A, \psi)$ {\it the $H$-Lie algebroid pair of $(\G, \Psi)$} and write $(A, \psi) = {\rm Lie}(\G, \Psi)$.
%We also call $(\G, \Psi)$ an integration of $(A, \psi)$.
}
\end{defn}

\subsection{Pullbacks and pushforwards of Lie groupoids}\label{ss:PP-groupoids} 
Assume again that $q: M \to M/H$ is a principal bundle.
Recall
\cite[$\S$1]{Higgins-Mackenzie} that
for a Lie groupoid $\overline{\s}, \overline{\t}: \K \rightrightarrows M/H$, one has the {\it pullback Lie groupoid}
$q^!\K$ over $M$, where
\[
q^!\K = \{ (x_1,k,x_2) \in M\times \K \times M \;|\; q(x_1) = \bar{\t}(k), \; \bar{\s}(k) =q(x_2)\},
\]
with source and target $\s(x_1, k, x_2) = x_2$ and $\t(x_1, k, x_2) = x_1$, and multiplication
$(x_1, k, x_2) (x_2, k', x_2^\prime) = (x_1, k \, k', x_2^\prime)$.
Note (see \cite[Corollary 1.9]{Higgins-Mackenzie}) that
\begin{equation}\label{eq:Liepb}
{\rm Lie}(q^!\K) \cong q^! {\rm Lie}(\K).
\end{equation}

The following lemma follows directly from the definitions.

\begin{lem}\label{lem:pull-pair}
Every Lie groupoid $\K$ over  $M/H$ gives rise to the $H$-Lie groupoid pair $(q^! \K, \Psi_\bullet)$ over $M$, where
$\Psi_\bullet:  H \ltimes M \rightarrow q^! \K,  (h, x) \mapsto (h\cdot x, \, 1_{q(x)}, \, x)$.
 Furthermore, with $B = {\rm Lie}(\K)$ and  $\psi_\bullet: \h \ltimes M \to q^!B$ as given in Lemma \ref{lem:qB-pair}, one has
\[
{\rm Lie}(q^! \K, \Psi_\bullet) \cong (q^!B, \psi_\bullet).
\]
\end{lem}

%Same as the case  of Lie algebroids, we now want to show that every $H$-Lie groupoid pair $(\G, \Psi)$ over $M$ is isomorphic to
%$(q^!\K, \Psi_\bullet))$ for some Lie groupoid $\K$ over $M/H$. The construction of $\K$ from $(\G, \Psi)$ is again
 The pushforward operation on $H$-Lie groupoid pairs, as we explain in Proposition \ref{prop:quot} below,
is an example of quotients of Lie groupoids \cite[$\S$2.4]{Mac}.
Part (1) of Proposition \ref{prop:quot} is also proved in \cite[Proposition 7.8]{FS19},
%\hen{also daniel, more generally?},
but  we include an outline of the proof for the convenience of the reader.  See also \cite{Alv} for a more general setting.

\begin{prop}\label{prop:quot}
Let $(\G, \Psi)$ be an $H$-Lie groupoid pair over $M$,  let $(A, \psi) = {\rm Lie}(\G, \Psi)$, and let  $H \times H$ act on $\G$ as in
\eqref{eq:double-action}.

(1) The orbit space $\overline{\G}^\Psi:=\G/(H\times H)$
 is a Lie groupoid over $M/H$ characterized by the property that the quotient map
$p:\G \to \overline{\G}^\Psi$ is a submersion and a groupoid morphism;

(2) One has ${\rm Lie}(\overline{\G}^\Psi) \cong \overline{A}^\psi$, where $\overline{A}^\psi$ is given in Lemma \ref{lem:A-bar}.
\end{prop}
%(2) One has the isomorphism of $H$-Lie groupoids
%\[
%\tilde{p}: \; (\G, \Psi) \to (q^! \overline{\G}^\Psi, \Psi_\bullet), \; \tilde{p}(g) = (\t(g), \, p(g), \, \s(g)).
%\]

\begin{proof}
The $(H \times H)$-action on $\G$ in \eqref{eq:double-action} is such that
\begin{equation}\label{eq:sourtar}
\t((h_1,h_2)\cdot g) = h_1\cdot \t(g),\;\;\; \s((h_1,h_2)\cdot g) =
h_2\cdot \s(g),\;\;\;\; h_1, h_2 \in H, \, g \in \G,
\end{equation}
which implies that  $(\t,\s): \G \to M\times M$ is $(H\times H)$-equivariant. Since
the $(H\times H)$-action on $M\times M$ is principal, so is the $(H\times H)$-action on $\G$ by Lemma~\ref{lem:principal}.
Hence the orbit space $\overline{\G}^\Psi=\G/(H\times H)$ is smooth,
the quotient projection $p: \G \to \overline{\G}^\Psi$ is a submersion\footnote{Note that Lemma~\ref{lem:principal}
can still be used if $\G$ is not necessarily
Hausdorff, see Remark~\ref{rem:nonhausdorff}.}
%If $\G$ is Hausdorff, the action of $H \times H$ on $\G$ is free and proper by the $(H \times H)$-equivariance of
% $(\t, \s)$, so the quotient $\G/(H \times H)$ is a
%smooth manifold and $p: \G \to \overline{\G}^\Psi$ is a submersion.},
and there are unique submersions $\overline{\s},\,\overline{\t}:  \overline{\G}^\Psi \to M/H$ such that
\begin{equation}\label{eq:diag-t}
q \circ \t = \overline{\t} \circ p \hs \mbox{and} \hs q \circ \s = \overline{\s} \circ p .
\end{equation}
Write ${p}(g)=\overline{g}$ for $g \in \G$.
To define a groupoid multiplication on $\overline{\G}$ with source $\overline{\s}$ and target $\overline{\t}$, suppose that
$g_1, g_2 \in \G$ are such that
$\overline{g}_1$, $\overline{g}_2 \in \overline{\G}^\Psi$ are composable. Then $q(\s(g_1))= q(\t(g_2))$, so there exists $h \in H$ such that
$\s(g_1) = h\cdot\t(g_2) = \t((h, e)\cdot g_2)$. Replacing $g_2$ by $(h, e)\cdot g_2$, we may assume that
 $g_1$, $g_2 \in \G$ are composable, and define $\bar{g}_1 \bar{g_2} = \overline{g_1g_2}$.
If $g_1' =
(h_1,h_2)\cdot g_1$ and $g_2' = (l_1,l_2)\cdot g_2$  are also composable, then by  \eqref{eq:sourtar},
\[
h_2 \cdot \s(g_1) = l_1 \cdot \t(g_2) = l_1 \cdot \s(g_1),
\]
and thus $h_2 = l_1$ by the freeness of the $H$-action on $M$.
Then (see Remark~\ref{rem:dlg} below)
\begin{equation}\label{eq:dlg}
((h_1,l)\cdot g_1) ((l,l_2)\cdot g_2)=(h_1,l_2)\cdot (g_1 g_2),
\end{equation}
hence $\overline{g_1  g_2} = \overline{g_1'  g_2'}$.  Thus $\overline{\G}^\Psi$ is a well-defined Lie groupoid over $M/H$, proving $(1)$.
%To verify (2), identify $M$ with its image in $\G$ and $M/H$ with its image in $\overline{\G}^\Psi$ via the identity sections, and let
%$B ={\rm Lie}(\overline{\G}^\Psi)$.  The
 %Lie algebroid morphism $A \to B$ defined by $p$ is then
%given by
%\[
%\dif p|_{A}: \; A = {\rm Ker}(\dif \s)|_{\sM} \rightarrow {\rm Ker}(\dif \bar{\s})|_{\sM/\sH} = B,
%\]
%where recall that $p: \G \to \overline{\G}^\Psi$, being a Lie groupoid morphisms, maps source fibers to source fibers.
%Let $x \in M$. By the definition of the $(H \times H)$-action on $\G$ in \eqref{eq:doublem:action},
%\[
%{\rm Ker} (\dif p)|_{x} = \{\psi(u, x) + {\rm inv} (\psi(v, x)): u, v \in \h\}.
%\]
%Moreover, for $u, v \in \h$, one has
%\[
%\dif \s (\psi(u, x) + {\rm inv} (\psi(v, x))) =  \dif \s ({\rm inv} (\psi(v, x))) = \dif \t (\psi(v, x)) = \rho_\sA(\psi(v, x)) = u_\sM(v)|_x,
%\]
%so $\psi(u, x) + {\rm inv} (\psi(v, x)) \in A|_x$ if and only if $v = 0$.
%Thus
%$A|_x \cap {\rm Ker} (\dif p)|_x= {\rm Im}(\psi)|_x$.
 %It follows that one has the injective vector space morphism
%\[
%A_x/{\rm Im}(\psi)_x \to B_{q(x)}, \; a + {\rm Im}(\psi)|_x\mapsto \dif p (a),
%\]
%which is an isomorphism because the two vector spaces have the same dimension.
%the same rank, the map $A/{\rm Im}(\psi) \to B$ is an isomorphism on each fiber.

%algebroid of $\overline{\G}$ is the image of the sub-bundle $A\to M$ of $T\G\to \G$ under $\dif p: T\G\to T\overline{\G}$.
%We will show that it coincides with $\overline{A}$.

To verify (2), recall that $A = {\rm Ker}(\dif \s)|_\sM \subset T\G|_{\sM}$. Let
$B = {\rm Lie}(\overline{\G}^\Psi) = {\rm Ker}(\dif \bar{\s})|_{\sM/\sH}$.
Then $p: \G \to \overline{\G}^\Psi$, being a Lie groupoid morphism, maps source fibers to source fibers, and
$\dif p|_{A}: A \to B$ is the Lie algebroid morphism defined by $p$.
%Consider the tangent action of $H\times H$ on $T\G$, still denoted by $Z\mapsto (h_1,h_2)\cdot Z$, and the
%corresponding infinitesimal action of  $\h\times \h$ on $\G$  given by
%\begin{equation}\label{eq:infact}
%\rho_{\scriptscriptstyle{\G}} \colon (\h\times \h)\times \G \to T\G,\;\;\;\;
%\rho_{\scriptscriptstyle{\G}} (u,v)|_g=(\psi (u,\t(g)))^\rt+(\mathrm{inv}(\psi (v,\s(g))))^\lt.
%\end{equation}
%The equivalence relation $Z\sim W$ on $T\G$ defined by $\dif p(Z)=\dif p (W)$ is then given by
%$$
%Z \sim (h_1,h_2)\cdot Z + {\rm Im}(\rho_{\scriptscriptstyle{\G}}).
%$$
%To describe when two elements in $A$ are equivalent, note that for $x\in M\subseteq \G$,
%\[
%(h_1,h_2)\cdot x  = \Psi(h_1h_2^{-1}, \,h_2 \cdot x),
%\]
%so $(h_1,h_2)\cdot x \in M$  if and only if $h_1=h_2$, and in this case, $(h, h)\cdot x = h \cdot x$.
%On the other hand, for $x \in M$ and $u, v \in \h$, since
%$\rho_{\scriptscriptstyle{\G}}(u,v)|_x = \psi(u, x) + \mathrm{inv}(\psi(v, x))$, one has
%$$
%\dif \s(\rho_{\scriptscriptstyle{\G}}(u,v)|_x) = \dif \s (\mathrm{inv}(\psi(v, x))) = \dif \t (\psi(v, x))=\rho (\psi(v, x))=\rho_\sM(v, x).
%$$
%Thus $\rho_{\scriptscriptstyle{\G}}(u,v)|_x \in A$ if and only if $v=0$. Hence $A \cap {\rm Im}(\rho_{\scriptscriptstyle{\G}}) = {\rm Im}(\psi)$.
One checks directly using the definitions that  for $a, b \in A$, $(\dif p)(a) = (\dif p)(b)$ if and only if
$a =  h\cdot b + \mathrm{Im}(\psi)$,
where $b\mapsto h\cdot b$ denotes the $H$-action on $A$ arising from the $H_{\rm diag}$-action on $\G$.
Thus $\dif p|_\sA: A \to B$ induces a well-defined fiber-wise injective Lie algebroid morphism $\tilde{p}: \overline{A}^\psi \to B$
covering the identity map on $M/H$. Since $\overline{A}^\psi$ and $B$ have the same rank, $\tilde{p}$ is a Lie algebroid isomorphism.
\end{proof}

\begin{rem}\label{rem:dlg}
{\em Note that \eqref{eq:dlg} says that the
action of $H\times H$ on $\G$ is ``morphic'' in the sense that the action
map $(H\times H) \times \G \to \G$ is a groupoid morphism
over the action map $H \times M \to M$, where
$(H\times H) \times \G$, as a Lie groupoid over $H \times M$, is the direct product of the pair groupoid $H\times H \rightrightarrows H$ with $\G$.
In other words, $\G$ carries an action of the double Lie groupoid $H\times H \rightrightarrows H$ in the sense of \cite[Definition 1.5]{BrownMackenzie}, which explains the groupoid structure on the quotient $\G/(H\times H)$ (see \cite[$\S$3.3]{StefThesis}).
\hfill $\diamond$}
\end{rem}

\begin{numex}\label{ex:HLM}
{\rm
Suppose that $L$ is a Lie group containing $H$ as a closed Lie sub-group and that
$L$ acts on $M$ extending the action of $H$ on $M$. One then has the
$H$-Lie groupoid pair
$(\G, \Psi)$ over $M$, where $\G = L \ltimes M$ is the action groupoid, and $\Psi: H \ltimes M \to \G$
is defined by $H \hookrightarrow L$.  The $(H \times H)$-action on $\G$ induced by $\Psi$ is
\[
(h_1, h_2)\cdot (l, x) = (h_1lh_2^{-1}, \, h_2\cdot x).
\]
Let $\l$ be the Lie algebra of $L$, and let $(\l/\h) \times_H M$ be the quotient of the $H$-action on $\l \times M$ by
$h \cdot (a, x) = ({\rm Ad}_h a, \, h\cdot x)$. One then has the $H$-Lie algebroid pair $(\l \ltimes M, \psi)$ over $M$,
where $\psi: \h \ltimes M \to \l \ltimes M$ is defined by the Lie algebra inclusion $\h \hookrightarrow \l$.
Consequently, $(\l/\h) \times_H M \cong \overline{(\l \ltimes M)}^\psi$ has a unique Lie algebroid structure over $M/H$
such that $p: \l \ltimes M \to (\l/\h) \times_H M$ is a Lie algebroid morphism covering $q: M \to M/H$.
By Proposition \ref{prop:quot}, the quotient manifold $(L\ltimes M)/(H \times H) = (L/H) \times_H M$ is a Lie groupoid over $M/H$ with Lie algebroid
$(\l/\h) \times_H M$.
This example appeared in \cite[$\S$8]{BlWe}.
}
\end{numex}

We now prove the following analog of Proposition \ref{prop:bijection} for Lie groupoids.

\begin{prop}\label{prop:bijection-groupoids}
(1) For any $H$-Lie groupoid pair $(\G, \Psi)$, the map
\[
I: \; \G \rightarrow q^!(\overline{\G}^\Psi), \; g \mapsto (\t(g), \, p(g), \, \s(g)),
\]
is an isomorphism of $H$-Lie groupoid pairs from $(\G, \Psi)$ to $(q^!(\overline{\G}^\Psi), \Psi_\bullet)$;

(2) For any Lie groupoid $\K$ over $M/H$, the map
\begin{equation}\label{eq:KG}
\overline{\left(q^!\K\right)}^{\Psi_\bullet} \rightarrow \K, \; [x_1, k, x_2] \mapsto k,
\end{equation}
is an isomorphism of Lie groupoids over $M/H$ covering the identity map of $M/H$.
\end{prop}

\begin{proof}
(1) One checks directly that $I: \G \to q^!(\overline{\G}^\Psi)$  is a Lie groupoid morphism covering the identity map of $M$. It also follows from the definitions that $I \circ \Psi = \Psi_\bullet$. It remains to show that
$I$ is a Lie groupoid isomorphism. It is straightforward to check that $I$ is bijective. Since  $\G$ and $q^!\overline{\G}^\Psi$
have the same dimension, one only needs to show that for every $g \in \G$,
the linear map
$\dif I|_g: T_g\G \to T_{I(g)} (q^!\overline{\G}^\Psi)$ is injective.

Suppose that $v_g \in T_g\G$ is such that $\dif I(v_g) = 0$. Then $\dif \t(v_g) = 0$,
$\dif \s(v_g) = 0$, and $\dif p(v_g) = 0$. It follows from $\dif p (v_g) = 0$ that
 $v_g = (\psi (u,\t(g)))^\rt+(\mathrm{inv}(\psi (v,\s(g))))^\lt$
for some $u, v \in \h$.
 The fact that $\dif \s(v_g) = 0$ then implies that
\[
0 = \dif \s (\mathrm{inv}(\psi(v, \s(g)))) = \dif \t (\psi(v, \s(g))) = \rho_\sA(\psi(v, \s(g)))=
\rho_\sM(v, \s(g)),
\]
so $v = 0$. The fact that $\dif \t(v_g) = 0$ implies that
\[
0 = \dif \t (\psi(u, \t(g))) = \rho_\sA(\psi(u, \t(g))) = \rho_\sM (u, \t(g)),
\]
so $u = 0$. Thus $v_g = 0$, and we conclude that $I: (\G, \Psi) \to (q^! \overline{\G}^\Psi, \Psi_\bullet)$ is an isomorphism of $H$-Lie
groupoid pairs over $M$.

(2) Given now a Lie groupoid $\bar{\s}, \bar{\t}: \K \rightrightarrows M/H$, consider the $H$-Lie groupoid pair
$(q^! \K, \Psi_\bullet)$.
The $(H \times H)$-action on $q^!\K$ defined by $\Psi_\bullet$ is then given by
\[
(h_1, h_2) \cdot (x_1, k, x_2) = (h_1\cdot x_1, \, 1_{q(x_1)}, \, x_1) (x_1, k, x_2) (x_2, \,1_{q(x_2)}, \, h_2\cdot x_2) =
(h_1\cdot x_1, k, h_2 \cdot x_2).
\]
Let
${\rm Graph}_{(\bar{\t}, \bar{\s})}(\K) = \{(\bar{\t}(k), k, \bar{\s}(k)): k \in \K\}\subset (M/H) \times
\K \times (M/H)$.
 Then
\[
(q^!\K)/(H \times H) \rightarrow {\rm Graph}_{(\bar{\t}, \bar{\s})}(\K), \;
[x_1, k, x_2] \mapsto (\bar{\t}(k), k, \bar{\s}(k)),
\]
is a diffeomorphism, leading to
the Lie groupoid isomorphism
\[
\overline{\left(q^!\K\right)}^{\Psi_\bullet} =  (q^!\K)/(H \times H) \rightarrow \K, \;
[x_1, k, x_2] \mapsto k.
\]
\end{proof}

\begin{them}\label{thm:PP-groupoids}
The assignment $(\G, \Psi) \mapsto \overline{\G}^\Psi = \G/(H \times H)$
 induces a one-to-one correspondence from isomorphism classes of $H$-Lie
groupoid pairs over $M$ to that of Lie groupoids over $M/H$, with the inverse correspondence defined by the assignment
$\K \mapsto (q^!\K,
\Psi_\bullet).$
\end{them}

\begin{proof}
Checking directly using the definitions, one sees that both assignments induce well-defined
maps between the corresponding sets of isomorphism classes.  The two maps are inverses of each other by Proposition \ref{prop:bijection-groupoids}.
\end{proof}

\subsection{$H$-Integrability}\label{ss:H-admi}
For a principal bundle $q: M \to M/H$, \eqref{eq:Liepb} shows that if a Lie algebroid $B$ over $M/H$ is integrable, then its pullback
Lie algebroid $q^!B$ over $M$ is also integrable. The next example shows that the converse is not true.

\begin{numex}\label{ex:integ}
 {\rm
Let $M =S^3\times \mathbb{R}$ and consider the $S^1$-action on $S^3$ defining the Hopf fibration $S^3\to S^2$.
Extend this action to $M$ by the trivial action on $\mathbb{R}$, so that
 $$
 q: M\to M/S^1=S^2\times \mathbb{R}
 $$
 is the Hopf map times the identity map of $\mathbb{R}$. Let $\pi$ be the Poisson structure on $S^2\times\mathbb{R}$
with symplectic leaves $(S^2\times \{ t\}, (1+t^2) \omega _{S^2})$, with $\omega _{S^2}$ the standard area form on $S^2$; it is well-known that
$\pi$ is not integrable \cite{We87} (see also \cite{CaFe,CF2}), but the pullback of the co-tangent bundle Lie algebroid $T_\pi^* (S^2 \times {\mathbb{R}})$
to $M$ is integrable: it is a regular
Lie algebroid on $S^3\times \mathbb{R}$ with leaves $S^3\times \{ t\}$,
 and its integrability follows from the fact that $\pi_2(S^3)$ vanishes \cite{CF1} (see also \cite[$\S$8]{BCWZ}).
\hfill $\diamond$
}
\end{numex}

%For a principal bundle $q: M \to M/H$ and a Lie algebroid $B$ over $M/H$,
One thus needs a stronger integrability
of $q^!B$ to guarantee integrability of $B$.
The following definitions are adapted from \cite{FS19}.

\begin{defn}\label{defn:A-H-integrable}
{\rm  (1) An $H$-Lie algebroid pair $(A, \psi)$ over $M$  is said to {\it integrable} if there exists an $H$-Lie groupoid pair
$(\G, \Psi)$ such that ${\rm Lie}(\G, \Psi) \cong (A, \psi)$, in which case $(\G, \Psi)$ is called an {\it integration} of $(A, \psi)$;

(2) For a Lie algebroid $B$ over $M/H$, the pullback Lie algebroid $q^!B$ over $M$ is said to be
{\it $H$-integrable} if the $H$-Lie algebroid pair $(q^!B, \psi_\bullet)$ is
integrable,   and in such a case an integration
of $(q^!B, \psi_\bullet)$  is also called an {\it $H$-integration} of $q^!B$.
}
\end{defn}

%\begin{defn}\label{defn:A-H-integrable}
%{\rm  For a Lie algebroid $B$ over $M/H$, the pullback Lie algebroid $q^!B$ over $M$ is said to be
%{\it $H$-integrable} if the $H$-Lie algebroid pair $(q^!B, \psi_\bullet)$ is
%integrable in the sense that there exists an $H$-Lie groupoid pair $(\G, \Psi)$ such that $(q^!B, \psi_\bullet) \cong {\rm Lie}(\G, \Psi)$, and in such a case,
%G, \Psi)$  is also called an {\it $H$-integration} of $q^!B$.
%}
%\end{defn}

\begin{them}\label{thm:inte-MMH}
A Lie algebroid $B$ over $M/H$ is integrable if and only if the pullback Lie algebroid $q^!B$ over $M$ is $H$-integrable. In such a case,
if $(\G, \Psi)$ is any $H$-integration of $q^!B$, then $\overline{\G}^\Psi $ is an integration of $B$.
\end{them}

\begin{proof}
The statements follow directly from Theorem \ref{thm:PP-algebroids} and Theorem \ref{thm:PP-groupoids}.
\end{proof}

\begin{rem}\label{re:H}
{\rm
(1) When $H$ is simply connected, integrability of an $H$-Lie algebroid pair $(A, \psi)$ over $M$
 is automatic as long as $A$ is integrable \cite[Proposition A.17]{BF:coupling}.
 Thus when $H$ is simply connected,
 a Lie algebroid $B$ over $M/H$ is integrable if and only if the Lie algebroid $q^!B$ over $M$ is integrable.

(2) For an arbitrary $H$, obstructions to integrability of an $H$-Lie algebroid pair $(A, \psi)$ are discussed in \cite[$\S$4]{BF} and \cite[$\S$5]{FS19}.
Briefly,
assume that $A$ is integrable, and let $\Sigma(A)$ be the source-simply-connected Lie groupoid integrating $A$. Assume also that
$H$ is connected and let $\widetilde{H}$ be the simply connected cover of $H$. Let
 $p: \widetilde{H} \to H$
 be the covering map and identify
$\pi_1(H) \cong {\rm Ker}(p) \subset \widetilde{H}$.
Let $\widetilde{H}$ act on $M$ through
$p$ and let  $\widetilde{H} \ltimes M$ be the action Lie groupoid over $M$. The Lie algebroid morphism $\psi: \h \ltimes M \to A$ then lifts to a
unique Lie groupoid morphism
$\widetilde{H} \ltimes M \rightarrow \Sigma(A)$.
It is shown in \cite[Proposition 4.4]{BF} (see also the proof of \cite[Theorem 5.3]{FS19}) that the $H$-Lie algebroid pair $(A, \psi)$ is integrable if and only if
the image of $\pi_1(H) \times M \subset \widetilde{H} \times M$ is an embedded Lie sub-groupoid of $\Sigma(A)$. Indeed, the quotient of $\Sigma(A)$ by this image is a Lie groupoid admitting a morphism of $H\ltimes M$ into it, hence defining an $H$-Lie groupoid pair integrating  $(A, \psi)$ (called its ``canonical $H$-integration'' in \cite[$\S$ 5.1]{FS19}).
\hfill $\diamond$
}
\end{rem}

\begin{rem}\label{re:Daniel-1}
{\rm
Theorem \ref{thm:inte-MMH} has been generalized in \cite{Alv} to arbitrary surjective submersions $q: M \to M'$.
\hfill $\diamond$
}
\end{rem}

\section{Lie algebroids associated to equivariant Lie algebra triples}\label{s:equiv-triple}
Let $G$ be an arbitrary Lie group. In $\S$\ref{ss:dressing} - $\S$\ref{ss:BHl}, we use results in $\S$\ref{s:PP} to
give a class of integrable Lie algebroids
associated to {\it $G$-equivariant Lie algebra triples} \cite[Definition 3.2]{L-B-M:Dirac}.
In $\S$\ref{ss:LA-vee} we interpret the results in $\S$\ref{ss:dressing} - $\S$\ref{ss:BHl}
in terms of homogeneous spaces of ${\mathcal{LA}}^\vee$-Lie groups \cite[$\S$4]{Mein}. We will show in $\S$\ref{s:homog-G} how
Poisson homogeneous spaces fit into the framework of this section.

%In this section, we present some examples of explicit integrations of certain Lie algebroids over $G/H$, where $G$ is a Lie group and $H \subset G$
%a closed Lie sub-group.
%Special cases considered here will give explicit symplectic groupoids of Poisson homogeneous spaces  in DDD.
\subsection{Equivariant Lie algebra triples and the dressing action}\label{ss:dressing}
Let $G$ be a Lie group with Lie algebra $\g$, and let $\d$ be a Lie algebra containing $\g$ as a Lie sub-algebra.
Recall that $(G, \d)$ is called a {\it Harish-Chandra pair} if the adjoint action of $\g$ on $\d$ integrates to an action of $G$ on $\d$, which we will denote
as ${\rm Ad}_g \in {\rm Aut}(\d)$ for $g \in G$.

\begin{defn}\label{defn:G-triples} \cite[Definition 3.2]{L-B-M:Dirac}
{\rm Let $G$ be a Lie group with Lie algebra $\g$. A {\it $G$-equivariant Lie algebra triple} is a triple $(\d, \g, \k)$, where $\d$ is a
Lie algebra, containing $\g$ and $\k$ as Lie sub-algebras, such that $(G, \d)$  is a Harish-Chandra pair and
$\d = \g + \k$ is a vector space direct sum.
}
\end{defn}

%By \cite[$\S$4]{Mein}, $G$-equivariant Lie algebra triples classify the so-called ${\mathcal{LA}}^\vee$ Lie group structures on $G$, which are
%Lie algebroids $E$ over $G$ with the structure of an ${\mathcal{LA}}$-groupoid over $G$ (as a Lie groupoid over the one point space ${\rm pt}$),
%and Harish-Chandra sub-pairs classify homogeneous spaces the ${\mathcal{LA}}^\vee$ Lie group $(G, E)$.
%In this sub
In this section, we fix a Lie group $G$ and a $G$-equivariant triple
$(\d, \g, \k)$. Let
\[
{\rm p}_\g\colon \d \to \g \hs \mbox{and} \hs {\rm p}_\k\colon \d \to \k
\]
be the projections with respect to the decomposition
$\d = \g + \k$.  Define
\begin{equation}\label{eq:aG1}
\rho_\d: \, \d \rightarrow {\mathfrak{X}}^1(G), \; \rho_\d(a)|_g = -\rt_g ({{\rm p}}_\g ({\rm Ad}_g a)),
\hs a \in \d, \, g \in G.
\end{equation}

\begin{rem}\label{re:dressing}
{\rm
For the motivation of the vector fields $\rho_\d(a)$,
suppose that $(\d, \g, \k)$ is {\it complete}, i.e.,
there exist a Lie group $D$ with Lie algebra $\d$ and containing $G$ as a Lie sub-group,  and a  Lie sub-group $K$ of $D$ with Lie algebra $\k$, such that
the group multiplication of $D$ gives a diffeomorphism
$K \times G \to D$. Then $D$ acts on $G$ %\cong K\backslash D$
via
\[
d \cdot g = g' \hs \mbox{if} \hs gd^{-1} = kg^\prime,\hs d \in D, \, g, g' \in G, \, k \in K.
\]
It is straightforward to check that the induced $\d$-action on $G$ is given by \eqref{eq:aG1}.
\hfill $\diamond$
}
\end{rem}

When $(\d, \g, \k)$ is not necessarily complete, \eqref{eq:aG1} still defines a Lie algebra action, called the {\it dressing action}, of $\d$ on $G$, a fact
stated in \cite[$\S$2.1]{Mein}.
As the dressing action is the starting point of our discussions, we prove this fact for the convenience of the reader.
%\hen{in case we eventually need to save space, we could just assume the fact and omit the lemma}

\begin{lem}\label{lem:d-G}
For any $G$-equivariant Lie algebra triple $(\d, \g, \k)$,
the map $\rho_\d: \d \rightarrow {\mathfrak{X}}^1(G)$ in \eqref{eq:aG1} is a Lie algebra homomorphism, and one has
\begin{equation}\label{eq:rho-gg}
\rho_\d({a)|_{g_1g_2}} = \rt_{g_2} (\rho_\d({\rm Ad}_{g_2} a)|_{g_1}), \hs a \in \d, \, g_1, g_2 \in G.
\end{equation}
\end{lem}

\begin{proof} Identity \eqref{eq:rho-gg} follows directly from the definition. It remains to show that $\rho_\d$ is a Lie algebra homomorphism.
Consider the right invariant Maurer-Cartan form $\theta^\rt$ on $G$, and recall from the end of $\S$\ref{s:intro} that
$d\theta^\rt + [\theta^{\rt}, \theta^{\rt}] = 0$.
For $a \in \d$, set
\[
\gamma_a = -\theta^\rt(\rho_\d(a)): \, G \rightarrow \g, \; \gamma_a(g) = -\rt_{g^{-1}} (\rho_\d(a)|_g) = {\rm p}_\g ({\rm Ad}_g a), \hs g \in G.
\]
For $V \in {\mathfrak{X}}^1(G)$ and $\gamma_1, \gamma_2 \in C^\infty(G, \d)$, let $V \cdot \gamma_1$ be the Lie derivative of
$\gamma_1$ in the direction of $V$, and let $[\gamma_1, \gamma_2] \in C^\infty(G, \d)$ be defined by point-wise Lie bracket in $\d$.
For $a, b \in \d$, it then follows from $(d\theta^\rt + [\theta^{\rt}, \theta^{\rt}])(\rho_\d(a), \rho_\d(b))=0$ that
\begin{equation}\label{eq:abG}
\theta^\rt ([\rho_\d(a), \rho_\d(b)]) = -\rho_\d(a) \cdot \gamma_b  + \rho_\d(b) \cdot \gamma_a+ [\gamma_a, \, \gamma_b].
\end{equation}
To prove $[\rho_\d(a), \rho_\d(b)] = \rho_\d([a, b])$, it thus remains to prove that
\begin{equation}\label{eq:gab}
\rho_\d(a) \cdot \gamma_b  - \rho_\d(b) \cdot \gamma_a- [\gamma_a, \, \gamma_b] = \gamma_{[a, b]} \in C^\infty(G, \g), \hs a, b \in \d.
\end{equation}
For $a \in \d$, introduce also $\delta_a: G \to \d$ and $\kappa_a: G \to \k$ by
\[
\delta_a(g) = {\rm Ad}_g a, \hs \kappa_a(g) = {\rm p}_\k({\rm Ad}_g a), \hs g \in G.
\]
Let $a, b \in \d$.
 For $g \in G$ and $u \in \g$, substituting $g' = \exp(tu)$  in
${\rm Ad}_{gg'} b = {\rm Ad}_g {\rm Ad}_{g'} b$ and
differentiating at $t = 0$, noting that $\frac{d}{dt}|_{t = 0} {\rm Ad}_{\exp(tu)} b = -[u, b]$ (see \eqref{eq:exp}), one gets
\begin{equation}\label{eq:Lieu}
({u^\lt} \cdot \delta_b)(g) = -{\rm Ad}_g [u, b] = -[{\rm Ad}_g u, \, \delta_b(g)],
\end{equation}
where recall that $u^\lt$ is the left invariant vector field on $G$ with value $u$ at $e$. Setting
\[
u = \lt_{g^{-1}} (\rho_\d(a)|_g) = -{\rm Ad}_{g^{-1}} {\rm p}_\g ({\rm Ad}_g a) = -{\rm Ad}_{g^{-1}} \gamma_a(g)
\]
in \eqref{eq:Lieu}, one gets $(\rho_\d(a) \cdot \delta_b)(g)  = [\gamma_a(g), \, \delta_b(g)]$, from where it follows that
\begin{equation}\label{eq:abG-1}
\rho_\d(a) \cdot \gamma_b = [\gamma_a, \, \gamma_b] + {\rm p}_\g [\gamma_a, \kappa_b] \hs \mbox{and} \hs
\rho_\d(a) \cdot \kappa_b = {\rm p}_\k [\gamma_a, \kappa_b].
\end{equation}
Putting \eqref{eq:abG-1} and the similar formula for $\rho_\d(b) \cdot \gamma_a$ in \eqref{eq:abG}, one gets
\[
\rho_\d(a) \cdot \gamma_b  - \rho_\d(b) \cdot \gamma_a- [\gamma_a, \, \gamma_b] = [\gamma_a, \gamma_b] + {\rm p}_\g [\gamma_a, \kappa_b] +
{\rm p}_\g [\kappa_a, \gamma_b] = {\rm p}_\g [\delta_a, \delta_b] = \gamma_{[a, b]}.
\]
\end{proof}

In view of Lemma  \ref{lem:d-G}, associated to the $G$-equivariant triple $(\d, \g, \k)$ one has
 the action Lie algebroid $\d \ltimes G$ over $G$.
We now explain the alternative description of  $\d \ltimes G$ given in \cite[Proposition 8]{Mein}.
By \eqref{eq:rho-gg}, $\d \ltimes G$ is a $G$-Lie algebroid with the $G$-action
\begin{equation}\label{eq:GdG}
g \cdot (a, g') = ({\rm Ad}_g a, \, g'g^{-1}).
\end{equation}
Let $\psi_{\g, \d}: \g \ltimes G \to \d \ltimes G$ be the Lie algebroid morphism defined by the Lie algebra inclusion $\g \to \d$. Then
$(\d \ltimes G, \psi_{\g, \d})$ is a $G$-Lie algebroid pair  over $G$.
Let $q: G \to G/G = {\rm pt}$. The pushforward $\k' :=\overline{(\d \ltimes G)}^{\psi_{\g, \d}}$ (see Lemma \ref{lem:A-bar}), being
a Lie algebroid over ${\rm pt}$, is then
a Lie algebra, and by (1) of Proposition \ref{prop:bijection}, one has an isomorphism
\[
(\d \ltimes G, \psi_{\g, \d}) \cong (q^! \k' , \psi_\bullet)
\]
of $G$-Lie algebroid pairs over $G$,
where $\psi_{\bullet}: \g \ltimes G \to q^! \k'$ is defined in Lemma \ref{lem:qB-pair}.
 We now show that as Lie algebras,
$\k'$ is identified with $\k^{\rm op}$, the vector space $\k$ with the negative of the original Lie bracket on $\k$.
Indeed, consider
the pullback Lie algebroid
$q^! \k^{\rm op}$ over $G$, which  by definition is the direct product Lie algebroid $TG \oplus \k^{\rm op}$ over $G$.
By Lemma \ref{lem:qB-pair}, one has
the $G$-Lie algebroid pair $(TG \oplus \k^{\rm op}, \phi_\bullet)$ over $G$, where
\[
\phi_\bullet: \; \g \ltimes G \rightarrow TG \oplus \k^{{\rm op}}, \; (u, g) \mapsto (-u^\lt|_g, \, 0).
\]
Define the fiber-wise surjective map
\[
\epsilon:  \;\d \ltimes G \rightarrow \k^{{\rm op}}, \; (a, \, g) \mapsto -{\rm p}_\k ({\rm Ad}_g a).
\]

\begin{lem}\label{lem:dG0} \cite[Proposition 8]{Mein}
The map $\epsilon$ is a Lie algebroid morphism. Equivalently, the map $I:  \d \ltimes G \rightarrow TG \oplus \k^{\rm op}$ given by
\begin{equation}\label{eq:Iag0}
I(a, g) = (\rho_\d(a)|_g, \, -{\rm p}_\k ({\rm Ad}_g a)) =
\left(-\rt_g ({{\rm p}}_\g ({\rm Ad}_g a)), \;  -{\rm p}_\k ({\rm Ad}_g a)\right),
\end{equation}
is an isomorphism of $G$-Lie algebroid pairs over $G$ from $(\d \ltimes G, \psi_{\g, \d})$ to $(TG \oplus \k^{\rm op}, \phi_\bullet)$.
\end{lem}

\begin{proof}
It follows from the definitions that $I \circ \psi_{\g, \d} = \phi_\bullet: \g \ltimes G \to TG \oplus \k^{\rm op}$.
It remains to show that $I: \d \ltimes G \to TG \oplus \k^{\rm op}$ is a Lie algebroid isomorphism covering the identity map of $G$.
As $I$ is a vector bundle isomorphism, we only need to show that $I$ is a Lie algebroid morphism. For $a \in \d$,
define $\kappa_a: G \to \k$, as in the proof of Lemma \ref{lem:d-G}, and $I_a \in \Gamma(TG \oplus \k^{\rm op})$ by
\[
\kappa_a(g) = {\rm p}_\k ({\rm Ad}_g a) \hs \mbox{and} \hs
I_a|_g = I(a, g) = (\rho_\d(a)|_g, \, -\kappa_a(g)).
\]
Let $a, b \in \d$. We then need to show that $I_{[a, b]} = [I_a, I_b] \in \Gamma(TG \oplus \k^{\rm op})$. By definition,
\[
[I_a, I_b] = ([\rho_\d(a), \rho_\d(b)], \, -\rho_\d(a) \cdot \kappa_b + \rho_\d(b) \cdot \kappa_a - [\kappa_a, \kappa_b]).
\]
Using the notation $\gamma_a$ in the proof of Lemma \ref{lem:d-G}, and using
\eqref{eq:abG-1},  one has
\begin{align*}
\rho_\d(a) \cdot \kappa_b - \rho_\d(b) \cdot \kappa_a + [\kappa_a, \kappa_b]& =
{\rm p}_\k [\gamma_a, \kappa_b] - {\rm p}_\k [\gamma_b, \kappa_a] + [\kappa_a, \kappa_b] \\
& = {\rm p}_\k [\gamma_a + \kappa_a, \gamma_b + \kappa_b] = \kappa_{[a, b]}.
\end{align*}
By Lemma \ref{lem:d-G}, we have $[I_a, I_b] = (\rho_\d([a b]), \, -\kappa ([a, b]) = I_{[a, b]}$.
\end{proof}

For a Lie group $K$, let $K^{\rm op}$ be the manifold $K$ with the opposite group structure.

\begin{cor}\label{cor:dG}
For any Harish-Chandra triple $(G, \d, \k)$ and any Lie group $K$ with Lie algebra $\k$, the product groupoid $K^{\rm op} \times G \times G\rightrightarrows
G$
is an integration of the action Lie algebroid $\d \ltimes G$, where $K^{\rm op}$ is regarded as a Lie groupoid over
the one point space, and $G \times G$ is the pair groupoid over $G$.
\end{cor}

\subsection{Integrations of action Lie sub-algebroids of $\d \ltimes G$}\label{ss:l-G}
Continuing with the set-up in $\S$\ref{ss:dressing}, let now $\l$ be any Lie sub-algebra of $\d$,
and let $\l \ltimes G$ be the action Lie algebroid defined by the restriction to $\l$ of the dressing action
of $\d$ on $G$. Being an action Lie algebroid, $\l \ltimes G$ is integrable by Dazord's construction \cite{Dazord}
(see also \cite[$\S$5.2]{CF1}), but in view of
Corollary \ref{cor:dG}, one can expect to construct integrations of $\l \ltimes G$ that are  sub-groupoids of
$K^{\rm op} \times G \times G \rightrightarrows G$. We now show that this is indeed the case under Assumption \ref{asp:GD} below.
We first simplify some terminology.

\begin{defn}\label{def:ST}
{\rm
For two Lie algebras $\fs$ and $\ft$, by an {\it integration of a Lie algebra homomorphism $\fs \to \ft$}
we mean a map $\phi: S \to T$, where $S$ and $T$ are two (not necessarily connected) Lie groups with respective Lie algebras $\fs$ and $\ft$,
and $\phi$ is a  Lie group homomorphism whose differential at the identity element of $S$ is the given $\fs \to \ft$. In this case, we also say that
the Lie algebra homomorphism $\fs \to \ft$ integrates to the Lie group homomorphism $\phi: S \to T$.
%\hfill $\diamond$
}
\end{defn}

\begin{defn}\label{defn:GD-inte}
{\rm A $G$-equivariant Lie algebra triple $(\d, \g, \k)$ is said to be {\it integrable} if the Harish-Chandra pair $(G, \d)$ is integrable in the sense that
there exists a Lie group homomorphism $G \to D$ integrating $\g \to \d$.}
\end{defn}

\begin{rem}\label{re:GD}
{\rm
The Harish-Chandra pair $(G, \d)$ is automatically integrable if $G$ is simply connected.
In general, suppose $G$ is connected, and let $\widetilde{G}$ be the simply connected cover of $G$ with covering map
$p: \widetilde{G} \to G$, and regard $\pi_1(G)$ as a sub-group in the center of $\widetilde{G}$.
 Let $\widetilde{D}$ be the simply connected Lie group with Lie algebra $\d$, and let $\phi: \widetilde{G} \to \widetilde{D}$
be the unique Lie group homomorphism integrating $\g \to \d$. Then  $(G, \d)$ is integrable
 if and only if $\phi(\pi_1(G))$ is a discrete sub-group of the center of
$\widetilde{D}$, and in this case $\widetilde{D}/\phi(\pi_1(G))$ is a maximal connected Lie group $D$ admitting an integration
$G \to D$ of $\g \to \d$ (c.f. \cite[Remark~1.13.6]{DK}).
%See \cite[Theorem 5.3]{FS19}.
%\hen{This is a known (standard?) result in the theory of Lie groups, citing a textbook seems more appropriate
%[e.g. Duistermaat-Kolk? should double check it]  }
\hfill $\diamond$
}
\end{rem}

\begin{assump}\label{asp:GD}
{\rm Assume that the $G$-equivariant Lie algebra triple $(\d, \g, \k)$ is integrable, and let $\phi_\sG:  G \to D, g \mapsto \bar{g}$,
be an integration of $\g \to \d$. Let $\phi_\sK: K \to D, k \mapsto \bar{k}$, be any integration of $\k \hookrightarrow \d$.
}
\end{assump}

Under Assumption \ref{asp:GD}, let
\begin{equation}\label{eq:GD-poid}
\G(D) = \{(k, g_1, g_2, d) \in K \times G \times G \times D|\, \bar{k} \bar{g}_1= \bar{g}_2d^{-1}\}.
\end{equation}
As $\G(D) \to K \times  G \times G, (k, g_1, g_2, d) \mapsto (k, g_1, g_2)$, is a diffeomorphism, $\G(D)$ has the Lie groupoid structure
isomorphic to the product
Lie groupoid $K^{\rm op} \times G \times G\rightrightarrows G$  in Corollary \ref{cor:dG}, and is thus an integration of the action Lie algebroid $\d \ltimes G$.
Specifically, the Lie groupoid structure on $\G(D) \rightrightarrows G$ is given by
\begin{align}\label{eq:GD-s1}
&\s (k,g_1,g_2,d)=g_2,\;\;\;\; \t (k,g_1,g_2,d)=g_1,\;\; \\
\label{eq:GD-s2}
&1_g=(e,g,g,e),\,\,\;\; (k,g_1,g_2,d)^{-1}\!=(k^{-1},g_2,g_1,d^{-1}),\\
\label{eq:GD-s3}
&(k_1,g_1,g,d_1)\cdot (k_2,g,g_2,d_2)=(k_2k_1,g_1,g_2,d_1d_2).
\end{align}
One can also see directly  that the
Lie algebroid $A_\sD = {\rm Ker} (\dif \s)|_{\sG}$ of $\G(D)$ is given by
\begin{equation}\label{eq:AsD}
A_\sD|_{g} = \{(-{\rm p}_{\k} \Ad _{{g}} a, -\rt_{g}  ({\rm p}_{\g} (\Ad _{{g}} a) ), 0, a): a \in \d\} \subset \k \times T_gG \times \{0\} \times \d.
\end{equation}
Recalling the dressing vector field $\rho_\d(a)$  in \eqref{eq:aG1}, one sees that the anchor of $A_\sD$ is
\[
\dif \t|_{\sA_\sD}: \; (-{\rm p}_{\k} \Ad _{{g}} a, -\rt_{g}  ({\rm p}_{\g} (\Ad _{{g}} a) ), 0, a) \mapsto
-\rt_{g}  ({\rm p}_{\g} (\Ad _{{g}} a) ) = \rho_\d(a)|_g,
\]
and one has the Lie algebroid isomorphism
 \begin{equation}\label{eq:d-AsD}
\d \ltimes G \rightarrow  A_\sD, \; (a, g)\mapsto  (-{\rm p}_{\k} (\Ad _{{g}}a), \,\rho_\d(a)|_g,\, 0, \, a).
\end{equation}

\begin{cor}\label{cor:GL0}
With Assumption \ref{asp:GD}, let $\l$ be any Lie sub-algebra of $\d$, and let $\phi_\sL: L \to D, l \mapsto \bar{l}$, be
any integration of $\l \hookrightarrow \d$. Then
\begin{equation}\label{eq:GL0}
\G(L)  :=\{ (k,g_1,g_2,l)\in K\times G\times G\times L \, | \, \bar{k}\bar{g}_1=\bar{g}_2\bar{l}^{-1} \}
\end{equation}
is a Lie sub-groupoid of $\G(D)$ via $\G(L) \to \G(D), (k, g_1, g_2, l) \mapsto (k, g_1, g_2, \bar{l})$, and $\G(L)$ is an integration of the action Lie algebroid
$\l \ltimes G$ over $G$.
\end{cor}

%Consider $K^{\scriptscriptstyle{op}} \times (G \times G) \times L$, viewed as a Lie groupoid over $G$ given by the
%direct product of the Lie groups $K^{\scriptscriptstyle{op}}$ and $L$ with the pair groupoid $G\times G$. Then $\G(L)$
%coincides with the fibered product of the groupoid morphisms
%$K^{\scriptscriptstyle{op}}\times (G\times G) \to D$, $(k,g_1,g_2)= \bar{g}_1^{-1}\bar{k}^{-1}\bar{g}_2$ and $\phi_{\sL}:  L \to D, \;\; l \mapsto \bar{l}$.
%Since the former is a submersion, $\G(L)$ sits in $K^{\scriptscriptstyle{op}} \times (G \times G) \times L$
%as a (embedded) Lie sub-groupoid (see e.g. \cite[$\S$5.3]{MM}
%and \cite[Proposition A.1.4]{BuCadH}).
%Note that, as a manifold, $\G(L)$ is locally diffeomorphic to $L\times G$ via $(k,g_1,g_2,l)\mapsto (l, g_2)$.

\subsection{Integrations of Lie algebroids associated to Harish-Chandra sub-pairs}\label{ss:BHl}

\begin{defn}\label{defn:sub-pair}
Let $(G, \d)$ be a Harish-Chandra pair. A {\it Harish-Chandra sub-pair} of $(G, \d)$
is a pair $(H, \l)$, where $H$ is a
closed Lie sub-group of $G$ with Lie algebra $\h$,  $\l$ is a Lie sub-algebra of $\d$ containing $\h$, and $\l$ is $H$-invariant
for the action $h \mapsto {\rm Ad}_h \in {\rm Aut}(\d)$.
\end{defn}

Let again $G$ be any Lie group and let $(\d, \g, \k)$ be a $G$-equivariant Lie algebra triple.
Let $(H, \l)$ be a Harish-Chandra sub-pair of $(G, \d)$. Regard $G$ as a principal $H$-bundle via the left $H$-action $h \mapsto \rt_{h^{-1}}$ on $G$, and let
 $q_\sH: G \to G/H$ be the projection. Let $(\l/\h) \times_H G$ be the quotient of the
$H$-action on $(\l/\h) \times G$ given by
\begin{equation}\label{eq:hag}
h \cdot ((a+\h), g) = ({\rm Ad}_h a + \h, \, gh^{-1}), \hs a \in \l.
\end{equation}
On the other hand, $H$ acts on the action Lie algebroid $\l \ltimes G$ by restriction of the $G$-action on $\d \ltimes G$ in \eqref{eq:GdG}, and
$(\l \ltimes G, \psi_{\h, \l})$ is an $H$-Lie algebroid pair over $G$, where
$\psi_{\h, \l}: \h \ltimes H \to \l \ltimes G$ is defined by
$\h \hookrightarrow \l$. The pushforward to $G/H$ of $\l \ltimes G$ (see Lemma \ref{lem:A-bar}) has
$(\l/\h) \times_H G$ as the underlying vector bundle and will be denoted as
\begin{equation}\label{eq:BHl}
B_{\sH, \l} := \overline{(\l \ltimes G)}^{\psi_{\\h, \l}} = (\l/\h) \times_H G.
\end{equation}

%We can now prove our main result on integrations of the Lie algebroid $B_{\sH, \l}$.

\begin{them}\label{th:inte-BHl}
Let $(H, \l)$ be a Harish-Chandra sub-pair of $(G, \d)$, and let $\phi_\sG: G \to D$ and $\phi_\sK: K \to D$ be as in
Assumption \ref{asp:GD}. Assume, in addition, that
\[
\phi_\sL: L \to D \hs \mbox{and} \hs \phi_{\sH, \sL}: H \to L
\]
are respectively integrations of $\l \hookrightarrow \d$ and $\h \hookrightarrow \l$ such that
$\phi_\sL \circ \phi_{\sH, \sL} = \phi_{\sG}|_\sH: H \to D$. Let $H \times H$ act on $\G(L)$ in \eqref{eq:GL0} by
\begin{equation}\label{eq:HHGL}
(h_1, h_2) \cdot (k, g_1, g_2, l) = (k,\,  g_1h_1^{-1}, \, g_2h_2^{-1}, \, \phi_{\sH, \sL}(h_1) l \phi_{\sH, \sL}(h_2^{-1})).
\end{equation}
Then there is a unique Lie groupoid structure on $\G(L)/(H \times H)$ over  $G/H$ such that the projection
$\G(L) \to \G(L)/(H \times H)$ is a groupoid morphism covering $q_\sH: G \to G/H$. Moreover,
the Lie algebroid of $\G(L)/(H \times H) \rightrightarrows G/H$
is isomorphic to $B_{\sH, \l}$.
\end{them}

\begin{proof}
Under the stated assumptions, one has the Lie groupoid morphism
\[
\Psi_{\sH, \sL}: \; H \ltimes G \rightarrow \G(L), \; (h, g) \mapsto (e, gh^{-1}, g, \phi_{\sH, \sL}(h)),
\]
and $(\G(L), \Psi_{\sH, \sL})$ is an $H$-Lie groupoid pair over $G$ integrating the $H$-Lie algebroid pair $(\l \ltimes G, \psi_{\h, \l})$.
Theorem \ref{th:inte-BHl} now follows from Proposition \ref{prop:quot}.
\end{proof}

\begin{rem}\label{re:H1}
{\rm
When $H$ is connected, since there is only one integration $H \to D$ of $\h \hookrightarrow \d$, the condition
$\phi_\sL \circ \phi_{\sH, \sL} = \phi_{\sG}|_\sH: H \to D$ in Theorem \ref{th:inte-BHl} is always satisfied.
\hfill $\diamond$
}
\end{rem}

We now show that the Lie algebroid $B_{\sH, \l}$ over $G/H$ is always integrable by removing  Assumption \ref{asp:GD} and those in Theorem \ref{th:inte-BHl}.

\begin{them}\label{th:inte-BHl-all}
For any Lie group $G$, any  $G$-equivariant Lie algebra triple $(\d, \g, \k)$,
and any Harish-Chandra sub-pair $(H, \l)$ of $(G, \d)$, the Lie algebroid $B_{\sH, \l} = (\l/\h) \times_H G$ over $G/H$ given in \eqref{eq:BHl}
 is integrable.
\end{them}

\begin{proof} Suppose that we have proved Theorem \ref{th:inte-BHl-all} for connected $G$.
For an arbitrary $G$,
let $G_0$ be the connected component of $G$ through the identity element. The distinct $G_0$-orbits in $G/H$, being open, connected, and
 pair-wise disjoint, are precisely all the connected components of $G/H$. Each connected component $\Gamma$ is of the form $G_0/H_{g_0}$, where
$H_{g_0} = G_0 \cap (g_0 H g_0^{-1})$ for some $g_0 \in G$. Restrict the dressing action of $\d$ on $G$ to the open subset $G_0g_0$ of $G$ and consider the
action groupoid $\d \ltimes (G_0g_0)$. Since for $a \in \d$, the dressing vector field
$\rho_\d(a)$ in \eqref{eq:aG1} satisfies
\[
\rho_\d(a)(gg_0) = \rt_{g_0} (\rho_\d({\rm Ad}_{g_0} a)|_g), \hs g \in G_0,
\]
one has the Lie algebroid isomorphism
\[
\d \ltimes (G_0g_0) \rightarrow \d \ltimes G_0, \; (a, gg_0) \mapsto ({\rm Ad}_{g_0} a, \, g), \hs g \in G.
\]
By restriction, the
$G$-equivariant Lie algebra triple $(\d, \g, \k)$ is also $G_0$-equivariant, and $(H_{g_0}, \l_{g_0})$ is a
Harish-Chandra sub-pair of $(G_0, \d)$, where $\l_{g_0} = {\rm Ad}_{g_0} \l$.  By our assumption,  $B_{\sH, \l}|_{\Gamma} \cong
B_{\sH_{g_0}, \l_{g_0}}$ is integrable. By taking the union over connected components, one sees that $B_{\sH, \l}$ is integrable.

We may thus assume that $G$ is connected. Let $\widetilde{G}$ be the simply connected cover of $G$, and let $p: \widetilde{G} \to G$ be
the projection. Let $\widetilde{H} = p^{-1}(H) \subset \widetilde{G}$, so that $\widetilde{G}/\widetilde{H} \to G/H,
\tilde{g}\widetilde{H} \mapsto p(\tilde{g})H$, is a diffeomorphism.
 By regarding $(\d, \g, \k)$ as a $\widetilde{G}$-equivariant Lie algebra triple and
$(\widetilde{H}, \l)$ as a Harish-Chandra sub-pair of $(\widetilde{G}, \d)$, and by noting that $B_{\sH, \l} \cong
B_{{\scriptscriptstyle \widetilde{H}}, \l}$, we may assume that $G$ is
simply connected.

Let now $H_0$ be the connected component of $H$ through the identity element, and consider the Lie algebroid $B_{\sH_0, \l}$ over
$G/H_0$. If $B_{\sH_0, \l}$ is integrable, then the action of $H/H_0$ on $B_{\sH_0, \l}$ lifts to an action on its
 source-simply-connected integration whose quotient by $H/H_0$ is an integration of $B_{\sH, \l}$. We may thus assume that $H$ is connected.

If $G$ simply connected and $H$ connected, then all the assumptions in Theorem \ref{th:inte-BHl} are satisfied, and the construction in
Theorem \ref{th:inte-BHl} gives explicit integrations of $B_{\sH, \l}$.
\end{proof}

%%%%%%%%%%%%%%%%%%%%%%%%%%%%%%%%%%%%%%%%%%%%%%%%%%%%%%%%%%%%%%%%
\subsection{Integrations of homogeneous spaces of ${\mathcal{LA}}^\vee$-Lie groups}\label{ss:LA-vee}
%\subsection{Remarks}\label{ss:LA-vee}
%\hen{perhaps a more specific title for this subsection? something alluding to the viewpoint via ${\mathcal{LA}}^\vee$-Lie groups?}
Let again $G$ be any Lie group and $(\d, \g, \k)$ a $G$-equivariant Lie algebra triple.
Consider now the action Lie algebroid
$E := \k \ltimes G$.
Note that
\[
g \cdot \xi = {\rm p}_\k ({\rm Ad}_g \xi), \hs g \in G, \, \xi \in \k,
\]
defines a left action of $G$ on $\k$. Regarding the $G$-action on $\k$ as one on $\k^{\rm op}$, one then has the action Lie groupoid
$E \rightrightarrows \k^{\rm op}$, which is
 an ${\mathcal{LA}}$-groupoid \cite[$\S$4]{Mein}
via
\begin{equation}\label{eq:LA-EG}
\xymatrix{
E \ar[d] \ar@<0.5ex>[r]  \ar@<-0.5ex>[r]   &\k^{\rm op}\ar[d]\\
G \ar@<0.5ex>[r]  \ar@<-0.5ex>[r] &{\rm pt}.}
\end{equation}
The pair $(G, E)$ is an {\it ${\mathcal{LA}}^\vee$-Lie group} \cite[page 1094]{Mein}, i.e.,
a vacant $\mathcal{LA}$-groupoid with double base being a point.
Moreover, every Lie algebroid $E$ on $G$ making
$(G, E)$ into an ${\mathcal{LA}}^\vee$-Lie group arises in this way \cite[top of page 1101]{Mein}. Furthermore,
for each Harish-Chandra sub-pair $(H, \l)$ of $(G, \d)$, the pair $(G/H, B_{\sH, \l})$ is a {\it homogeneous
space of the ${\mathcal{LA}}^\vee$-Lie group $(G, E)$}, and all homogeneous spaces of $(G, E)$ arise this way
by \cite[Proposition 9]{Mein}. We thus have the following reformulation of Theorem \ref{th:inte-BHl-all}.

\begin{them}\label{th:inte-LAvee}
Every homogeneous space of any ${\mathcal{LA}}^\vee$-Lie group is integrable.
\end{them}

\begin{rem}\label{re:Daniel-2}
{\rm
{\it Dirac Lie groups} as defined in \cite{L-B-M:Dirac} are examples of ${\mathcal{LA}}^\vee$-Lie groups \cite[page 1097]{Mein}.
Theorem \ref{th:inte-LAvee} for Dirac Lie groups has also been proved in \cite{Alv}. The proof in \cite{Alv}, however, uses
Dazord's integrations of action Lie algebroids by means of holonomy groupoids, while ours uses the explicit integrations $\G(L)$
which involve only integrations of (finite dimensional) Lie algebras into
Lie groups.
\hfill $\diamond$
}
\end{rem}

We now turn to integrations of the ${\mathcal{LA}}^\vee$-Lie group $(G, E)$.
Taking $L = K$, the Lie groupoid $\G(K) \rightrightarrows G$ integrates the Lie algebroid $E = \k \ltimes G$ by Corollary \ref{cor:GL0}.
Now $\G(K)$ also has a Lie groupoid structure over $K^{\rm op}$,
with source and target maps
\[
(k_1, g_1, g_2, k_2) \mapsto k_2, \hs (k_1, g_1, g_2, k_2) \mapsto k_1^{-1},
\]
and multiplication $(k_1, g_1, g_2, k_2)(k_2^{-1}, g_1^\prime, g_2^\prime, k_2^\prime) =(k_1, g_1g_1^\prime, g_2g_2^\prime, k_2^\prime)$.
By \cite[Theorem 5.2]{Mac-Mokri}, these two groupoid structures on $\G(K)$ make it into
a {\it double Lie groupoid} \cite{Mac0}
\begin{equation}\label{eq:GKG}
\xymatrix{
\G(K) \ar@<0.5ex>[d] \ar@<-0.5ex>[d] \ar@<0.5ex>[r]  \ar@<-0.5ex>[r]   &K^{\rm op}\ar@<0.5ex>[d] \ar@<-0.5ex>[d]\\
G\ar@<0.5ex>[r]  \ar@<-0.5ex>[r] &{\rm pt}
}
\end{equation}
integrating the ${\mathcal{LA}}$-groupoid in \eqref{eq:LA-EG}.

The Lie groupoids $\G(L)$ in \eqref{eq:GL0} has a natural left action of $\G(K) \rightrightarrows K^{\rm op}$
along the map $J: \G(L)\to K^{\rm op}$, $J(k, g_1, g_2, l) \mapsto k^{-1}$, given by $\kappa: \G(K)\times_{K^{\rm op}}\G(L) \to \G(L)$,
$$
\kappa((k_1, g_1, g_2, k_2),(k_2^{-1}, g_1', g_2', l))= (k_1, g_1 g_1', g_2 g_2', l).
$$
As the source map of $\G(K)\rightrightarrows K^{\rm op}$
is a Lie groupoid morphism from $\G(K) \rightrightarrows G$ to $K^{\rm op}\rightrightarrows {\rm pt}$, and as
 ${J}$ is a Lie groupoid morphism from $\G(L) \rightrightarrows G$
to $K^{\rm op}\rightrightarrows {\rm pt}$,
the fibered product  $\G(K)\times_{K^{\rm op}} (\G(L)$ is a Lie groupoid over $G\times G$.
With respect to this Lie groupoid structure, the action map $\kappa$
is a groupoid morphism
covering the multiplication map $G\times G \to G$ and
defines an action of the double Lie groupoid $\G(K)$ on the Lie groupoid $\G(L)$
in the sense of \cite[Definition 1.5]{BrownMackenzie}.

 Let now $(H, \l)$ be a Harish-Chandra sub-pair of $(G, \d)$, and consider the Lie groupoid  $\G(L) /(H \times H) \rightrightarrows G/H$
in Theorem \ref{th:inte-BHl}. Then the map $J$ gives rise to
\[
\overline{J}: \; \G(L) /(H \times H)\rightarrow K^{\rm op}, \;  [k, g_1, g_2, l] \mapsto k^{-1},
\]
which is still a morphism of Lie groupoids, and $\G(K)\times_{K^{\rm op}} (\G(L)/(H \times H))$ is a Lie groupoid over $G\times (G/H)$. The action $\kappa$ descends to
an action
\begin{align*}
\overline{\kappa}:\; \;&\G(K)\times_{K^{\rm op}} (\G(L)/(H \times H)) \rightarrow \G(L)/(H \times H), \\
&  ((k_1,g_1,g_2,k_2), \; [k_1^\prime ,g_1^\prime ,g_2^\prime,l]) \mapsto
[k_1,g_1g_1^\prime, g_2g_2^\prime,l] \;\; \mbox{when}\;\; k_2^{-1} = k_1^\prime,
\end{align*}
which is a groupoid morphism
covering the action map $G\times (G/H) \to G/H$ and
defines an action of the double Lie groupoid $\G(K)$ on the Lie groupoid $\G(L)/(H \times H)$ in the same sense as before. Thus $\overline{\kappa}$ is
an integration of the action of the ${\mathcal{LA}}^\vee$-Lie group $(G, E)$ on its homogeneous space $(G/H, B_{\sH, \l})$.

%%%%%%%%%%%%%%%%%%%%%%%%%%%%%%%%%%%%%%%%%%%%%%%%%%%%%%%%%%%%%%%%%%%%%%%%%
\section{Integrations of Dirac structures via pullbacks}\label{s:Dirac}
Let $q: M \to M/H$ be any principal bundle. After recalling some basic facts from Dirac geometry, we prove in this section some
general results relating integrations of Dirac structures on $M/H$ with that
 of their pullback Dirac structures on $M$ by $q$.

%%%%%%%%%%%%%%%%%%%%%%%%%%%%%%%%%%%%
\subsection{Dirac structures and pre-symplectic groupoids}\label{ss:kernel}
We will now recall some facts about Dirac manifolds
\cite{villa, courant} and their integrations \cite{BCWZ}.

For a smooth manifold $M$, the vector bundle $\ttt M:=TM\oplus T^*M$
has a non-degenerate fiber-wise bilinear form given at
each $x\in M$ by
\begin{equation}\label{eq:lara-tttM}
\langle (X,\alpha ),(Y,\beta )\rangle =\beta (X) + \alpha (Y),
\end{equation}
and a bracket $\lcf \cdot ,\cdot \rcf$ on $\Gamma (\ttt M)$, called
the {\it Courant-Dorfman bracket}, given by
\begin{equation}\label{eq:Dorfman}
\lcf (X,\alpha ),(Y,\beta )\rcf = ([X,Y], \Lie_\sX \beta - i_\sY\dif \alpha ),
\end{equation}
where $\Lie_\sX$ stands for Lie derivative by $X$.
Denote by $\pr_\sT\colon \ttt M \to TM$ and $\pr _{\sT^*}\colon \ttt
M\to T^*M$ the canonical projections. The data $(\ttt M, \langle \cdot,\cdot \rangle, \lcf \cdot,\cdot \rcf, \pr_\sT)$
is the motivating example for the general notion of {\em Courant algebroid}, introduced in \cite{LWX} (see also \cite{L-B-M}).

A  {\it Dirac structure} on $M$ is a vector sub-bundle $E\subset
\ttt M$ which is {\em lagrangian} with respect to $\langle \cdot ,\cdot
\rangle$ and {\em involutive} with respect
to $\lcf \cdot ,\cdot \rcf$.
Prototypical examples of Dirac structures include closed
2-forms $\omega\in \Omega^2(M)$ (in which case
$E$ is the graph of the map $TM\to
T^*M$, $X\mapsto i_\sX\omega$) and Poisson structures $\pi\in
\mathfrak{X}^2(M)$ (in which case $E$ is the graph of $\pi^\sharp: T^*M\to TM$, $\alpha\mapsto i_\alpha\pi$).
A key fact is that, for any Dirac structure $E$ on $M$, the vector bundle
$E\to M$ acquires the structure of a Lie algebroid, with anchor
$\rho_\sE=\pr_\sT|_E:  E\to TM$
and Lie bracket on $\Gamma(E)$ given by the restriction of $\lcf
\cdot ,\cdot \rcf$. Dirac structures in the context of general Courant algebroids are defined analogously
\cite{LWX}.

The {\it kernel} of a Dirac structure $E$ on $M$ is the
(generalized) distribution defined by
$$
\Ker(E):=E\cap TM\subseteq TM.
$$
A Dirac structure $E$ is (the graph of) a Poisson structure if and only if $\Ker(E) = \{0\}$.
For $E$ defined by a closed 2-form
$\omega$, one has $\Ker(E)=\Ker(\omega)$.

Given Dirac manifolds $(M_1,E_1)$ and
$(M_2,E_2)$, a map $\varphi: M_1\to M_2$ is called a (forward) {\it Dirac
map} if
$$
E_2 |_{\varphi(x)} = \{(\dif \varphi (X), \beta ) \;|\;
(X,\varphi^*\beta) \in E_1 |_x\}
$$
for all $x \in M_1$. We call $\varphi$ a {\it strong Dirac map} (see e.g. \cite{ABM,BCWZ}) if, in addition,
$$
\Ker(\dif \varphi)\cap \Ker(E_1)=\{0\}.
$$

\begin{rem}\label{rem:action}
{\rm
Note that if $\varphi$ is a Dirac map, then $\Ker(E_2)|_{\varphi(x)}
= \dif \varphi (\Ker(E_1) |_x )$ for all $x \in M$. The condition that $\varphi$ be a strong
Dirac map is equivalent to
\begin{equation}\label{eq:kernelphi}
\dif \varphi : \Ker(E_1)|_x \to
\Ker(E_2)|_{\varphi(x)}
\end{equation}
be an isomorphism for all $x \in M$. A strong Dirac map $\varphi: (M_1,E_1)\to (M_2,E_2)$ gives rise to a
Lie algebroid action of $E_2$ on $M_1$, defined by the map
$$
\varphi^*E_2 \to TM_1, \; (Y,\beta)|_{\varphi(x)} \mapsto X,
$$
where $X$ is unique such that $Y=\dif \varphi(X)$ and
$(X,\varphi^*\beta) \in E_1|_x$. This map restricts to a vector-bundle map
$\varphi^*\Ker(E_2)\to \Ker(E_1)$, which is the point-wise inverse to
\eqref{eq:kernelphi}.
\hfill $\diamond$
}
\end{rem}

The fact that Poisson manifolds are infinitesimal counterparts of
symplectic groupoids \cite{We87} extends to Dirac geometry as follows. Recall
that a 2-form $\omega$ on a Lie groupoid $\G$ is called
{\it multiplicative} if
\begin{equation}\label{eq:mult}
m^*\omega =\pr _1^*\omega +\pr _2^*\omega
\end{equation}
as $2$-forms on $\G^{(2)}$, where $\pr _i\colon \G^{(2)}\to \G$, $i=1,2$, are the natural
projections. If $A = \Ker(\dif \s)|_\sM$
is the Lie algebroid of $\G$ with anchor $\rho = \dif \t|_\sA: A \to TM$, then every $\omega \in \Omega^2(\G)$ gives rise to the
vector bundle morphisms
\begin{align}\label{eq:mu-omega}
&\mu_\omega: \; A\to T^*M, \; \mu_\omega(a)=(i_a\omega) |_{TM},\\
\label{eq:delta-omega}
&\delta_\omega: \; A \to \ttt M, \; \delta_\omega(a) = (\rho(a), \, \mu_\omega(a)).
\end{align}
%\hen{$\delta$ also used later for Lie-algebra differential.... but probably no confusion}

\begin{defn}\label{defn:pre} \cite{BCWZ}
{\rm A {\it pre-symplectic groupoid} is a Lie
groupoid $\G\rightrightarrows M$ equipped with a closed,
multiplicative 2-form $\omega\in \Omega ^2(\G)$ such that  $\mathrm{dim}(\G)=2 \mathrm{dim}(M)$ and
$\delta_\omega: A \to \ttt M$ is injective.
}
\end{defn}

Any pre-symplectic groupoid  $(\G,\omega)$ defines a Dirac structure
$E = \delta_\omega(A)$ on $M$, which is also uniquely determined by the fact that the target map $\t$
is a Dirac map \cite[Theorem ~2.2]{BCWZ}, while $\s$ is
anti-Dirac; that is, $\s$ is a Dirac map if we equip $M$
with the {\it opposite} Dirac structure
$$
{E}^{\scriptscriptstyle{op}}=\{(X,\alpha)\,|\, (X,-\alpha)\in E\}.
$$
(See Proposition \ref{prop:ker}(1) below.)  In this situation,
$\delta_\omega:A \to E$
 is an isomorphism of Lie algebroids, and
we say that $(\G,\omega)$ is an {\it integration} of the Dirac structure $E$.

 Conversely, if a Dirac structure $E$ is integrable as a Lie algebroid, its
source-simply-connected integration $\G$ carries a unique multiplicative 2-form
$\omega$ so that $(\G,\omega)$ is an integration of the Dirac
structure $E$. Moreover, if a pre-symplectic groupoid $(\G,\omega)$
integrates a Dirac structure $E$, then $E$ corresponds to a Poisson
structure if and only if $\omega$ is symplectic, in which case $\t$
is a Poisson map. For details on the relation between
pre-symplectic groupoids and Dirac structures, see
\cite[$\S$2.4]{BCWZ}.

Let $(\G, \omega)$ be an integration of a Dirac structure $E$ on $M$.
Analogous to the Lie algebroid isomorphism $\delta_\omega: A \to E$,  there is an
isomorphism of Lie algebroids
\begin{equation}\label{eq:identif2}
\delta_\omega^{\scriptscriptstyle{op}}: \; A^{\scriptscriptstyle{op}} \to {E^{\scriptscriptstyle{op}}}, \;\; \; a\mapsto
(\rho^{\scriptscriptstyle{op}}(a),\mu_\omega^{\scriptscriptstyle{op}}(a)),
\end{equation}
where $\mu_\omega^{\scriptscriptstyle{op}}: A^{\scriptscriptstyle{op}} \to
T^*M$ is given by $a\mapsto (i_a\omega)|_{TM}$. Upon the
identifications $A\cong E$ and $A^{\scriptscriptstyle{op}}\cong {E^{\scriptscriptstyle{op}}}$,
the isomorphism $\mathrm{inv}: A \to
A^{\scriptscriptstyle{op}}$ becomes $E\to {E^{\scriptscriptstyle{op}}},
(X,\alpha)\mapsto (X,-\alpha)$.
We will often assume these identifications.

Our main result in this section is the following explicit relation between the kernel of a Dirac structure and the kernel of an integrating pre-symplectic form.

\begin{prop}\label{prop:ker}
Let $(\G,\omega )$ be a pre-symplectic groupoid integrating a Dirac
structure $E$ on $M$, and let $M^{\scriptscriptstyle{op}}$ denote $M$ equipped with the opposite Dirac structure. Then

(1) the map $(\t,\s)\colon \G \to M\times
M^{\scriptscriptstyle{op}}$ is  a strong Dirac map;

(2) under the identifications $\delta_\omega: A \stackrel{\sim}{\to} E$ and
$\delta_\omega^{\scriptscriptstyle{op}}:  A^{\scriptscriptstyle{op}} \stackrel{\sim}{\to} E^{\scriptscriptstyle{op}}$,
 one has
\begin{equation}\label{eq:kernel:omega}
\Ker (\omega)|_g = \{ a^\rt +  \mathrm{inv}(b)^\lt \, \mid \, a \in
\Ker(E)|_{\t(g)}, b\in \Ker(E)|_{\s(g)}\}, \;\;\; g \in \G.
\end{equation}
\end{prop}

\begin{pf} Let $g\in \G$ with $\s(g)=x$ and $\t(g)=y$.
For $a\in A|_y$ and $b\in A^{\scriptscriptstyle{op}}|_x$, recall that $\dif \t (a^\rt)|_g=\rho(a)|_{y}$
and $\dif \s (b^\lt)|_g=\rho^{\scriptscriptstyle{op}}(b)|_{x}$, and that condition \eqref{eq:mult}
implies that
\begin{equation}\label{eq:id1}
i_{a^\rt}\omega =\t^\ast (\mu_\omega(a)),\qquad i_{b^\lt}\omega =\s^\ast (\mu_\omega^{\scriptscriptstyle{op}}(b)),
\end{equation}
see e.g. \cite[$\S$3]{BCWZ}.
Consider the identifications $A \to E$ and $A^{\rm op} \to E^{\rm op}$, and let $a=(X,\alpha) \in E|_y$, $b=(Y,\beta) \in E^{\scriptscriptstyle{op}}|_x$.
%in such a way that $\mathrm{inv}(b)=(Y,-\beta) \in E^{\scriptscriptstyle{op}}|_x$.
Then $Z= (a^\rt + b^\lt)|_g$ satisfies $(\dif \t, \dif \s)(Z)=(X,Y)$ and
$$
i_{\sZ}\omega = i_{a^\rt}\omega + i_{b^\lt}\omega =
\t^*\alpha + \s^*\beta = (\t,\s)^* (\alpha,\beta),
$$
which shows that $(\t,\s)\colon \G \to M\times M^{\scriptscriptstyle{op}}$ is a Dirac map.
By \cite[Corollary 4.8]{BCWZ}, the injectivity of $\delta_\omega$ in Definition~\ref{defn:pre} is equivalent to the condition
\begin{equation}\label{eq:ker}
\Ker (\omega)_g\cap \Ker (\dif \s)_g \cap \Ker (\dif \t)_g=\{0\},
\end{equation}
for all $g\in \G$, indicating that $(\t, \s)$ is in fact a strong Dirac map. This proves (1).

By Remark~\ref{rem:action} we have the Lie algebroid
action of $E \times E^{\scriptscriptstyle{op}}$ on $\G$ given by
\begin{equation}\label{eq:LAaction}
\t^*E
\times \s^*E^{\scriptscriptstyle{op}} \to T\G, \;\;\;\; (a,b)\mapsto a^\rt + b^\lt,
\end{equation}
and its restriction defines an isomorphism
$\Ker (E)|_{\t(g)}\times \Ker (E^{\scriptscriptstyle{op}})|_{\s(g)} \to \Ker (\omega)|_{g}$, for all $g \in \G$. Using the identification
$\mathrm{inv}: E\to E^{\scriptscriptstyle{op}}$, \eqref{eq:kernel:omega} follows. This proves (2).
\end{pf}

\subsection{Integrations of Dirac structures via pullbacks}\label{ss:PP-Dirac}
Let now $H$ be a Lie group with Lie algebra $\h$, and consider an $H$-principal bundle
$q: M \to M/H$ as in $\S$\ref{s:PP}. Recall
the action Lie algebroid $\h \ltimes M$ over $M$ with anchor
$\rho_\sM:  \mathfrak{h} \ltimes M\rightarrow TM$, and note that
\[
{\rm Im}(\rho_\sM) = \Ker (\dif q: TM \to T(M/H)).
\]
Let $F$ be an arbitrary Dirac structure on $M/H$. The pullback
of $F$ by $q$  (see, e.g., \cite{villa}) is the Dirac structure given by
\begin{equation}\label{eq:EF}
E  = q^!F = \{(X, q^*\beta)|\,(\dif q(X), \beta) \in F\} \subset \ttt M.
\end{equation}
Let $H$ act on $\ttt M$ by tangent and co-tangent lifts.
%\hen{statement about morphisms added in next result}
Part (1) of the following Lemma \ref{lem:E-F} is proved in
\cite[Lemma 6.5]{MeinLectures}, and Part (2) follows from the definitions. See also \cite[$\S$5]{villa}.
%\hen{I think just part (1) below is proved in \cite{MeinLectures} . Part (2) follows from the definitions, we can just mention that without proof...}

\begin{lem}\label{lem:E-F}
(1) The assignment $F \mapsto E =  q^! F$ is a one-to-one correspondence from Dirac structures on $M/H$ and $H$-invariant Dirac structures
on $M$ satisfying
\begin{equation}\label{eq:rho-kerE}
{\rm Im} (\rho_\sM) \subset \Ker (E).
\end{equation}
 Under this correspondence, $F$ is the {\it pushforward} of $E$, i.e.,
\[
F = q_!E :=\{(\dif q(X), \alpha)| (X, q^*\alpha) \in E\} \subset \ttt (M/H),
\]
and one has $\Ker (F) = \dif q (\Ker (E))$ and $\Ker (E) = (\dif q)^{-1}(\Ker(F))$. In particular, $F$ is the graph of a Poisson structure on $M/H$ if and only if
$E = q^!F$ is an $H$-invariant Dirac structure on $M$ such that ${\rm Ker}(E) = {\rm Im} (\rho_\sM)$.

(2) For $i = 1, 2$, let $q_i: M_i \to M_i/H$ be a principal bundle,  $F_i$ a Dirac structure on $M_i/H$, and $E_i = q_i^!F_i$.
Let $f: M_1\to M_2$ be an $H$-equivariant map covering $\bar{f}: M_1/H\to M_2/H$. Then $\bar{f}: (M_i/H, F_1) \to (M_2/H, F_2)$ is a Dirac map
if and only if so is $f: (M_1, E_1) \to (M_2, E_2)$.
\end{lem}

Comparing with the definition of pullbacks of Lie algebroids in $\S$\ref{ss:PP-algebroids},  it is clear that for a Dirac structure $F$ on $M/H$,
the Lie algebroid underlying the pullback Dirac structure  $E = q^! F$
is isomorphic to the pullback Lie algebroid of $F$. More specifically, $(E, \psi)$ is  an $H$-Lie algebroid pair over $M$ (see Definition \ref{defn:H-pair}),
where
\begin{equation}\label{eq:psi-E}
\psi: \; \h \ltimes M \rightarrow E, \, (u, x) \mapsto (u_\sM|_x, 0),
\end{equation}
 and the Lie algebroid morphism
$E \rightarrow F,  (X, q^*\beta) \mapsto (\dif q (X), \beta)$,
gives a Lie algebroid isomorphism $\overline{E}^\psi \to F$. As integrability of a Dirac structure
is defined as that of the underlying Lie algebroid, we immediate have the following consequence of Theorem \ref{th:inte-BHl}.

\begin{prop}\label{pr:EF-inte}
A Dirac structure $F$ on $M/H$ is integrable if and only if the $H$-Lie algebroid pair $(q^!F, \psi)$ over $M$ is integrable.
%, where $E = q^!F$
%is the pullback Dirac structure of $F$, regarded as a Lie algebroid, and $\psi: \h \ltimes M \to E$ is given in \eqref{eq:psi-E}.
\end{prop}

We now turn to pre-symplectic groupoids integrating Dirac structures.

\begin{defn}\label{defn:admi1}
{\rm
Let $F$ be a Dirac structure on $M/H$.
A pre-symplectic groupoid $(\G, \omega)$ integrating the pullback Dirac structure $E = q^!F$ is said to be {\it $H$-admissible} if
the Lie algebroid morphism $\psi: \h \times M \to E$ in  \eqref{eq:psi-E} integrates to
a Lie groupoid morphism $\Psi: H\ltimes M \to \G$,  and
if $\omega$ is $(H\times H)$-invariant for the $(H\times H)$-action on $\G$ in \eqref{eq:double-action}.
\hfill
}
\end{defn}

%The following theorem is the basis for the main results of the paper.

\begin{them}\label{th:inte-EF}
Let $F$ be a Dirac structure on $M/H$ and $E = q^!F$ the pullback Dirac structure on $M$.
If $(\G, \omega)$ is an $H$-admissible pre-symplectic groupoid integrating $E$, then
there is a unique pre-symplectic structure $\overline{\omega}$ on the quotient Lie groupoid $\overline{\G} = \G/(H \times H)$
such that $p^*\overline{\omega} = \omega$, and  $(\overline{\G}, \overline{\omega})$ is a pre-symplectic groupoid integrating $F$.
\end{them}

\begin{pf} Note first that the $(H \times H)$-action on $\G$ induces  the $(\h\times \h)$-action on $\G$ by
\[
\rho_{\scriptscriptstyle{\G}} \colon (\h\times \h)\times \G \to T\G,\;\;\;\;
\rho_{\scriptscriptstyle{\G}} (u,v)|_g=(\psi (u,\t(g)))^\rt+(\mathrm{inv}(\psi (v,\s(g))))^\lt.
\]
As ${\rm Im}(\rho_\sM) \subset {\rm Ker}(E)$, it follows from Proposition \ref{prop:ker} that
\begin{equation}\label{eq:hh-ker-0}
\rho_{\scriptscriptstyle{\G}} (\h\times\h)\subset \Ker (\omega).
\end{equation}
The $(H\times H)$-invariance of $\omega$ then implies that
$\omega$ is $(H \times H)$-basic, i.e., $\omega = {p}^*\overline{\omega}$ for a unique
2-form $\overline{\omega}$ on {$\overline{\G}$}. Since $p: \G\to \overline{\G}$ is a groupoid morphism, $\overline{\omega}$ is multiplicative.
One also has ${\rm dim}(\G/(H \times H)) = 2 {\rm dim} (M/H)$.
%\hen{added what comes next}
Finally, let $\overline{A}$ be the Lie algebroid of $\overline{\G}$, and consider the corresponding map
$\delta_{\overline{\omega}}: \overline{A}\to \mathbb{T}(M/H)$. One may verify that
$\Ker(\delta_\omega)=\{0\}$ implies that $\Ker(\delta_{\overline{\omega}}) =\{0\}$. Indeed,
suppose $\overline{a}\in \Ker(\delta_{\overline{\omega}})$, so $\overline{\rho}(\overline{a}) = \dif q (\rho(a))=0$ and
$q^*\mu_{\overline{\omega}}(\overline{a})=\mu(a)=0$, where $a$ is any such that $\overline{a}=\dif p (a)$.
Then $\rho(a)= u_M = \rho(\psi(u))$, so we can modify $a$ to $a'= a-\psi(u)$ such that $\dif p (a')=\overline{a}$,
and $\mu(a')=0$ and $\rho(a')$=0, which gives $a'=0$, and thus $\overline{a}=0$.
We conclude that $(\overline{\G}, \overline{\omega})$ is a pre-symplectic groupoid integrating $(M/H, F)$.
\end{pf}

\begin{rem}\label{rem:H-conneced}
{\rm
For any pre-symplectic groupoid $(\G, \omega)$ integrating $E = q^!F$  and a Lie groupoid morphism $\Psi: H\ltimes M \to \G$ integrating
$\psi: \h \times M \to E$ in  \eqref{eq:psi-E},
it follows from $\dif \omega = 0$ and \eqref{eq:hh-ker-0} that $\Lie_{\rho_{\scriptscriptstyle{\G}} (u,v)} \omega =0$ for all $(u, v) \in \h \times \h$.
If $H$ is connected, then the condition that
 $\omega$ be
$(H \times H)$-invariant is automatically satisfied.
\hfill $\diamond$
}
\end{rem}

\begin{cor}\label{co:pi-inte}
Let $\pi$ be a Poisson structure on $M/H$ and $E = q^!{\rm gr}(\pi)$ the pullback Dirac structure on $M$.
If $(\G, \omega)$ is an $H$-admissible pre-symplectic groupoid integrating $E$, then
there is a unique symplectic structure $\overline{\omega}$ on the quotient Lie groupoid $\overline{\G} = \G/(H \times H)$
such that $p^*\overline{\omega} = \omega$, and  $(\overline{\G}, \overline{\omega})$ is a symplectic groupoid integrating $(M/H, \pi)$.
\end{cor}

%%%%%%%%%%%%%%%%%%%%%%%%%%%%%%%%%%%%%%%%%%%%%%%%%%%%%%%%%%%%%%%%%%%%%%%%%%%%%%%%%
\section{Every Poisson homogeneous space is integrable}\label{s:homog-G}
Following \cite{MeinLectures}, we review in $\S$\ref{ss:Poi-Dirac} and $\S$\ref{ss:affine-E} classifications of
Poisson homogeneous spaces in terms of equivariant Manin triples (of Lie algebras) and Harish-Chandra sub-pairs. Applying  Theorem
\ref{th:inte-BHl-all}, we prove in $\S$\ref{ss:main-inte} the main theorem stated in the Introduction that every Poisson homogeneous space is
integrable.

\subsection{Poisson Lie groups and equivariant Manin triples}\label{ss:Poi-Dirac}
Recall that a {\em quadratic Lie algebra} is a Lie
algebra together with a non-degenerate symmetric ad-invariant bilinear form.
Let $G$ be any Lie group, not necessarily connected or simply connected.

\begin{defn}\label{defn:G-Manin} \cite[Definition 5.10]{MeinLectures}
{\rm
A {\it $G$-equivariant Manin triple} is a triple
\[
((\d, \lara), \g, \k),
\]
 where $(\d, \g, \k)$ is a $G$-equivariant Lie algebra triple
as in Definition \ref{defn:G-triples}, and $(\d, \lara)$ is a quadratic Lie algebra such that $\lara$ is
$G$-invariant and both $\g$ and $\k$ are lagrangian with respect to $\lara$. The decomposition $\d = \g + \k$ is also called a {\it lagrangian splitting} of $\d$.
}
\end{defn}

Recall that a Poisson structure $\pi_\sG$ on $G$ is said to be {\it multiplicative} if
the group multiplication map $(G, \pi_\sG) \times (G, \pi_\sG) \to (G, \pi_\sG)$ is Poisson, and in such a case
the pair $(G, \pi_\sG)$ is called a {\it Poisson Lie group} \cite{Drinfeld2}.
Let $\pi\sG$ be a multiplicative Poisson structure on $G$. Then $\pi_{\sG}(e)=0$, and the linearization of $\pi_\sG$ at $e$, i.e.,
\begin{equation}\label{eq:delta-s}
\delta_*: \g \to \wedge^2 \g, \;\; u\mapsto (\Lie_{u^\lt} \pi_\sG)|_e =
\frac{d}{d\epsilon}\Big|_{\epsilon =0} \rt_{\exp(-\epsilon u)}(\pi_{\sG}|_{\exp(\epsilon u)}),
\end{equation}
defines a Lie bracket $[\cdot,\cdot]_{\g^*}$ on $\g^*$ via
$\delta_*(u)(\xi,\eta)=-u([\xi,\eta]_{\g^*})$, where $u \in \g$ and $\xi,\eta\in \g^*$.
The pair  $(\g, \delta_*)$ is called the {\it Lie bialgebra} \cite{Drinfeld} of the Poisson Lie group
$(G, \pi_\sG)$.
Let $\d = \g \oplus \g^*$ be the direct sum vector space, equipped with the bilinear form
\begin{equation}\label{eq:bilinear:d}
\la u + \xi, \;w + \eta \ra  = \eta(u) + \xi(w),
\end{equation}
Then there is a unique Lie bracket on $\d$ making $(\d, \lara)$ into a quadratic Lie algebra and such that
both $\g$ and $\g^*$ are Lie sub-algebras of $\d$.
Furthermore, by \cite{Drinfeld2}, the adjoint action
 of $\g$ on $\d$ integrates to
an action of $G$ on $\d$
via
\begin{equation}\label{Gond}
\Ad_g (u + \xi) : = \Ad_{g} u + i_{\Ad_{g^{-1}}^{*} \xi} (\rt_{g^{-1}} (\pi_{\sG}|_g)) +\Ad_{g^{-1}}^{*} \xi,
\end{equation}
where $g \mapsto \Ad_{g^{-1}}^{*}$ is  the  co-adjoint representation of $G$ on $\g^*$.
 It follows from \eqref{Gond} and $\pi_\sG$ being a bi-vector field on $G$ that $\lara$ is $G$-invariant.
Thus $((\d,\lara), \g, \g^*)$ is a $G$-equivariant Manin triple.
Using \eqref{Gond}, the Poisson structure $\pi_\sG$ on $G$ can be recovered from the
$G$-equivariant Manin triple $((\d, \lara), \g, \g^*)$ as
\begin{equation}\label{eq:piG-D0}
\pi_{\sG}|_g(\xi_1^\lt, \xi_2^\lt) =
\langle {\rm p}_\g (\Ad_{{g}}\xi_1), \, {\rm p}_{\g^*} (\Ad_{{g}}\xi_2)\rangle, \hs \xi_1, \xi_2 \in \g^*,
\end{equation}
where $\xi^\lt$ for $\xi \in \g^*$ is the left invariant $1$-form
 on $G$ with value $\xi$ at the identity element of $G$, and
${\rm p}_\g: \d \to \g$ and ${\rm p}_{\g^*}: \d \to \g^*$ are the projections with respect to the decomposition $\d = \g + \g^*$.
Note that  \eqref{eq:piG-D0} gives
\begin{equation}\label{eq:piGsharp}
\pi_\sG^\#(\xi^\lt)|_g = \rt_g (\pr_\g {\rm Ad}_g \xi), \hs g \in G, \, \xi \in \g^*.
\end{equation}
For a given Lie group $G$, the assignment
$(G, \pi_\sG) \mapsto ((\d, \lara), \g, \g^*)$
 is thus a
one-to-one correspondence from Poisson Lie group structures on $G$ to $G$-equivariant Manin triples, with the inverse correspondence
given via \eqref{eq:piG-D0}. For a proof of this fact using Dirac geometry, see \cite[Theorem 5.12]{MeinLectures} (proved in $\S$6.1 therein).

\subsection{Affine Dirac structures and Poisson homogeneous spaces}\label{ss:affine-E}
For a Poisson Lie group  $(G, \pi_\sG)$ with the corresponding $G$-equivariant Manin triple $((\d, \lara), \g, \g^*)$, we now recall from \cite{MeinLectures}
an alternative presentation
of the Courant algebroid $\ttt G$. Note first that for any Poisson structure $\pi_\sM$ on any manifold $M$, one has
the splitting
\[
\ttt M = TM \oplus {\rm gr} (\pi_\sM) = \{(X + \pi_\sM^\#(\alpha), \alpha): (X, \alpha) \in \ttt M\}.
\]
When applied to the Poisson Lie group $(G, \pi_\sG)$ and using left trivializations of both $TG$ and $T^*G$,
one arrives at the vector bundle isomorphism
\begin{equation}\label{eq:I}
{\bf e}: \;\d \times G \stackrel{\cong}{\longrightarrow} \ttt G,\;\;\; {\bf e}(u+\xi, g) = - (u^\lt +\pi_{\sG}^\sharp(\xi^\lt), \, \xi^\lt)|_g,\hs u \in \g, \, \xi \in \g^*.
\end{equation}
On the other hand, by \eqref{eq:aG1} and \eqref{eq:piGsharp} one has the dressing action of $\d$ on $G$ given by
\begin{equation}\label{eq:dress}
\rho_\d: \d \to \mathfrak{X}^1(G),\;\;\; \rho_\d(u+\xi)  = -(u^\lt + \pi_{\sG}^\sharp(\xi^\lt)).
\end{equation}
One checks directly using \eqref{eq:piG-D0} that the stabilizer sub-algebra of the dressing action at $g \in G$ is ${\rm Ad}_{g^{-1}} \g^*$
which is a lagrangian Lie sub-algebra of $(\d, \lara)$.
By \cite{L-B-M}, the trivial vector bundle $\d \times G$ thus has the structure of an {\it action Courant algebroid} over $G$,
with anchor given by the dressing action of $\d$ on $G$ and the Lie bracket on $\Gamma(\d \times G)$ uniquely
determined by the Lie bracket on $\d$ identified with the space of constant sections.  We denote by $\d_\sG$ the vector bundle $\d \times G$ with the
action Courant algebroid structure.

\begin{lem}\label{lem:d-tttG}
\cite[Lemma 6.4]{MeinLectures}
For any Poisson Lie group $(G, \pi_\sG)$, the map  ${\bf e}$ in \eqref{eq:I} is a $G$-equivariant isomorphism of Courant algebroids, where
$G$ acts on $\ttt G$ by tangent and co-tangent  lifts and on $\d_\sG$ by
$g \cdot (a, g') = ({\rm Ad}_g a, \, g'g^{-1})$.
\end{lem}

 It follows from Lemma \ref{lem:d-tttG} that each
lagrangian Lie sub-algebra $\l$ of $(\d, \lara)$ gives rise to a Dirac structure ${\bf e}(\l_\sG) \subset \ttt G$ on $G$, where
\[
\l_{\sG}: = \l\times G \subset \d_\sG,
\]
and the Lie algebroid underlying the Dirac structure ${\bf e}(\l_\sG)$ is isomorphic to the action Lie algebroid $\l \ltimes G$
defined by the restriction of the dressing action to $\l$. In particular,
\[
{\bf e}(\g_\sG) = TG \hs \mbox{and} \hs
{\bf e}((\g^*)_\sG) = {\rm gr}(\pi_\sG).
\]
To characterize the Dirac structures on $G$ that are of the form ${\bf e}(\l_\sG)$ for a lagrangian Lie sub-algebra of $(\d, \lara)$,
consider a second action of $G$ on $\d_\sG$ given by
\[
g \bullet (a, \, g') = (a, \, gg'),\hs g, g' \in G, \, a \in \d.
\]
Denote also by $\bullet$ the unique $G$-action on $\ttt G$ such that ${\bf e}: \d_\sG \to \ttt G$ is $G$-equivariant. Explicitly, one has, by definition,
\[
g \bullet \left((u^\lt + \pi^\#_\sG(\xi^\lt), \;\xi^\lt)|_{g'}\right) = (u^\lt + \pi^\#_\sG(\xi^\lt), \;\xi^\lt)|_{gg'} , \hs g, g' \in G, \, u \in \g, \xi \in \g^*.
\]
By \eqref{eq:rho-gg}, the dressing vector fields $\rho_\d(u+\xi) =-(u^\lt + \pi^\#_\sG(\xi^\lt))$ satisfy
\begin{equation}\label{eq:rho-d-gg}
\rho_\d(u+\xi)|_{gg'} = \rt_{g'} \left(\rho_\d({\rm Ad}_{g'}(u + \xi))|_g\right), \hs  u \in \g, \xi \in \g^*, g, g' \in G.
\end{equation}
Writing an arbitrary element in $\ttt_{g'} G = T_{g'} G \oplus T_{g'}^* G$ as $(v_{g'}, \alpha_{g'})$, one has, by \eqref{eq:rho-d-gg},
\begin{equation}\label{eq:gg-E}
g \bullet (v_{g'}, \,\alpha_{g'})  = \left(\lt _g v_{g'} + \rt_{g'} \left(\pi_\sG^\#((r_{g'}^*\alpha_{g'})^\lt)|_g\right), \; \lt_{g^{-1}}^* \alpha_{g'}\right)
\in \ttt_{gg'} G.
\end{equation}
Note that the $\bullet$-action
 is the unique action of $G$ on $\ttt G$ that leaves invariant the splitting $\ttt G = TG \oplus {\rm gr} (\pi_\sG)$ and such that
its restriction to $TG$ is the tangent lift of the action of $G$ on itself by left multiplication,  and its
restriction to ${\rm gr}(\pi_\sG)$ is, via the isomorphism
\[
{\rm gr}(\pi _\sG) \cong T^*G, \; (\pi_\sG^\#(\alpha), \alpha) \mapsto \alpha,
\]
the co-tangent lift of the same action.
We refer to \cite[$\S$6.1]{MeinLectures} for the original definition of the $\bullet$-action of $G$ on $\ttt G$.
%and to \cite[$\S$5.2]{Mein} in the more general setting of Dirac actions of Dirac Lie groups.

%\hen{This paragraph is confusing...}
The following Lemma \ref{lem:GE} is a special case of \cite[Theorem 22]{Mein}. We outline a direct proof of this result,
as \cite[Theorem 22]{Mein} deals with Dirac Lie groups and has a more involved proof.

\begin{lem}\label{lem:GE} For a Poisson Lie group $(G, \pi_\sG)$ and a Dirac structure $E$ on $G$, the following statements are equivalent:

(1) The group multiplication
$m:  (G,\pi_{\sG}) \times (G,E) \to (G,E)$
 is a Dirac map;

(2)  $E \subset \ttt G$ is $G$-invariant for the $\bullet$-action of $G$ on $\ttt G$;

(3) $E= {\bf e}(\l_{\sG})$ for a lagrangian Lie sub-algebra $\l$ of $(\d, \lara)$; in this case,
$\l = E|_e$.
\end{lem}

\begin{proof}
The equivalence between (2) and (3) follows from the $G$-equivariance of the Courant algebroid isomorphism
${\bf e}: \d_\sG \to \ttt G$ for the $\bullet$-actions of $G$.
To prove the
equivalence of (1) and (2), we note by a direct calculation that
 $m:  (G,\pi_{\sG}) \times (G,E) \to (G,E)$ is a Dirac map
if and only if
\[
E|_{gg'} = \left\{\left(\lt _g v_{g'} + \rt_{g'} \left(\pi_\sG^\#((r_{g'}^*\alpha_{g'})^\lt)\right)\Big{|}_g, \; \lt_{g^{-1}}^* \alpha_{g'}\right): \, (v_{g'}, \alpha_{g'})
\in E|_{g'}\right\}, \hs g, g' \in G,
\]
which, by \eqref{eq:gg-E}, is equivalent to $E_{gg'} = g \bullet E|_{g'}$. Thus (1) and (2) are equivalent.
\end{proof}

\begin{rem}\label{re:strong}
{\rm
Using the fact that ${\rm Ker}({\rm gr}(\pi_\sG)) = 0$, it is easy to see that
the group multiplication
$m:  (G,\pi_{\sG}) \times (G,E) \to (G,E)$
 is a Dirac map if and only if it is a strong Dirac map
 (see Remark \ref{rem:action}).
\hfill $\diamond$
}
\end{rem}

\begin{defn}\label{defn:affine0}
{\rm
For a Poisson Lie group $(G, \pi_\sG)$, a Dirac structure $E$ on $G$ is said to be $(G, \pi_\sG)$-affine
if it satisfies either one of the equivalent conditions in Lemma \ref{lem:GE}.
}
\end{defn}

Turning now to
 Poisson homogeneous spaces, let  $(G, \pi_\sG)$ be a Poisson
Lie group, and let $((\d, \lara), \g, \g^*)$ be the associated $G$-equivariant Manin triple.
For a closed Lie sub-group $H$ of $G$, let $H$ act on $G$ by $h \mapsto \rt_{h^{-1}}$ and on $\ttt G$ by tangent and co-tangent lifts.
Let  $q: G \to G/H$ be the projection, and let $G$ act on $G/H$ by $g \cdot (g'H) = gg'H$.
Recall that if $\pi$ is a Poisson structure on $G/H$ such that the action map
$(G, \pi_\sG) \times (G/H, \pi) \to G/H, \pi)$
 is Poisson, then $(G/H, \pi)$ is called a {\it Poisson homogeneous space}
of  $(G, \pi_\sG)$.

Given a Poisson homogeneous space $(G/H, \pi)$ of $(G, \pi_\sG)$, by Part (2) of Lemma \ref{lem:E-F},
 the pullback Dirac structure $E = q^!{\rm gr}(\pi)$ on $G$ satisfies (1) in Lemma \ref{lem:GE} and
is thus $(G, \pi_\sG)$-affine. Let
\begin{equation}\label{eq:lpie}
\l_\pi =(q^!{\rm gr}(\pi))_e = \{u + \xi \,|\,  u \in \g,\, \xi \in \mathrm{Ann}(\h), i_\xi
(\pi|_{q(e)}) = u + \h\} \subset \d,
\end{equation}
where $\mathrm{Ann}(\h) = \{\xi \in \g^*\,|\,\xi|_\h = 0\}$. By Lemma \ref{lem:GE}, $\l_\pi$ is a lagrangian Lie sub-algebra of $(\d, \lara)$ and
$q^!{\rm gr}(\pi) = {\bf e}((\l_\pi)_\sG)$. The following classification theorem of Drinfeld on Poisson homogeneous spaces of
$(G, \pi_\sG)$ is now a consequence of Lemma \ref{lem:E-F}. See \cite[Theorem 5.14]{MeinLectures} and  \cite[$\S$6.2]{MeinLectures} for a detailed proof.

\begin{prop}\label{prop:Hl}
 For any Poisson Lie group $(G, \pi_\sG)$, the
assignment
\[
(G/H, \pi)  \mapsto (H, \l_\pi)
\]
 is a one-to-one correspondence from Poisson homogeneous spaces $(G/H, \pi)$
of $(G, \pi_\sG)$ to Harish-Chandra sub-pairs $(H, \l)$ of $(G, \d)$ such that $\l$ is a lagrangian
Lie sub-algebra of $(\d, \lara)$ and $\l \cap \g = \h$.
\end{prop}

The lagrangian Lie sub-algebra $\l_\pi$ of $(\d, \lara)$ in \eqref{eq:lpie} is called the {\it Drinfeld lagrangian Lie sub-algebra}
associated to the Poisson homogeneous space $(G/H, \pi)$.

\subsection{Every Poisson homogeneous space is integrable}\label{ss:main-inte}
We can now prove the theorem stated in the Introduction.

\begin{them}\label{th:main-intro}
Every Poisson homogeneous space is integrable.
\end{them}

\begin{proof}
Let $(G,\pi_\sG)$ be any Poisson Lie group, and let $((\d, \lara), \g, \g^*)$ be the associated $G$-equivariant Manin triple.
Let $(G/H, \pi)$ be any Poisson homogeneous space
of $(G, \pi_\sG)$. By Lemma \ref{lem:E-F}, the co-tangent bundle Lie algebroid $T_{\pi}^*(G/H)$ over $G/H$ is isomorphic to
the Lie algebroid $B_{\sH, \l_{\pi}}$ in \eqref{eq:BHl} associated to the $G$-equivariant Lie algebra
triple $(\d, \g, \g^*)$ and
the Harish-Chandra sub-pair $(H, \l_\pi)$ of $(G, \d)$. By Theorem \ref{th:inte-BHl-all},   the Lie algebroid $T_{\pi}^*(G/H)$, and thus the Poisson
structure $\pi$, is integrable.
\end{proof}

\section{Explicit integrations}\label{s:explicit}
For a Poisson Lie group $(G, \pi_\sG)$ admitting a Drinfeld double (Definition \ref{defn:D-double}), we construct
%in $\S$\ref{ss:explicit-affine}
explicit pre-symplectic groupoids for any given $(G, \pi_\sG)$-affine Dirac structures on $G$.
Pushforwards of such pre-symplectic groupoids will then be shown to give symplectic groupoids integrating
Poisson homogeneous spaces of $(G, \pi_\sG)$. We explain in $\S$\ref{s:hol} that the construction
also applies to the holomorphic setting.

\subsection{Drinfeld doubles of Poisson Lie groups}\label{ss:Drinfeld-double}
%(For further generalizations to Dirac Lie groups, see \cite{Jotz,Mein,Rob}.)
In this section, we start with an arbitrary Manin triple $((\d, \lara), \g, \k)$, i.e., a quadratic Lie algebra $(\d, \lara)$ with two lagrangian Lie sub-algebras
$\g$ and $\k$ such that $\d = \g \oplus \k$ is a vector space direct sum.  We  identify $\k \cong \g^*$ using $\lara$.
By a Lie group {\it integrating $(\d, \lara)$} we mean a (not necessarily connected) Lie group
$D$ with Lie algebra $\d$ such that the adjoint action of $D$ on $\d$ preserves $\lara$. Note that the $\Ad_\sD$-invariance of $\lara$ is automatic if $D$ is connected.

Recall that the $r$-matrix on $\d$ is given by
\begin{equation}\label{eq:R-d}
R = \frac{1}{2}\sum_{i=1}^n \xi_i \wedge u_i \in \wedge^2\d,
\end{equation}
where $\{u_i\}_{i=1}^n$ is any basis of $\g$ and $\{\xi_i\}_{i=1}^n$ the dual basis of $\k \cong\g^*$.  Let $D$ be any
Lie group integrating $(\d, \lara)$, and for $V \in \wedge^k \d$, let $V^\lt$ and $V^\rt$
respectively denote the left and right invariant $k$-vector fields on $D$ with value $V$ at $e \in D$. Then  the bi-vector field
\begin{equation}\label{eq:piD}
\pi_\sD = R^\rt - R^\lt
\end{equation}
on $D$ is multiplicative and has Schouten bracket  $[\pi_\sD, \pi_\sD] = [R, R]^\rt - [R, R]^\lt$, where
$[R, R] \in \wedge^3\d$ is given by \cite[(2.6)]{LY:DQ}
\begin{equation}\label{eq:RR}
\langle [R, R], \, a \wedge b \wedge c\rangle = 2\langle a, \, [b, c]\rangle, \hspace{.2in} a, b, c \in \d.
\end{equation}
The invariance of $\lara$ under the adjoint action of $D$ implies that $[R, R]^\rt = [R, R]^\lt$, so
$(D, \pi_\sD)$ is a Poisson Lie group.

\begin{lem}\label{lem:G-D-Poi}
Let $D$ be any integration of $(\d, \lara)$, and let  $\phi_\sG: G \to D, g \mapsto \bar{g}$, be any integration of
$\g \hookrightarrow \d$. Then $(G, \pi_\sG)$ is a Poisson Lie group with $\pi_\sG$ given by
\begin{equation}\label{eq:piG-bar}
\pi_{\sG}|_{{g}}(\xi_1^\lt, \xi_2^\lt) =
\langle {\rm p}_\g (\Ad_{\bar{g}}\xi_1), \, {\rm p}_{\g^*} (\Ad_{\bar{g}}\xi_2)\rangle, \hspace{.2in}
\xi_1, \, \xi_2 \in \g^*,
\end{equation}
and $(\d, \lara)$ is the Drinfeld double Lie algebra of $(G, \pi_\sG)$. Furthermore, $\phi_\sG: (G, \pi_\sG) \to (D, \pi_\sD)$ is a Poisson Lie group morphism.
\end{lem}

\begin{proof}
Since the adjoint action of $D$ on $\d$ preserves $\lara$,
the action $g \mapsto {\rm Ad}_{\bar{g}}$ of $G$ on $\d$ makes $((\d, \lara), \g, \g^*)$ into a $G$-equivariant Manin triple, which by
\cite[Theorem 5.12]{MeinLectures} (see $\S$\ref{ss:Poi-Dirac}) gives rise to a multiplicative Poisson structure $\pi_\sG$ on $G$
via \eqref{eq:piG-bar}. The Lie bracket on $\g^*$ coincides with the one obtained by the linearization of $\pi_\sG$ at $e$ via \eqref{eq:delta-s}, so
$(\d, \lara)$ is the Drinfeld double Lie algebra of $(G, \pi_\sG)$.

To show that $\phi_\sG: (G, \pi_\sG) \to (D, \pi_\sD)$ is Poisson, let
$g \in G$ and note that since $\Ad_{\bar{g}}(\g) = \g$, one has
$R = \frac{1}{2}\sum_{i=1}^n  {\rm p}_{\g^*} {\rm Ad}_{\bar{g}^{-1}} \xi_i \wedge {\rm Ad}_{\bar{g}^{-1}} u_i$. Thus
\begin{align*}
\pi_{\sD}|_{\bar{g}} & =\lt_{\bar{g}} \left(-R + {\rm Ad}_{\bar{g}^{-1}} R\right)
 =\lt_{\bar{g}} \left( -R +\frac{1}{2}\sum_{i=1}^n ({\rm p}_{\g} ({\rm Ad}_{\bar{g}^{-1}}\xi_i) + {\rm p}_{\g^*} {\rm Ad}_{\bar{g}^{-1}} \xi_i)
\wedge {\rm Ad}_{\bar{g}^{-1}} u_i\right)\\
& =  \frac{1}{2}\sum_{i=1}^n \lt_{\bar{g}} ({\rm p}_{\g} ({\rm Ad}_{\bar{g}^{-1}}\xi_i)) \wedge \rt_{\bar{g}} u_i.
\end{align*}
On the other hand, one can re-write $\pi_\sG$ given in \eqref{eq:piG-bar} as
\[
\pi_{\sG}|_{{g}} =\frac{1}{2}\sum_{i=1}^n \lt_{{g}} {\rm p}_{\g} ({\rm Ad}_{{\bar{g}}^{-1}}\xi_i) \wedge \rt_{{g}} u_i.
\]
It follows that $\phi_\sG$ is Poisson.
\end{proof}

%\hfill $\diamond$
%}
%\end{rem}

\begin{defn}\label{defn:D-double}
{\rm
Let $(G, \pi_\sG)$ be a Poisson Lie group with Drinfeld double Lie algebra $(\d, \lara)$.
A {\it Drinfeld double} of $(G, \pi_\sG)$ is a pair $(D, \phi_\sG)$, where
$D$ is a Lie group integrating $(\d,\lara)$, and
$\phi_\sG: G \to D, g \mapsto \bar{g},$ is an integration of $\g \hookrightarrow \d$, such that
\begin{equation}\label{eq:Ad-gg}
{\rm Ad}_g = {\rm Ad}_{\bar{g}}: \d \to \d, \hs g \in G,
\end{equation}
where $g \mapsto {\rm Ad}_g$ is the $G$ action on $\d$ defined by $\pi_\sG$ given in \eqref{Gond}.
}
\end{defn}

\begin{rem}\label{re:G-D}
{\rm
(1) Condition \eqref{eq:Ad-gg}  ensures that $\pi_\sG$ coincides with the multiplicative Poisson structure on $G$ in \eqref{eq:piG-bar}
that is determined by the
integration $\phi_\sG: G \to D$. In particular $\phi_\sG: (G, \pi_\sG) \to (D, \pi_\sD)$ is Poisson whenever $(D, \phi_\sG)$ is a
Drinfeld double of $(G, \pi_\sG)$.
 Note that \eqref{eq:Ad-gg} is automatic when $G$ is connected.

(2) Every simply connected Poisson Lie group $(G, \pi_\sG)$
admits a Drinfeld double in the sense of Definition \ref{defn:D-double}, but we do not  know whether or
not this is true for an arbitrary Poisson Lie group $(G, \pi_\sG)$. On the other hand, starting from
any Manin triple $((\d, \lara), \g, \k)$ and any integration $D$ of $(\d, \lara)$, the Poisson Lie groups $(G, \pi_\sG)$
 in Lemma \ref{lem:G-D-Poi} have $(D, \phi_\sG)$ as Drinfeld doubles by construction.
All the examples related to semi-simple Lie groups that we will discuss in $\S$\ref{ss:semisimple} arise in this way.
\hfill $\diamond$
}
\end{rem}

\begin{rem}\label{re:Gs}
{\rm
Suppose that $\phi_{\sG^*}: G^* \to D$ is any integration of the Lie algebra inclusion $\g^* \to \d$.
 By Lemma \ref{lem:G-D-Poi}, one also has a multiplicative Poisson structure $\pi_{\sG^*}$ on $G^*$ such that
$\phi_{\sG^*}: (G^*, \pi_{\sG^*}) \to (D, -\pi_\sD)$ is Poisson. Explicitly,
\begin{equation}\label{eq:piGs}
\pi _{\sG^*} (w_1^\rt, w_2^\rt)|_k=\langle  {\rm p}_{\g}  (\Ad_{{\bar{k}}^{-1}} w_1),\,
 {\rm p}_{\g^*} ( \Ad_{{\bar{k}}^{-1}} w_2) \rangle,\hspace{.2in} k \in G^*, \, w_1, w_2 \in \g,
\end{equation}
where  $w^\rt$ for $w \in \g$ denotes the right invariant
$1$-form\footnote{Here, by the notation set up at the end of $\S$\ref{s:intro},  $w^\rt$ for $w \in \g$ denotes both
the right invariant vector field on $G$ with value $w$ at $e \in G$ and the right invariant $1$-form on $G^*$ with value $w$ at $e \in G^*$. The meaning of the
notation should be clear from the context. Similarly, for $\xi \in \g^*$, $\xi^\lt$ denotes both the left invariant vector field
on $G^*$ and the left invariant $1$-form on $G$ defined by $\xi$.} on $G^*$ with value $w \in \g \cong (\g^*)^*$
at the identity element of $G^*$.
The Poisson Lie group $(G^*, \pi{\sG^*})$ is dual to $(G, \pi_\sG)$ in the sense that the Lie bialgebra of $(G^*, \pi_{\sG^*})$ is $(\g^*, \delta)$, where
\begin{equation}\label{eq:delta}
\delta: \;\; \g^*\to \g^*\wedge\g^*, \;\; \delta(\xi)(w_1, w_2) = -\xi([w_1, w_2]_\g), \hspace{.2in} \xi \in \g^*, \, w_1, w_2 \in \g.
\end{equation}
For later use, we note that it follows directly from \eqref{eq:piGs} that
\begin{equation}\label{eq:sharpGs}
\pi_{\sG^*}^\sharp (w^\rt)|_k= -({\rm p}_{\g^*} \Ad _{{\bar{k}}^{-1}} w)^\lt|_k, \hspace{.2in} k \in G^*, \, w \in \g.
\end{equation}
\hfill $\diamond$
}
\end{rem}

\subsection{Explicit integrations of affine Dirac structures on Poisson Lie groups}\label{ss:explicit-affine}
%(For further generalizations to Dirac Lie groups, see \cite{Jotz,Mein,Rob}.)
Let $(G, \pi_\sG)$ be a Poisson Lie group.
Under the assumption that $(G, \pi_\sG)$ admits a {\it Drinfeld double}, we construct
 in this section explicit
pre-symplectic groupoids integrating every given $(G, \pi_\sG)$-affine Dirac structure on $G$.

%\begin{nota}\label{nota:G0}
%{\rm
\subsubsection{The construction}\label{ss:construction}  Assume that $(G, \pi_\sG)$ has a Drinfeld double $(D, \phi_\sG)$ and write
$\phi_{\sG}:  G \to D,  g \mapsto \bar{g}$.
Let $\l$ be any lagrangian Lie sub-algebra of  $(\d, \lara)$, and let
\[
\phi_{\sG^*}: \;\;G^*\to D, \;\; k\mapsto \bar{k}, \hspace{.2in} \mbox{and} \hspace{.2in}
\phi_{\sL}: \;\; L \to D, \;\; l \mapsto \bar{l},
\]
be any integrations of the respective Lie algebra inclusions $\g^* \to \d$ and $\l \to \d$  (see Definition \ref{def:ST}).
Define, as in \eqref{eq:GD-poid}, the Lie groupoid
\[
\G(D) = \{(k, g_1, g_2, d) \in G^* \times G \times G \times D|\, \bar{k} \bar{g}_1= \bar{g}_2d^{-1}\}
\]
over $G$, with structure maps given in \eqref{eq:GD-s1} - \eqref{eq:GD-s3}, and the Lie sub-groupoid
\begin{equation}\label{eq:GL}
\G(L)  :=\{ (k,g_1,g_2,l)\in G^*\times G\times G\times L \, | \, \bar{k}\bar{g}_1=\bar{g}_2\bar{l}^{-1} \}.
\end{equation}
%of $\G(D)$.
By Corollary \ref{cor:GL0}, $\G(L)$ integrates the action Lie algebroid $\l \ltimes G$ over $G$. We will show that $\G(L)$ admits a
multiplicative pre-symplectic  $2$-form $\omega_\sL$ such that $(\G(L), \omega_\sL)$ is an integration of $E = {\bf e}(\l_\sG)$
as a Dirac structure on $G$.
Note that $\G(L)$ depends on the choices of $(D, \phi_{\sG}), (G^*, \phi_{\sG^*})$ and $(L, \phi_{\sL})$,
but we only indicate its dependence on $L$ for notational simplicity.

We will define $\omega_\sL \in \Omega^2(\G(L))$ as the pullback to $\G(L)$ of a $2$-form $\omega_\sD$ on $\G(D)$, which, in turn, is
defined as the skew-symmetric part of a $2$-tensor field $\tau_\sD$ on $\G(D)$.

\begin{nota}\label{nota:theta}
{\rm
If $M$ is any manifold and if, for $i = 1, 2$, $\q_i$ is a Lie sub-algebra of $\d$ and $\theta_i \in \Omega^1(M, \q_i)$, let
$\langle \theta_1 \otimes \theta_2 \rangle$ be the section of $T^*M \otimes T^*M$ given by
\[
\langle \theta_1 \otimes \theta_2 \rangle(v_1, v_2) = \langle \theta_1(v_1), \,\theta_2(v_2)\rangle, \hs v_1, v_2 \in {\mathfrak{X}}^1(M),
\]
and let $\langle \theta_1 \wedge \theta_2 \rangle$ be the $2$-form on $M$ given by
\[
\langle\theta_1 \wedge \theta_2\rangle :=\langle \theta_1 \otimes \theta_2 \rangle-\langle \theta_2 \otimes \theta_1 \rangle.
\]
For a Lie group $A$, let $\theta_\sA^\lt$ and $\theta_\sA^\rt$ denote the respective left and right invariant Maurer-Cartan forms on $A$.
\hfill $\diamond$
}
\end{nota}

By abuse of notation, let $\theta^\rt_{\sG^*}, \theta^\lt_{\sGo}, \theta^\rt_{\sGt},  \theta^\rt_{\sD}$ also denote
 the pullbacks to
\[
\G(D)\subset G^*\times G\times G\times D
\]
 of the respective left- and right-invariant
Maurer-Cartan forms on the factors. Using Notation \ref{nota:theta},
define
the tensor field $\tau_\sD \in \Gamma(T^* (\G(D)) \otimes T^*(\G(D)))$  by
\begin{equation}\label{eq:tauD}
\tau_\sD = \langle{\theta^\lt_{\sGo} \otimes \theta^\rt_{\sD}}\rangle +
\langle{\theta^\rt_{\sG^*} \otimes \theta^\rt_{\sGt}}\rangle.
\end{equation}
Define the $2$-form $\omega_\sD$ on $\G(D)$ to be the skew-symmetric part of $\tau_\sD$, i.e.,
\begin{equation}\label{eq:sigmaD}
\omega_\sD  = \frac{1}{2} \langle{\theta^\lt_{\sGo} \wedge  \theta^\rt_{\sD}}\rangle + \frac{1}{2}
\langle{\theta^\rt_{\sG^*} \wedge  \theta^\rt_{\sGt}}\rangle,
\end{equation}
%\begin{align}\label{eq:tau-D}
%\tau_\sD &= \langle{\theta^\lt_{\sGo} \otimes \theta^\rt_{\sD}}\rangle +
%\langle{\theta^\rt_{\sG^*} \otimes \theta^\rt_{\sGt}}\rangle,\\
%\label{eq:omega-D0}
%\omega_\sD & = \frac{1}{2} \langle{\theta^\lt_{\sGo} \wedge  \theta^\rt_{\sD}}\rangle + \frac{1}{2}
%\langle{\theta^\rt_{\sG^*} \wedge  \theta^\rt_{\sGt}}\rangle.
%\end{align}
and let $\omega_\sL \in \Omega^2(\G(L))$ be the pullback of $\omega_\sD$ via
$\G(L) \to \G(D),  (k, g_1, g_2, l) \mapsto
(k, g_1, g_2, \bar{l})$.
Let $\theta^\rt_\sL$ also denote the pullback to $\G(L)\subset G^*\times G\times G\times L$ of the right invariant
Maurer-Cartan form on $L$. One then also has
\begin{equation}\label{eq:severaform}
    \omega_\sL= \frac{1}{2}\langle{\theta^\lt_{\sGo} \wedge \theta^\rt_{\sL}}\rangle + \frac{1}{2}\langle{\theta^\rt_{\sG^*} \wedge\theta^\rt_{\sGt}}\rangle
		\in \Omega^2(\G(L)).
\end{equation}
The following main result of this section will be proved in $\S$\ref{ss:proof-presymp}.

\begin{them}\label{thm:presymp} For any lagrangian Lie sub-algebra $\l$ of $(\d, \lara)$ and
any Lie groupoid $\G(L)$ defined in \eqref{eq:GL}, the 2-form $\omega_\sL\in \Omega^2(\G(L))$ in \eqref{eq:severaform}
is multiplicative and closed, and $(\G(L), \omega_\sL)$ is a pre-symplectic groupoid integrating the Dirac structure ${\bf e}(\l_\sG)$ on $G$.
\end{them}

\begin{rem}\label{rem:LLL}
{\rm
In the construction of the  groupoid $\G(L)$, we do not require
any of the Lie groups $D, G, G^*$ or $L$ to be connected,  nor the Lie group homomorphisms $\phi_\sS: S \to D$ for $S \in \{G, G^*, L\}$ be
injective.
Thus Theorem \ref{thm:presymp} in general
gives more than one integration of
the Dirac structure ${\bf e}(\l_\sG)$ on $G$.
In particular, starting from a quadratic Lie algebra $(\d, \lara)$, one may take
any connected Lie group $D$ with Lie algebra
$\d$ as  an integration of $(\d, \lara)$. Given a triple $\g, \g'$ and $\l$ of  lagrangian Lie sub-algebras of $(\d, \lara)$ such that
$\d = \g + \g'$ as a vector space, one may identify  $\g' \cong \g^*$
using $\lara$, and by  Lemma \ref{lem:G-D-Poi}
any integrations
\[
\phi_{\sG}: G \to D, \hs \phi_{\sG^*}: G^* \to D, \hs \phi_\sL: L \to D,
\]
respectively of  $\g \hookrightarrow \d$, $\g^* \hookrightarrow \d$, and $\l \hookrightarrow \d$,
can then be used in Theorem \ref{thm:presymp} to construct
a pre-symplectic groupoid $(\G(L), \omega)$ integrating the
$(G, \pi_{\sG})$-affine Dirac structure ${\bf e}(\l_\sG)$ on $G$.
The construction becomes canonical if we take $G$, $G^*$ and $L$ to be the (unique) connected Lie subgroups of $D$ with respective Lie algebras
$\g$, $\g'$ and $\l$.
\hfill $\diamond$
}
\end{rem}

\begin{numex}\label{ex:PLinteg}
{\rm ({\bf Affine Poisson structures on $G$}).
Assume, in the context of Theorem \ref{thm:presymp},  that $\g \cap \l = 0$. Then ${\rm Ker}({\bf e}(\l_\sG)) = 0$, so ${\bf e}(\l_\sG)$ is the graph of
a Poisson structure $\pi$ on $G$, which is {\it $(G, \pi_\sG)$-affine} \cite{Lu3}
in the sense that the group multiplication map $m: (G, \pi_\sG) \times (G, \pi) \to (G, \pi)$ is a
Poisson map. In this case the $2$-form $\omega_\sL$ in Theorem \ref{thm:presymp} is symplectic, so $(\G(L), \omega_\sL)$ is
a symplectic groupoid of $(G, \pi)$.
Let
\[
\Gamma = \{(g_1, l, k, g_2) \in G \times L \times G^* \times G|\;\bar{g}_1 \bar{l} = \bar{k} \bar{g}_2\},
\]
and equip $\Gamma$ with the Lie groupoid structure such that the map
\begin{equation}\label{eq:Gamma}
\G(L)\to \Gamma,\quad  (k, g_1, g_2, l) = (g_1, l, k^{-1}, g_2)
\end{equation}
is an isomorphism.
 We see from the expression in \eqref{eq:severaform} that $\omega_\sL$ agrees with the symplectic form on $\Gamma$ from \cite[Examples~3.1 and 3.2]{S12}.

The original description in \cite[$\S$5.5]{Lu3} of the symplectic
structure $\omega_\Gamma$ on $\Gamma$ integrating the affine Poisson structure $\pi$ on $G$ was via a non-degenerate Poisson structure defined
on $\Gamma$ as the pullback of a suitable Poisson structure on $G \times L$
by the local diffeomorphism $\Gamma \to G \times L,  (g_1, l, k, g_2) \mapsto (g_1, l)$,
see \cite[Remark 5.29]{Lu3}.
One verifies by a direct calculation (we omit the details)   that \eqref{eq:Gamma} indeed identifies $\omega_\Gamma$ with $\omega$.
\hfill $\diamond$
}
\end{numex}

\subsubsection{Another explicit formula for $\tau_\sD$}\label{ss:tauD-explicit}
As the map
\[
P: \; \G(D) \to G \times D, \; (k, g_1, g_2, d) \mapsto (g_2, d),
\]
is a local diffeomorphism, we know that the tangent bundle $T(\G(D))$, being
 isomorphic to the pullback by $P$ of the tangent bundle of the Lie group
$G \times D$, is trivial.
In this section, we use a trivialization of $T(\G(D))$ to establish another explicit formula for
the tensor field $\tau_\sD$ on $\G(D)$, thus also for the $2$-form $\omega_\sL$ in Theorem \ref{thm:presymp},
in terms of the multiplicative Poisson bi-vector
fields $\pi_\sG$ on $G$ and $\pi_{\sG^*}$ on $G^*$.
%, given respectively in \eqref{eq:piG-bar} and \eqref{eq:piGs}.

Let $\gamma = (k, g_1, g_2, d) \in \G(D) \subset G^* \times G \times G \times D$ and write
$X \in T_\gamma(G^* \times G \times G \times D)$ as
\[
X = (\eta^\lt|_k, \, v^\rt|_{g_1}, \, w^\rt|_{g_2}, \, a^\rt|_d), \hs \mbox{where}\;\;
\eta \in \g^*, \, v, w \in \g, \, a \in \d.
\]
By the defining relation of $\G(D)$ in $G^* \times G \times G \times D$, one sees that
$X \in T_\gamma(\G(D))$ if and only if $\eta \in \g^*$ and $v \in \g$ are uniquely determined by $w$ and $a$ via
\begin{equation}\label{eq:v-eta-0}
\eta + v= {\rm Ad}_{\bar{k}^{-1}} w - {\rm Ad}_{\bar{g}_1} a.
\end{equation}
Writing $a = u + \xi$ with $u \in \g$ and $\xi \in \g^*$ and by \eqref{eq:dress} and \eqref{eq:sharpGs},
\eqref{eq:v-eta-0} is equivalent to
\begin{align*}
\eta & =  \pr_{\g^*} {\rm Ad}_{\bar{k}^{-1}} w  - \pr_{\g^*} {\rm Ad}_{\bar{g}_1} a = -\lt_{k^{-1}}\pi_{\sG^*}^\#(w^\rt)|_k -{\rm Ad}_{g_1^{-1}}^* \xi,\\
v &=\pr_\g {\rm Ad}_{\bar{k}^{-1}} w - \pr_\g {\rm Ad}_{\bar{g}_1} a ={\rm Ad}_k^* w  +\rt_{g_1^{-1}} \rho_\d(a)|_{g_1} =
 {\rm Ad}_k^* w - \rt_{g_1^{-1}}( u^\lt+ \pi_\sG^\#(\xi^\lt))|_{g_1}.
\end{align*}
For $w \in \g$ and $a \in \d$, define vector fields $\widetilde{w}$ and $\widetilde{a}$ on $\G(D)$ by
\begin{align}\label{eq:wt-w}
&\widetilde{w}|_\gamma = ( -\pi_{\sG^*}^\#(w^\rt)|_k, \, ({\rm Ad}_k^*w)^\rt|_{g_1}, \, w^\rt|_{g_2}, \, 0) \in T_\gamma(\G(D)),\\
\label{eq:wt-a}
&\widetilde{a}|_\gamma = (-(\pr_{\g^*} {\rm Ad}_{\bar{g}_1} a)^\lt|_k, \, \rho_\d(a)|_{g_1}, \, 0, \, a^\rt|_d) \in T_\gamma(\G(D)),
\end{align}
where $\gamma = (k, g_1, g_2, d) \in \G(D)$. As every tangent vector of $\G(D)$ at $\gamma$ is of the form
$\widetilde{w}|_\gamma + \widetilde{a}|_\gamma$ for unique $w \in \g$ and $a \in \d$, one has the vector bundle isomorphism
\begin{equation}\label{eq:TDG-trivial}
{\mathcal{T}}: \; \G(D) \times (\g \oplus \d) \rightarrow T(\G(D)), \; (\gamma, w, a) \mapsto (\widetilde{w} + \widetilde{a})|_\gamma.
\end{equation}
Recall that the source map of $\G(D) \rightrightarrows G$ is $\s (k, g_1, g_2, d) = g_2$. Let
\begin{equation}\label{eq:pD}
\pr_\sD: \; \G(D) \rightarrow D, \; (k, g_1, g_2, d) \mapsto d.
\end{equation}
Then ${\rm Ker} (\dif \s)|_\gamma = \{\widetilde{a}|_\gamma: a \in \d\}$  and ${\rm Ker} (\dif \pr_\sD) =\{\widetilde{w}|_\gamma: w \in \g\}$.
One thus has
\begin{equation}\label{eq:TGD-decomp}
T(\G(D)) = {\rm Ker} (\dif \s) + {\rm Ker} (\dif \pr_\sD)
\end{equation}
as a direct sum of vector bundles, and
${\mathcal{T}}$ in \eqref{eq:TDG-trivial} restricts to an isomorphism
\[
 \G(D) \times \d \, \stackrel{\sim}{\longrightarrow} \, {\rm Ker}(\dif \s)
\hs \mbox{and} \hs \G(D) \times \g \, \stackrel{\sim}{\longrightarrow} \, {\rm Ker}(\dif \pr_\sD).
\]
We will refer to ${\rm Ker} (\dif \s)$  as the {\it vertical vector sub-bundle} of $T(\G(D))$ and to
${\rm Ker} (\dif \pr_\sD)$ as the {\it horizontal vector sub-bundle} of $T(\G(D))$.
Note that $\widetilde{a}$, for every $a \in \d$,
 is a right-invariant vector field on $\G(D)$.
Following the definition of $\tau_\sD$, it is now straightforward to express
$\tau_\sD$ using vertical and horizontal tangent vectors.

\begin{lem}\label{lem:tauD-wa}
For $\gamma = (k, g_1, g_2, d) \in \G(D)$,
and for $w_i \in \g$ and $a_i = u_i + \xi_i \in \d$, where $i = 1, 2$ and $u_i \in \g$ and $\xi_i  \in \g^*$, one has
\begin{align}\label{eq:tauD-explicit}
\tau_\sD(\widetilde{w}_1 + \widetilde{a}_1,\widetilde{w}_2+ \widetilde{a}_2)|_{\gamma} = & -\langle \xi_2,  u_1\rangle
- \pi_{\sG}(\xi_1^\lt,\xi_2^\lt)|_{g_1} - \langle \Ad^*_{g^{-1}_1}\xi_1, \Ad_k^*w_2 \rangle
 \nonumber\\
& + \langle \Ad^*_{g^{-1}_1} \xi_2,\Ad_k^*w_1\rangle
- \pi_{\sG^*}({w}^\rt_1,{w}^\rt_2)|_k.
\end{align}
\end{lem}

\begin{rem}\label{rem:reformulation}
{\rm
For an alternative interpretation of $\tau_\sD$, consider the fiber-wise bilinear form $\tau$ on the vector bundle $\ttt G^* \times \ttt G$ over $G^* \times G$
given as follows:
for $(k, g) \in G^* \times G$ and $V_i = ((Y_i, \beta_i), (X_i, \alpha_i)) \in \ttt_k G^* \times \ttt_g G$ for $i = 1, 2$,
\[
\tau(V_1, V_2) = \beta_1(Y_2) + \alpha_1(X_2) + (\lt^*_k \beta_1, \rt_g^* \alpha_2) - (\lt^*_k \beta_2, \rt_g^* \alpha_1),
\]
where the last two terms are defined using the pairing between $\g$ and $\g^*$.
%$(\lt^*_k \beta_1, \rt_g^* \alpha_2)$ is the pairing between  $\lt^*_k \beta_1 \in \g$ and $\rt_g^* \alpha_2 \in \g^*$, and similarly for
%$(\lt^*_k \beta_2, \rt_g^* \alpha_1)$.
We also regard $\tau$ as a fiber-wise bilinear form on the pullback vector bundle $\chi^*(\ttt G^* \times \ttt G)$,
where
\[
\chi: \; \G(D) \rightarrow G^* \times G, \;(k, g_1, g_2, d) \mapsto (k, g_1).
\]
Define the vector bundle embedding ${\mathcal{I}}: T(\G(D)) \rightarrow \chi^* (\ttt G^* \times \ttt G)$ by
%\begin{equation}\label{eq:calI}
\begin{align*}
{\mathcal{I}} ((\widetilde{w} + \widetilde{a})|_\gamma)   &=
((\pi_{\sG^*}^\#(w^\rt), w^\rt)|_k, \, (\rho_\d(a), \, -\xi^\lt|)|_{g_1})\\
 & = ((\pi_{\sG^*}^\#(w^\rt), w^\rt)|_k, \, -((u^\lt + \pi_\sG^\#(\xi^\lt)), \, \xi^\lt|)_{g_1}),
\end{align*}
where again $w \in \g, a = u + \xi \in \d$ with $u \in \g$ and $\xi \in \g^*$, and $\gamma = (k, g_1, g_2, d) \in \G(D)$.
Note that ${\mathcal{I}}(T(\G(D))) = {\rm gr}(\pi_{\sG^*}) \times \ttt G$.
Using \eqref{eq:tauD-explicit}, one checks directly that
$\tau_\sD$, as a fiber-wise bilinear form on $T(\G(D))$,  is the pullback to $T(\G(D))$ of $-\tau$ by ${\mathcal{I}}$.
\hfill $\diamond$
}
\end{rem}

The following statements on the symmetric part
$\sigma_\sD \in S^2(T^*(\G(D))$ and the skew-symmetric part  $\omega_\sD \in \Omega^2(\G(D))$ of $\tau_\sD$ follow immediately from
Lemma \ref{lem:tauD-wa}.

\begin{lem}\label{lem:onegaD-explicit}
For $\gamma = (k, g_1, g_2, d) \in \G(D)$,
and for $w_i \in \g$ and $a_i = u_i + \xi_i \in \d$, where $i = 1, 2$ and $u_i \in \g$ and $\xi_i  \in \g^*$, one has
\begin{align}
\label{eq:sigmaD-exlicit}
\sigma_\sD(\widetilde{w}_1 + \widetilde{a}_1,\widetilde{w}_2+ \widetilde{a}_2)|_{\gamma} =& -\frac{1}{2}\langle a_1, a_2\rangle,\\
\label{eq:omegaD-exlicit}
\omega_\sD(\widetilde{w}_1 + \widetilde{a}_1,\widetilde{w}_2+ \widetilde{a}_2)|_{\gamma} = &
\frac{1}{2}(\langle \xi_1,u_2\rangle- \langle \xi_2,u_1\rangle)- \pi_{\sG}(\xi_1^\lt,\xi_2^\lt)|_{g_1} - \langle \Ad^*_{g^{-1}_1}\xi_1, \Ad_k^*w_2 \rangle \nonumber\\& + \langle \Ad^*_{g^{-1}_1} \xi_2,\Ad_k^*w_1\rangle
- \pi_{\sG^*}({w}^\rt_1,{w}^\rt_2)|_k.
\end{align}
\end{lem}

Note that $\sigma_\sD$ can be written more compactly as
%, i.e.,
%\begin{equation}\label{eq:sigma-D}
%\sigma_\sD  = \frac{1}{2} \left(\langle{\theta^\lt_{\sGo} \otimes \theta^\rt_{\sD}}\rangle +\langle{\theta^\rt_{\sD} \otimes \theta^\lt_{\sGo}}\rangle
%+ \langle{\theta^\rt_{\sG^*} \otimes  \theta^\rt_{\sGt}}\rangle +  \langle{\theta^\rt_{\sGt} \otimes \theta^\rt_{\sG^*}}\rangle\right)
%\in  \in \Gamma(S^2 (T^*\G(D)))
%\end{equation}
\begin{equation}\label{eq:sigmaD-theta}
\sigma_\sD  = -\frac{1}{2} \langle \theta_\sD^\rt \otimes \theta_\sD^\rt\rangle.
\end{equation}

As an immediate consequence, we see that
for the sub-groupoid $\G(L)$ of $\G(D)$  in \eqref{eq:GL},  since $\l$ is lagrangian with respect to $\lara$,
the pullback of $\sigma_\sD$ to $\G(L)$  is
identically $0$. Thus the $2$-form $\omega_\sL$ on $\G(L)$ coincides with the pullback of $\tau_\sD$ to $\G(L)$.
Note that the trivialization of $T(\G(D))$ in \eqref{eq:TDG-trivial} restricts to a trivialization of $T(\G(L))$ by $G \times (\g \oplus \l)$
so $T(\G(L))$ also splits  into vertical and horizontal sub-bundles.
For clarity and future references, we state the formula for $\omega_\sL$ using vertical and horizontal tangent vectors.

\begin{lem}\label{lem:omegaL-wa} The $2$-form
$\omega_\sL$ on the Lie groupoid $\G(L) \rightrightarrows G$ in \eqref{eq:GL} is given by
\begin{align}\label{eq:omegaL-explicit}
\omega_\sL(\widetilde{w}_1 + \widetilde{a}_1,\widetilde{w}_2+ \widetilde{a}_2)|_{\gamma} = & -\langle \xi_2,  u_1\rangle
- \pi_{\sG}(\xi_1^\lt,\xi_2^\lt)|_{g_1} - \langle \Ad^*_{g^{-1}_1}\xi_1, \Ad_k^*w_2 \rangle
 \nonumber\\
& + \langle \Ad^*_{g^{-1}_1} \xi_2,\Ad_k^*w_1\rangle
- \pi_{\sG^*}({w}^\rt_1,{w}^\rt_2)|_k,
\end{align}
where $\gamma = (k, g_1, g_2, d) \in \G(L)$,  and for $i = 1, 2$,
$w_i \in \g$, and $a_i = u_i + \xi_i \in \l$ with $u_i \in \g$ and $\xi_i  \in \g^*$.
\end{lem}

%\hen{mention appendix here?}

A way to derive formula \eqref{eq:omegaL-explicit} from general principles in presented in Appendix~\ref{appA}.

%(see also \eqref{eq:AsD}) that the Lie algebroid $A_\sD$ of $\G(D) \rightrightarrows G$ is given by
%\begin{equation}\label{eq:AsDg}
%A_\sD|_{g} =
%\{(-{\rm p}_{\g^*} \Ad _{\bar{g}} a, \rho_\d(a)|_g, 0, a): a \in \d\} = \{\widetilde{a}|_{(e, g, g, e)}: a \in \d\}.
%\end{equation}

\subsubsection{Properties of  $\tau_\sD$ and $\omega_\sD$}\label{ss:tauD}
In this section, we elucidate some properties of the tensor field $\tau_\sD$ and the $2$-form $\omega_\sD$ on $\G(D)$ in relation to the groupoid structure
$\G(D) \rightrightarrows G$.
We first look at $\tau_\sD$ at the units of the groupoid $\G(D) \rightrightarrows G$.

For $g \in G$, $w \in \g$, and $a \in \d$, write
\[
\widetilde{w}|_g = \widetilde{w}|_{(e, g, g,e)} = (0, w^\rt|_g, w^\rt|_g, 0), \hs
\widetilde{a}|_g = \widetilde{a}|_{(e, g, g, e)} =(-{\rm p}_{\g^*} \Ad _{\bar{g}} a, \rho_\d(a)|_g, 0, a).
\]
Then $A_\sD = \Ker (\dif \s)|_\sG = \{\widetilde{a}|_g: a \in \d\}$ is
the Lie algebroid of $\G(D) \rightrightarrows G$ with anchor
\[
\rho_\sD: \; A_\sD \to TG, \;\rho_\sD(\widetilde{a}|_g) = \rho_\d(a)|_g = -\rt_g \pr_\g {\rm Ad}_{\bar{g}} a = -(u^\lt + \pi_\sG^\#(\xi^\lt))|_g, \hs a = u + \xi \in \d.
\]
Identify $T_gG \cong \{\widetilde{w}|_g: w \in \g\}$.
Then  $T_{(e, g, g, e)} \G(D) = A_\sD|_g + T_gG$ as a vector space direct sum. Define
$\mu_{\tau_\sD}: A_\sD \to T^*G$ by $\widetilde{a}|_g  \mapsto\left(\tau_\sD(\widetilde{a}|_g, \cdot )\right)|_{T_g\sG}$, and define
\begin{equation}\label{eq:delta-tauD}
\delta_{\tau_\sD} = (\rho_\sD, \mu_{\tau_\sD}): \; A_\sD \to \ttt G = TG \oplus T^*G.
\end{equation}
By \eqref{eq:tauD-explicit},  one has
$\mu_{\tau_\sD}(\widetilde{a}|_g) = -\xi^\lt|_g$ for $a = u + \xi  \in \d$. Thus
\[
\delta_{\tau_\sD}(\widetilde{a}|_g) = (\rho_\d(a), \,  -\xi^\lt)|_g = -(u^\lt + \pi_\sG^\#(\xi^\lt), \, \xi^\lt)|_g, \hs a = u + \xi \in \d.
\]
In particular, $\delta_{\tau_\sD}: A_\sD \to \ttt G$ is a vector bundle isomorphism. Recall from
\eqref{eq:I} the Courant algebroid isomorphism ${\bf e}: \d \times G \to \ttt G$, and
 from \eqref{eq:d-AsD} the
Lie algebroid isomorphism
$\varphi: \; \d \ltimes G \to A_\sD,  \;(g, a) \mapsto \widetilde{a}|_g$. The following is now immediate.

\begin{lem}\label{lem:delta-tauD}
As vector bundle isomorphisms, one has ${\bf e} = \delta_{\tau_\sD} \circ \varphi: \; \d \times G \rightarrow \ttt G.$
\end{lem}

\begin{rem}\label{rem:tauD-units}
{\rm
Under the isomorphism
$\delta_{\tau_\sD}: A_\sD \to \ttt G$, the restriction of $\tau_\sD$ to $A_\sD$ becomes the
bilinear form
$-\tau_{\scriptscriptstyle{\ttt G}}$ on the fibers of $\ttt G$, where
\[
\tau_{\scriptscriptstyle{\ttt G}}((X_1,\alpha_1), (X_2, \alpha_2)) = X_1(\alpha_2), \hs X_i \in {\mathfrak{X}}^1(G), \alpha_i \in \Omega^1(G), i = 1, 2.
\]
The restriction to $A_\sD$ of the symmetric part $\sigma_\sD$ of $\tau_\sD$ then becomes $-\frac{1}{2}$ of
the fiber-wise symmetric bilinear form on $\ttt G$ given in \eqref{eq:lara-tttM}.
Similarly, the restriction of the $2$-form $\omega_\sD$ to $A_\sD$ becomes, up to a factor of $-\frac{1}{2}$, the standard fiber-wise
skew-symmetric bilinear form on $\ttt G$.
\hfill $\diamond$
}
\end{rem}

%\begin{lem}\label{lem:thetas-2}
%For any $\gamma = (k, g_1, g_2, d) \in \G(D)$ and $\tilde{d} = \bar{k} \bar{g}_1 = \bar{g}_2 d^{-1} \in D$, one has
%\begin{equation}\label{eq:thetas-3}
%{\rm Ad}_{\tilde{d}} \circ (\theta_{\sGo}^{\lt}|_\gamma + \theta_{\sD}^{\rt}|_\gamma) = \theta_{\sGt}^{\rt}|_\gamma -\theta_{\sG^*}^{\rt}|_\gamma: \;
%T_\gamma (\G(D)) \rightarrow \d.
%\end{equation}
%\end{lem}

We now turn to the {\it multiplicativity}, or rather the {\it failure of multiplicativity},
 of $\tau_\sD$ with respect to the groupoid structure $\G(D) \rightrightarrows G$.
Let $\G(D)^{(2)}$ be the sub-manifold of $\G(D) \times \G(D)$ of composable elements, and let $m: \G(D)^{(2)} \to \G(D)$ be the
multiplication map. Let $\pr_1, \pr_2: \G(D)^{(2)} \to \G(D)$ be the projections to the two factors.

\begin{lem}\label{lem:tauD-1}
As tensor fields on $\G(D)^{(2)}$, one has
\begin{equation}\label{eq:tauD-multi}
m^* \tau_\sD - \pr_1^* \tau_\sD - \pr_2^* \tau_\sD = -\langle \pr_1^* \theta_\sD^\lt \otimes \pr_2^* \theta_\sD^\rt \rangle.
\end{equation}
\end{lem}

\begin{proof} Let $\gamma = (k, g_1, g_2, d) \in \G(D)$ and $\gamma' = (k', g_1^\prime, g_2^\prime, d^\prime) \in \G(D)$ be
such that $g_2 = g_1^\prime$, so that $(\gamma, \gamma') \in \G(D)^{(2)}$  and
$\gamma \gamma' = (k'k, g_1, g_2^\prime, dd')$. The multiplication law of $\G(D) \rightrightarrows G$ then gives the following  identities
 on linear maps $T_{(\gamma, \gamma')} \G(D)^{(2)}\to \d$:
\begin{align}\label{eq:m-theta-1}
& m^* \theta ^\lt_{\sGo}= \pr_1 ^* \theta ^\lt_{\sGo},\qquad \quad \;\;\qquad m^* \theta^\rt_{\sGt}=\pr_2 ^* \theta ^\rt _{\sGt},\\
\label{eq:m-theta-2}
&m^* \theta ^\rt _{\sD}= \pr_1 ^* \theta ^\rt _{\sD}+\Ad _{d}  \pr_2^*\theta ^\rt _{\sD},\quad\;
m^* \theta ^\rt _{\sG^*}= \Ad _{k'} \,\pr_1^*\theta ^\rt_{\sG^*}+\pr_2 ^* \theta ^\rt _{\sG^*},
\end{align}
where $\Ad _{q} \theta := {\rm Ad}_q \circ \theta$ for $q \in D$ and $\theta: T_{(\gamma, \gamma')} \G(D)^{(2)}\to \d$.
 One thus has
\[
(m^*\tau_\sD)|_{(\gamma, \gamma')}
=\langle \pr _1^* \theta ^\lt _{\sGo} \otimes (\pr_1 ^* \theta ^\rt _{\sD}+\Ad _{d} \pr_2^*\theta ^\rt_{\sD})\rangle
 +\langle ( \Ad _{k'} \pr_1^*\theta ^\rt_{\sG^*}+\pr_2 ^* \theta ^\rt _{\sG^*} )\otimes \pr_2 ^* \theta ^\rt _{\sGt} \rangle.
\]
Setting $Q = (m^*\tau_\sD  -\pr _1^*\tau_\sD -\pr _2^*\tau_\sD)|_{(\gamma, \gamma')}$, one then has
\begin{align}\nonumber
Q & =
%m^*\tau_\sD  -\pr _1^*\tau_\sD -\pr _2^*\tau_\sD
 \langle \pr _1^* \theta ^\lt _{\sGo} \otimes \Ad _{d} \pr_2^*\theta ^\rt _{\sD}\rangle
 +\langle \Ad _{k'} \pr_1^*\theta ^\rt _{\sG^*} \otimes \pr_2 ^* \theta ^\rt _{\sGt} \rangle -
  \langle  \pr_1^*\theta ^\rt _{\sG^*} \otimes \pr_1 ^* \theta ^\rt _{\sGt} \rangle
 -  \langle  \pr_2^*\theta ^\lt _{\sGo}  \otimes \pr_2 ^* \theta ^\rt _{\sD} \rangle,\\
\label{eq:Q}
& = \langle ({\rm Ad}_{d^{-1}} \pr_1^* \theta^\lt_{\sGo} - \pr_2^* \theta^\lt_{\sGo}) \otimes \pr_2^*\theta ^\rt_{\sD}\rangle
+ \langle \pr_1^* \theta^\rt_{\sG^*} \otimes ({\rm Ad}_{(k')^{-1}} \pr_2^* \theta^\rt_{\sGt} - \pr_1^* \theta^{\rt}_{\sGt})\rangle.
\end{align}
Now the composability condition gives $\pr_1 ^* \theta ^\lt _{\sGt}=\pr_2 ^* \theta ^\lt _{\sGo}$ and $\pr_1 ^* \theta ^\rt _{\sGt}=\pr_2 ^* \theta ^\rt _{\sGo}$,
so
\begin{align}\label{eq:Ad-1}
&{\rm Ad}_{d^{-1}} \pr_1^* \theta^\lt_{\sGo} - \pr_2^* \theta^\lt_{\sGo} = {\rm Ad}_{d^{-1}} \pr_1^* \theta^\lt_{\sGo} - \pr_1^* \theta^\lt_{\sGt}
= \pr_1^*({\rm Ad}_{d^{-1}} \theta^\lt_{\sGo} - \theta^\lt_{\sGt}),\\
\label{eq:Ad-2}
&{\rm Ad}_{(k')^{-1}} \pr_2^* \theta^\rt_{\sGt} - \pr_1^* \theta^{\rt}_{\sGt} = {\rm Ad}_{(k')^{-1}} \pr_2^* \theta^\rt_{\sGt} - \pr_2^* \theta^{\rt}_{\sGo}
= \pr_2^*({\rm Ad}_{(k')^{-1}} \theta^\rt_{\sGt} -  \theta^{\rt}_{\sGo}).
\end{align}
On the other hand, \eqref{eq:v-eta-0} gives
\begin{align}\label{eq:Ad-3}
&{\rm Ad}_{d^{-1}} \theta^\lt_{\sGo} - \theta^\lt_{\sGt}
=  -\theta^\lt_\sD - {\rm Ad}_{g_2^{-1}} \theta_{\sG^*}^\rt: \; T_\gamma \G(D) \to \d,\\
\label{eq:Ad-4}&{\rm Ad}_{(k')^{-1}} \theta^\rt_{\sGt} -  \theta^{\rt}_{\sGo}
=  {\rm Ad}_{g_2} \theta^\rt_\sD + \theta^\lt_{\sG^*}: \;
T_{\gamma'} \G(D) \to \d.
\end{align}
Combining \eqref{eq:Ad-1} - \eqref{eq:Ad-4} to substitute into \eqref{eq:Q} and further using the ${\rm Ad}$-invariance and the fact that $\g^*$ is lagrangian,
one now has
\begin{align*}
Q
&= -\langle \pr_1^* \theta_\sD^\lt \otimes \pr_2^* \theta_\sD^\rt \rangle + \langle \pr_1^* \theta_{\sG^*}^\rt \otimes \pr_2^*
\theta_{\sG^*}^\lt  \rangle  = -\langle \pr_1^* \theta_\sD^\lt \otimes \pr_2^* \theta_\sD^\rt \rangle.
\end{align*}
%arrives at \eqref{eq:tauD-multi}.
\end{proof}

We now compute $\dif \omega_\sD \in \Omega^3(\G(D))$. Let
 $\langle \theta_\sD^\rt, [\theta_\sD^\rt, \theta_\sD^\rt]\rangle \in \Omega^3(\G(D))$
be given by
\[
\langle \theta_\sD^\rt, [\theta_\sD^\rt, \theta_\sD^\rt]\rangle(x_1, x_2,x_3) =
\langle \theta_\sD^\rt(x_1), [\theta_\sD^\rt(x_2), \theta_\sD^\rt(x_3)]\rangle, \hs x_1, x_2, x_3 \in {\mathfrak{X}}^1(\G(D)).
\]
Note that the ad-invariance of $\lara$ implies that $\langle \theta_\sD^\rt, [\theta_\sD^\rt, \theta_\sD^\rt]\rangle$ is indeed skew-symmetric.

\begin{lem}\label{lem:domegaD}
One has $\dif \omega_\sD = -\frac{1}{2}\langle \theta_\sD^\rt, [\theta_\sD^\rt, \theta_\sD^\rt]\rangle  \in \Omega^3(\G(D))$.
\end{lem}

\begin{proof}
By our convention at the end of $\S$\ref{s:intro},
%the left and the right invariant Maurer-Cartan forms $\theta^\rt_{\scriptscriptstyle{A}}$
%and $\theta^\lt_{\scriptscriptstyle{A}}$ of any Lie group $A$ satisfy
%$d\theta^\rt_{\scriptscriptstyle{A}} + [\theta^\rt_{\scriptscriptstyle{A}}, \, \theta^\rt_{\scriptscriptstyle{A}}] = 0$ and
%$d\theta^\lt_{\scriptscriptstyle{A}} - [\theta^\lt_{\scriptscriptstyle{A}}, \, \theta^\lt_{\scriptscriptstyle{A}}] = 0$, where
%$[\theta^\rt_\sA, \theta^\rt_\sA](x_1, x_2) = [\theta^\rt_\sA(x_1), \theta^\rt_\sA(x_2)]$ for $x_1, x_2 \in {\mathfrak{X}}^1(A)$.
one has
\[
2\dif \omega_\sD = \langle [\theta^\lt_{\sGo}, \theta^\lt_{\sGo}] \wedge \theta^\rt_{\sD}\rangle
+ \langle \theta^\lt_{\sGo} \wedge [\theta^\rt_{\sD}, \theta^\rt_{\sD}]\rangle
-\langle [\theta^\rt_{\sG^*}, \theta^\rt_{\sG^*}] \wedge \theta^\rt_{\sGt}\rangle
+ \langle \theta^\rt_{\sG^*} \wedge [\theta^\rt_{\sGt}, \theta^\rt_{\sGt}]\rangle.
\]
Fix $\gamma = (k, g_1, g_2, d) \in \G(D)$ and $x_1, x_2, x_3 \in T_\gamma \G(D)$. Let
$\Delta =2\dif \omega_\sD(x_1, x_2, x_3)$. Then
\begin{align*}
\Delta
= &\langle [\theta^\lt_{\sGo}(x_1), \, \theta^\lt_{\sGo}(x_2)], \,
\theta^\rt_{\sD}(x_3)\rangle
+ \langle \theta^\lt_{\sGo}(x_1), \,
[\theta^\rt_{\sD}(x_2), \theta^\rt_{\sD}(x_3)]\rangle \\
&-\langle [\theta^\rt_{\sG^*}(x_1), \, \theta^\rt_{\sG^*}(x_2)], \,
\theta^\rt_{\sGt}(x_3)\rangle
+\langle \theta^\rt_{\sG^*}(x_1), \,
[\theta^\rt_{\sGt}(x_2), \, \theta^\rt_{\sGt}(x_3)]\rangle + \mbox{c.p.}.
\end{align*}
Here and in what follows, for  any expression
$P(y_1, y_2, y_3)$ we set
\[
P(y_1, y_2, y_3) + \mbox{c.p.} = P(y_1, y_2, y_3) + P(y_2, y_3, y_1) + P(y_3, y_1, y_2).
\]
Using the following $\d$-valued linear maps
 \[
\theta^\rt_{\sGo} ={\rm Ad}_{\bar{g}_1} \circ \theta^\lt_{\sGo}, \hs \theta^\lt_{\sG^*}= {\rm Ad}_{\bar{k}^{-1}}\circ \theta^\rt_{\sG^*},\hs
\alpha = {\rm Ad}_{\bar{g}_1} \circ \theta_\sD^\rt, \hs \beta = {\rm Ad}_{\bar{k}^{-1}} \circ \theta^\rt_{\sGt}
\]
on $T_\gamma \G(D)$ and the  ${\rm Ad}$-invariance of $\lara$,  one has
\begin{align*}
\Delta & = \langle [\theta_{\sGo}^\rt(x_1),  \theta_{\sGo}^\rt(x_2)], \alpha(x_3)\rangle + \langle \theta_{\sGo}^\rt(x_1), [\alpha(x_2), \alpha(x_3)] \rangle\\
& \; \;\;-\langle [\theta_{\sG^*}^{\lt}(x_1), \theta_{\sG^*}^{\lt}(x_2)], \beta(x_3)\rangle +
\langle \theta^\lt_{\sG^*}(x_1), \,
[\beta(x_2), \, \beta(x_3)]\rangle + \mbox{c.p.}\\
& = \langle [\theta_{\sGo}^\rt(x_1), \, \theta_{\sGo}^\rt(x_2) + \alpha(x_2)], \,  \alpha(x_3)]\rangle
+ \langle [\theta_{\sG^*}^{\lt}(x_1), \, -\theta_{\sG^*}^{\lt}(x_2)+ \beta(x_2)], \, \beta(x_3)]\rangle + \mbox{c.p.}.
\end{align*}
Note now that
$\theta^\rt_{\sGo} + \alpha = -\theta^\lt_{\sG^*} + \beta$ by \eqref{eq:v-eta-0}. For $x \in T_\gamma(\G(D))$, write
\[
\alpha_\g(x) = \pr_\g \alpha(x), \hs \alpha_{\g^*}(x) = \pr_{\g^*} \alpha(x), \hs \beta_\g(x) = \pr_\g \beta(x), \hs \beta_{\g^*}(x) = \pr_{\g^*} \beta(x).
\]
Then $\theta_{\sGo}^\rt(x) = \beta_\g(x)-\alpha_{\g}(x)$ and $\theta_{\sG^*}^{\lt}(x) = \beta_{\g^*}(x)-\alpha_{\g^*}(x)$ for $x \in T_\gamma \G(D)$. Thus
\begin{align*}
\Delta & = \langle [\beta_\g(x_1)-\alpha_{\g}(x_1), \;\beta_\g(x_2)+\alpha_{\g^*}(x_2)], \;\alpha(x_3)\rangle \\
&\; + \langle [\beta_{\g^*}(x_1)-\alpha_{\g^*}(x_1), \; \beta_\g(x_2)+\alpha_{\g^*}(x_2)], \;\beta(x_3)\rangle + {\rm c.p.}.
\end{align*}
Using the definition and the ad-invariance of $\lara$ to expand terms,  one has
\[
\Delta = -\langle \alpha(x_1), [\alpha(x_2), \alpha(x_3)]\rangle + \langle \beta(x_1), [\beta(x_2), \beta(x_3)]\rangle.
\]
As $\g \subset \d$ is lagrangian, one has $\langle \beta(x_1), [\beta(x_2), \beta(x_3)]\rangle = 0$. Thus
\[
\Delta  = -\langle \alpha(x_1), [\alpha(x_2), \alpha(x_3)]\rangle = -\langle \theta_\sD^\rt(x_1),\, [\theta_\sD^\rt(x_2), \,
\theta_\sD^\rt(x_3)]\rangle .
\]
\end{proof}

\subsubsection{Proof of Theorem \ref{thm:presymp}}\label{ss:proof-presymp}
As $\l$ is a lagrangian Lie sub-algebra of $(\d, \lara)$, the $2$-form
$\omega_\sL$ on $\G(L)$ coincides with the pullback of $\tau_\sD$ to $\G(L)$. Using again the fact that $\l$ is lagrangian,
it now follows immediately from
Lemma \ref{lem:tauD-1} and  Lemma \ref{lem:domegaD} that $\omega_\sL$ is both multiplicative and closed.
By Lemma \ref{lem:delta-tauD}, the image of ${\rm Lie}(\G(L)) \subset A_\sD$ under $\delta_{\tau_\sD}: A_\sD \to \ttt G$ is precisely
${\bf e}(\l_\sG)$. By Definition \ref{defn:pre}, $(\G(L), \omega_\sL)$ is an integration of the Dirac structure
${\bf e}(\l_\sG)$ on $G$.
This finishes the proof of Theorem \ref{thm:presymp}.

\begin{rem}\label{rem:tauD}
{\rm
As seen in the proof of Theorem \ref{thm:presymp}, the properties of the tensor field $\tau_\sD$ on $\G(D)$ ensure {\it simultaneous integrations}
of all Dirac structures on $G$ coming from lagrangian Lie sub-algebras of $(\d, \lara)$.
It  would be interesting to understand in what sense the pair $(\G(D), \tau_\sD)$ is an {\it integration} of the action Courant algebroid $\d \times G$,
which is isomorphic to the Courant algebroid $\ttt G$.
\hfill $\diamond$
}
\end{rem}

%%%%%%%%%%%%%%%%%%%%%%%%%%%%%%%%%%%%%%%%%%%%%%%%%%
\subsection{Action of a double symplectic groupoid}\label{ss:doublesg}
By choosing  $\l=\g^* \subset \d$ and setting $L=G^*$,
Theorem~\ref{thm:presymp} defines a symplectic groupoid $(\G(G^*),\omega)$, where $\omega :=\omega_{\sG^*}$,
 integrating the Poisson Lie group $(G,\pi_\sG)$.  As  proved in \cite{LuW0}
(see also $\S$\ref{ss:LA-vee}),
$\G(G^*)$ carries an additional Lie groupoid structure over $G^*$, with respect to which $(\G(G^*),\omega)$ is an integration
of the Poisson Lie group $(G^*,-\pi_{\sG^*})$; these two groupoid structures on $\G(G^*)$ make $(\G(G^*),\omega)$ into
a {\em double symplectic groupoid} \cite{Mac0}. Explicitly (compare with $\S$\ref{ss:LA-vee}), the groupoid structure over $G^*$ has
source and target
$(k_1, g_1, g_2, k_2) \mapsto k_2$ and $(k_1, g_1, g_2, k_2) \mapsto k_1^{-1}$,
and multiplication
\[
(k_1, g_1, g_2, k_2)(k_2^{-1},  g_1^\prime, g_2^\prime, k_2^\prime)
 = (k_1, g_1 g_1^\prime, g_2g_2^\prime, k_2^\prime).
\]

By $\S$\ref{ss:LA-vee}, the Lie groupoids $\G(L)$ in \eqref{eq:GL} carry a natural (left) action of $\G(G^*)\rightrightarrows G^*$
along the map $J: \G(L)\to G^*$, $J(k,g_1,g_2,l)=k^{-1}$, via
$\kappa : \G(G^*)\times_{G^*} \G(L)\to \G(L)$,
\begin{equation}\label{eq:actionGL}
%\kappa : \G(G^*)\times_{G^*} \G(L)\to \G(L),\;\;
\kappa((k_1, g_1, g_2 ,k_2),(k_2^{-1}, g_1',g_2',l)) = (k_1, g_1 g_1', g_2 g_2',l),
\end{equation}
and $\kappa$ is an action of the double Lie groupoid $\G(G^*)$ on the Lie groupoid $\G(L)$, in the sense of \cite[Definition 1.5]{BrownMackenzie}.

Note also that $J: (\G(L),\omega) \to (G^*,-\pi_{\sG^*})$  is a Dirac map: indeed, for $\alpha =w^\lt \in T_{k^{-1}}^*G^*$,
$w\in \g$, the horizontal tangent vector $\widetilde{w}_\gamma$ at $\gamma=(k,g_1,g_2,l) \in \G(L)$ satisfies
\[
(\dif {J})(\widetilde{w}_\gamma )=-\pi _{\sG^*}^\sharp (\alpha )|_{k^{-1}},\qquad i_{\widetilde{w}_\sigma }\omega_\sL = {J}^*\alpha.
\]
By \cite[Lemma~4.8]{BuCr}, if the source fibers of $\G(G^*) \rightrightarrows G^*$ are connected, then $J: (\G(L),\omega) \to (G^*,-\pi_{\sG^*})$
being a Dirac map ensures the identity
\begin{equation}\label{eq:omegacond}
\kappa^* \omega_\sL = {\rm p}_1^*\omega + {\rm p}_2^*\omega_\sL \in \Omega^2(\G(G^*)\times_{G^*} \G(L)),
\end{equation}
where ${\rm p}_1$, ${\rm p}_2$ are the projections of $\G(G^*)\times_{G^*}\G(L)$ onto each factor.
By a calculation similar to that for the proof of Lemma \ref{lem:tauD-1}, which we omit, one can show that \eqref{eq:omegacond} holds without
the connectedness assumption on the source fibers.

%In turn, the fact that $J$ is a Dirac map is the infinitesimal counterpart of the following property of the action $\kappa$ (see ):

%and $\omega_{\scriptscriptstyle{\G(G^*)}}$ and $\omega$ are the corresponding 2-forms given by Theorem~\ref{thm:presymp}. (The map $J$ being a Dirac %map ensures that property \eqref{eq:omegacond} holds for the action of the source-connected sub-groupoid of $\G(G^*)$, but one can verify that it actually holds %for $\G(G^*)$.)

%\section{Integration of Poisson homogeneous spaces}\label{s:homog-GH}

\subsection{Explicit integrations of Poisson homogeneous spaces}\label{ss:explicit-GH}
Assume that $(G, \pi_\sG)$ is a Poisson Lie group admitting a Drinfeld double $(D, \phi_\sG)$, and let $(G/H, \pi)$ be a
Poisson homogeneous space of $(G, \pi_\sG)$.
Let $\l = \l_\pi$ be the Drinfeld lagrangian Lie
sub-algebra $\l$ given in \eqref{eq:lpie}, and let ${\bf e}(\l_\sG)$ be the pullback Dirac structure of $\pi$ on $G$.
Let
$(\G(L), \omega_\sL)$ be a pre-symplectic groupoid over $G$ integrating
 the Dirac structure ${\bf e}(\l_\sG)$ as  in Theorem \ref{thm:presymp}.
Assume $\phi_{\sH, \sL}: H \to L$ is an integration of
$\h \hookrightarrow \l$ satisfying
\begin{equation}\label{eq:HL}
\phi_\sL \circ \phi_{\sH, \sL} = \phi_\sG|_\sH: \;\; H \to D.
\end{equation}
Let $H \times H$ act on $\G(L)$ by
\begin{equation}\label{eq:HxH:action}
(h_1,h_2)\cdot (v,g_1,g_2,l)= (v,\,  g_1h^{-1}_1,\, g_2h_2^{-1},\, \phi_{\sH, \sL}(h_1) \,l\,\phi_{\sH, \sL}(h_2)^{-1}).
\end{equation}
By Theorem \ref{th:inte-BHl}, one has the quotient groupoid $\overline{\G(L)} =\G(L)/(H \times H)\rightrightarrows G/H$ integrating the
 Lie algebroid $T^*_\pi(G/H)$. Let $p: \G(L) \to\overline{\G(L)}$ be the projection.

\begin{them}\label{thm:GH}
Assuming \eqref{eq:HL}, the pre-symplectic groupoid $(\G(L), \omega_\sL)$ is an $H$-admissible integration of the Dirac structure
${\bf e}(\l_\sG)$ on $G$
(see Definition \ref{defn:admi1}).
Consequently, there is a unique symplectic structure $\overline{\omega}$ on $\overline{\G(L)}$ such that
$p^* \overline{\omega} = \omega_\sL$, and  $(\overline{\G(L)}, \overline{\omega})$ is a symplectic groupoid
integrating the Poisson homogeneous space $(G/H, \pi)$.
\end{them}

\begin{pf} By Theorem \ref{th:inte-EF}, it remains to check the $(H \times H)$-invariance of $\omega_\sL$.
Let  $h=(h_1,h_2)\in H\times H$ and we write $h_i$ for $\phi_{\sH, \sL}(h_i)$, $i = 1, 2$,  for notational simplicity.
Let  $\Phi_h: \G(L) \to \G(L)$ be given by
$$
\Phi_h(k,g_1,g_2,l) = (k, g_1h^{-1}_1,g_2h_2^{-1},{h_1}\,l\,{h_2}^{-1}).
$$
With the notation as in \eqref{eq:severaform}, one checks directly that
\[
\Phi_h^*\theta^\lt_{\sGo}= \Ad_{h_1}\theta^\lt_{\sGo}, \hs
\Phi_h^*\theta^\rt_{\sL}= \Ad_{h_1}\theta^\rt_{\sL}, \hs
\Phi_h^*\theta^\rt_{\sG^*} = \theta^\rt_{\sG^*}, \hs  \Phi_h^*\theta^\rt_{\sGt}= \theta^\rt_{\sGt}.
\]
Hence $\Phi_h^*\omega_\sL=\omega_\sL$ by the $\Ad$-invariance of $\lara$.
\end{pf}

%\begin{rem}\label{rem:HL}
%{\rm
%In applications, one typically has $G \subset D$ and $L \subset D$ as subgroups of $D$ and $H \subset G \cap L$,
%so the additional assumption
%\eqref{eq:HL} in Theorem \ref{thm:GH} holds automatically. Note also that if $H$ is connected, then
 %$H \subset L$ is also automatic, as $H$ must be contained in the connected component of $L$
%through the identity element.
%\hfill $\diamond$
%}
%\end{rem}

\smallskip
In the context of Theorem \ref{thm:GH},
just as for $\G(L)$ discussed in $\S$\ref{ss:doublesg} and by $\S$\ref{ss:LA-vee},
%Proposition \ref{pr:double-act},
the groupoid morphism $J: \G(L)\to G^*$, $(k,g_1,g_2,l)\mapsto k^{-1}$ is $(H\times H)$-invariant
and gives rise to $\overline{J}: \overline{\G(L)}\to G^*$, and the action $\kappa$ in \eqref{eq:actionGL} descends to an action
$$
\overline{\kappa}: \G(G^*)\times_{G^*} \overline{\G(L)}\to \overline{\G(L)}
$$
of the double Lie groupoid $\G(G^*)$ covering the action map $G\times (G/H)\to G/H$.
The fact that $J$ is a Dirac map  implies that $\overline{J}: (\overline{\G(L)}, \overline{\omega}) \to (G^*,-\pi_{\sG^*})$ is Poisson. By
\eqref{eq:omegacond}, $\overline{\kappa}^*\overline{\omega} =
{\rm p}_1^*\omega_{\scriptscriptstyle{\G(G^*)}} + {\rm p}_2^*\overline{\omega}$, so
$\kappa$ is a {\em symplectic action} in the sense of \cite[Definition 3.7]{MikWein}.

%%%%%%%%%%%%%%%%%%%%%%%%%%%%%%%%%%%%%%%%%%%%%%%%%%%%%%%%%%%
\subsection{The holomorphic setting}\label{s:hol}
%\addtocontents{toc}{\protect\setcounter{tocdepth}{1}}
The concepts of Lie groupoids, Lie algebroids, and integrability of Lie algebroids and Poisson structures generalize straightforwardly to the holomorphic category, and we refer to \cite{GSX:hol-Poi,GSX:groupoids} for details.
By \cite[Lemma 2.3]{GSX:hol-Poi}, the real and imaginary parts $\pi_\sR$, $\pi_{\sI} \in \Gamma(\wedge^2 TM)$
of  a holomorphic Poisson structure
\[
\pi = \pi_{\sR} + i \pi_{\sI} \in \Gamma(\wedge^2 T^{1, 0}M)
\]
 on a complex manifold $M$
are real Poisson structures on $M$ regarded as a real manifold.
Moreover,  if $\phi: (M, \pi) \to (M^\prime, \pi^\prime)$ is a Poisson map of complex Poisson manifolds,
then $\phi: (M, \pi_\sR) \to (M^\prime, \pi_{\sR}^\prime)$ is a Poisson map of real Poisson manifolds.
It is shown in \cite[Theorem 3.22]{GSX:groupoids} that
$\pi$ is integrable as a holomorphic Poisson structure
if and only if $\pi_\sR$ (or equivalently $\pi_\sI$) is integrable as a real Poisson structure.

\begin{them}\label{thm:main-cplx}
Every complex Poisson homogeneous space $(G/H, \pi)$ of any complex Poisson Lie group $(G, \pi_\sG)$ is integrable.
\end{them}

\begin{proof}
Let $\pi_{\sG, \sR}$ and $\pi_\sR$ be respectively the real parts of $\pi_{\sG}$ and  $\pi$. Then $(G, \pi_{\sG, \sR})$ is a
real Poisson Lie group, and $(G/H, \pi_\sR)$ is a (real) Poisson homogeneous space of
$(G, \pi_{\sG, \sR})$. By Theorem \ref{th:main-intro}, $\pi_{\sR}$ is integrable as a real Poisson structure on $G/H$.
By \cite[Theorem 3.22]{GSX:groupoids},
$\pi$ is integrable as a holomorphic Poisson structure on $G/H$.
\end{proof}

 To describe explicit integrations, define a {\it holomorphic Dirac structure} on a complex manifold $M$
as a holomorphic involutive lagrangian sub-bundle $E$ of the
holomorphic Courant algebroid $\ttt^{1,0} M:=T^{1, 0}M\oplus (T^{1, 0}M)^*$.
An {\em integration} of a holomorphic Dirac structure $E$ is a holomorphic pre-symplectic groupoid $(\G,\omega)$ over $M$, integrating
$E$ as a holomorphic Lie algebroid, such that the target map is a Dirac map (see \cite{BDN}).
With the obvious holomorphic analog of Definition \ref{defn:admi1} for $H$-admissibility of pre-symplectic groupoids,
Theorem \ref{th:inte-EF} also holds in the holomorphic setting.

A complex Poisson Lie group $(G, \pi_\sG)$ gives rise to a  complex Lie bialgebra $(\g, \delta_*)$; its Drinfeld
double Lie algebra $(\d, \lara_\mathbb{C})$ is now a complex quadratic Lie algebra.
We have a natural holomorphic analog of the isomorphism ${\bf e}$ in \eqref{eq:I}, which defines,
for each complex lagrangian Lie sub-algebra $\l$ of $(\d, \lara_\mathbb{C})$, a holomorphic (affine) Dirac structure ${\bf e}(\l_\sG)$ on $G$.
Assuming that $(G, \pi_\sG)$ admits a Drinfeld  double $(D, \phi_\sG)$ (defined by the natural holomorphic analog of Definition \ref{defn:D-double}),
the groupoid $\G(L)$ over $G$ given in \eqref{eq:GL} is holomorphic,
and the same formula \eqref{eq:severaform} defines a holomorphic $2$-form
\[
\omega_\sL = \frac{1}{2}\langle{\theta^\lt_{\sGo} \wedge \theta^\rt_{\sL}}\rangle_\mathbb{C} + \frac{1}{2}\langle{\theta^\rt_{\sG^*} \wedge\theta^\rt_{\sGt}}\rangle _\mathbb{C}
\]
on $\G(L)$.
Verbatim arguments in
the proofs of Theorem \ref{thm:presymp}  and
Theorem \ref{thm:GH} show that they also hold in the holomorphic category. We summarize the results as follows:

\begin{them}\label{thm:hol}
(1) For any complex lagrangian Lie sub-algebra $\l$ of $(\d, \lara_\mathbb{C})$, $(\G(L), \omega_\sL)$ is a holomorphic pre-symplectic
groupoid integrating the holomorphic $(G, \pi_\sG)$-affine
Dirac structure ${\bf e}(\l_\sG)$ on $G$;

(2) Let $(G/H, \pi)$ be a complex Poisson homogeneous space of $(G, \pi_\sG)$, and let $\l$ be the associated complex Drinfeld lagrangian Lie sub-algebra.
If there exists a holomorphic integration $\phi_{\sH, \sL}: H \to L$ of the Lie algebra inclusion $\h \to \l$
such that \eqref{eq:HL} holds, then one has the quotient holomorphic symplectic groupoid $(\G(L)/(H \times H), \overline{\omega})$,
defined the same way as in Theorem \ref{thm:GH}, which integrates the complex Poisson manifold $(G/H, \pi)$.
\end{them}

%%%%%%%%%%%%%%%%%%%%%%%%%%%%%%%%%%%%%%%%%%%%%%%%%%%%%%%%%%%%%%

\section{Examples: old and new}\label{s:examples}
%\addtocontents{toc}{\protect\setcounter{tocdepth}{2}}

%%%%%%%%%%%%%%%%%%%%%%%%%%%%%%%%%%%%%%%%

\subsection{Special cases and previous examples}\label{ss:previous}
In this section, we consider various particular cases of Poisson homogeneous spaces,  comparing our construction of their integrations with results in the literature.

\subsubsection{A case of global $L$-action}\label{ss:complete}
We will compare our integration with the one in \cite{Lu-note}.
Assume that $(G, \pi_\sG)$ admits a connected Drinfeld double $D$ such that $G \subset D$ is a closed subgroup.
Let
$(G/H, \pi)$ be a Poisson homogeneous space of $(G, \pi_\sG)$ with Drinfeld lagrangian Lie sub-algebra $\l$. Let $L$
be the connected Lie subgroup of $D$ with Lie algebra $\l$, and suppose that the dressing
action of $\l$ on $G$ integrates to a left action $(l,g)\mapsto {}^l\!g$ of $L$ on $G$.
Assume furthermore that $H \subset L$ is a closed subgroup and that
$^h \kern-2pt g=gh^{-1}$ for $h\in H$. Let  $(H \times H)$ act on $G \times L$ by
$(h_1, h_2)\cdot (g, \, l)  = (gh_1^{-1}, \, h_2lh_1^{-1})$.
It is shown in \cite{Lu-note} that the quotient $G \times_H (L/H)$ is a symplectic groupoid\footnote{It is actually verified in \cite{Lu-note}  that $G\times_H (L/H)$ is a Poisson groupoid that is symplectic near the identity section; but then \cite[Theorem 5.3]{Mac-Xu} implies that it is a symplectic groupoid.} of $(G/H, \pi)$,
where the groupoid structure on $G \times_H (L/H)$ is the unique one such that the projection $G \times L \to G \times_H (L/H)$ from the action groupoid
$G \times L$ is a groupoid morphism.

Let $G^*$ be the connected subgroup of $D$ with Lie algebra $\g^*$, and recall from \cite[Lemma 5.15]{Lu-note}
that $(^l g)l{g}^{-1}\in G^*$ for any $g\in G$ and $l\in L$.
It follows that
\begin{equation}\label{eq:F}
F:\;\; G\times L\to \G(L), \;\; (g, l)\mapsto (gl^{-1}(^lg)^{-1},\, ^l\kern-2pt g,g,l),
\end{equation}
is $(H \times H)$-equivariant, inducing a local diffeomorphism of groupoids $\overline{F}$ from $G\times_H (L/H)$ to $\G(L)/(H\times H)$.
The pullback by $\overline{F}$ of the symplectic form on $\G(L)/(H\times H)$ agrees with the one on $G\times_H (L/H)$, since they integrate the same Poisson structure and
$G\times_H (L/H)$ is source-connected. Thus $\overline{F}$ is a local diffeomorphism of symplectic groupoids.

When $G^*\times G\to D$, $(v, g) \mapsto \bar{v}\bar{g}$ is a diffeomorphism, the map $F$ in \eqref{eq:F} is a groupoid isomorphism
from $L\ltimes G$ to $\G(L)$, and $\overline{F}$ is an isomorphism of symplectic groupoids.
%A condition ensuring  the completeness of the $\l$-action on $G$ is that the map $G^*\times G\to D$, $(v, g) \mapsto \bar{v}\bar{g}$ be a diffeomorphism.
%In this case the map $F$ in \eqref{eq:F} is a groupoid isomorphism
%from $L\ltimes G$ to $\G(L)$, which induces an isomorphism of symplectic groupoids from $G\times_H (L/H)$  to $\G(L)/(H\times H)$.

%
%{\color{red} from referee: really need $D$ in this discussion? One could ask about integration of this
%action to the group $L$ directly.}

\subsubsection{Relatively complete Poisson homogeneous spaces}\label{ss:relative}
We now consider the construction in \cite{BCST}.
Let $(G/H, \pi)$ be a Poisson homogeneous space of a Poisson Lie group $(G, \pi_\sG)$ satisfying $\pi|_{eH}=0$.
Then $\textrm{Ann}(\h) = \{\xi \in \g^*\,|\,\xi|_\h = 0\}$ is a Lie sub-algebra of $\g^*$, and the corresponding Drinfeld
lagrangian Lie sub-algebra is $\l = \h + \mathrm{Ann}(\h)$ as a vector space.
Let $G^*$ be the connected Lie subgroup of $D$ with Lie algebra $\g^*$, and
let $H^\perp$ be connected Lie subgroup of $G^*$ with Lie algebra $\textrm{Ann}(\h)$.
 Following \cite{BCST}, $(G/H, \pi)$ is said to be {\it relatively complete} if the group multiplication of $D$ gives a diffeomorphism
$H \times H^\perp \to L$, where $L$ is the connected subgroup of $D$ with Lie algebra $\l$.

Assume that $(G/H, \pi)$ is relatively complete.
Consider the map $J\colon \G(G^*)\to G^*$,  $J(k_1,g_1,g_2,k_2)=k_1^{-1}$, as defined in $\S$\ref{ss:doublesg}.
Since $H^\perp$ is a coisotropic subgroup of $G^*$, using that $J$ is a groupoid morphism from $\G(G^*)\rightrightarrows G$ to $G^*$
and a Poisson map into $(G^*,-\pi _{G^*})$, we see that $J^{-1}(H^\perp)$ is a coisotropic sub-groupoid of $\G(G^*)\rightrightarrows G$.
It is proven in \cite[Theorem~15]{BCST} that the coisotropic reduction of $J^{-1}(H^\perp)$ agrees with the quotient by the
free and proper action of $H$ on  $J^{-1}(H^\perp)$ by
\[
h \cdot  (k_1,g_1,g_2,k_2)= (k_1,g_1 ((h^{-1})^{k_2}),g_2h^{-1},(^{h^{-1}} \kern-2pt k_2)),
\]
where for
$h\in H$ and $v\in H^\perp$, $^h\kern-1pt v \in H^\perp$ and $h^v\in H$ are the unique elements such that
$hv=(^h\kern-1pt v )(h^v)\in L$.
By \cite{BCST}, the quotient $J^{-1}(H^\perp)/H$ is a symplectic groupoid integrating $(G/H,\pi)$.
On the other hand,  the assumptions for Theorem \ref{thm:GH} are all satisfied, so we have a symplectic
groupoid $(\G(L)/(H \times H), \overline{\omega})$ integrating $(G/H, \pi)$.
A direct calculation shows that we have the groupoid isomorphism
$J^{-1}(H^\perp)/H \to \G(L) / (H\times H )$  given by
$ (k_1,g_1,g_2,k_2)H \mapsto    (k_1,g_1,g_2,k_2) (H\times H)$.

%{\color{red} add word about weak relatively complete?}

%%%%%%%%%%%%%%%%%%%%%%%%%%%%%%%%%%%%%%%%%%%%%%%%%%

\subsubsection{Poisson homogeneous spaces from coisotropic sub-algebras}\label{subss:LBS}
Let $(\d, \lara)$ be a quadratic Lie algebra, and let $D$
be any Lie group  integrating $(\d, \lara)$. Let $Q$ be any closed subgroup of $D$ such that
the Lie algebra $\q$ of $Q$ is coisotropic with respect to $\lara$, i.e., $\q^\perp:=\{a \in \d\,|\, \langle a, \q\rangle = 0\}\subset \q$.
We first recall from \cite{LY:DQ} a construction of Poisson structures on $D/Q$ and the associated Poisson homogeneous spaces.

Assume that $\d = \g +\g'$ is a lagrangian splitting of $(\d, \lara)$. Identifying $\g' \cong \g^*$ via $\lara$, we write the splitting as $\d = \g + \g^*$.
Recall from $\S$\ref{ss:Drinfeld-double} that such a splitting gives rise to the element $R \in \wedge^2 \d$ and the multiplicative
Poisson structure $\pi_\sD = R^\rt - R^\lt$ on $D$.  Let $p_{\sD/\sQ}: D\to D/Q$ be the projection.
By \cite[Theorem 2.3]{LY:DQ}, $\pi :=  p_{\sD/\sQ}(R^\rt)$
 is a well-defined Poisson bivector field on $D/Q$. It also follows from the definition that
$(D/Q, \pi)$ is a Poisson homogeneous space of the Poisson Lie group $(D, \pi_\sD)$.

\begin{defn}\label{def:pi-DQ}
{\rm
We refer to $\pi = p_{\sD/\sQ}(R^\rt)$ as the Poisson structure on $D/Q$ defined by the lagrangian splitting
$\d = \g +\g^*$.
\hfill $\diamond$
}
\end{defn}

Recall from Lemma \ref{lem:G-D-Poi} and Remark \ref{re:Gs} that
any integrations $\phi_\sG: G \to D$ and $\phi_{\sG^*}: G^* \to D$ of the  respective
 Lie algebra inclusions $\g \to \d$ and $\g^* \to \d$ give rise to the pair $(G, \pi_\sG)$ and $(G^*, \pi_{\sG^*})$ of dual Poisson Lie groups,
where $\pi_\sG$ and $\pi_{\sG^*}$ are
respectively given in \eqref{eq:piG-bar} and \eqref{eq:piGs}, and that
\[
\phi_\sG: (G, \pi_\sG) \to
(D, \pi_\sD) \hs \mbox{and} \hs \phi_{\sG^*}: (G^*, \pi_{\sG^*}) \to (D, -\pi_\sD)
\]
 are Poisson Lie group homomorphisms.
Let $G$ and $G^*$ act on $D/Q$ through  $\phi_\sG$ and $\phi_{\sG^*}$.
The following fact from \cite{LY:DQ} will be the basis for most of the examples of Poisson homogeneous spaces
related to semi-simple Lie groups we will look at in $\S$\ref{ss:semisimple}.

\begin{prop}\label{prop:GG-DQ}\cite[Theorem 2.3 and Proposition 2.5]{LY:DQ}
Every $G$-orbit $\O$ (resp. $G^*$-orbit $\O'$) in $D/Q$ is a Poisson sub-manifold with
respect to $\pi$. Consequently, $(\O, \pi|_{\sO})$ (resp. $(\O^\prime, -\pi|_{\sO^\prime})$) is a Poisson homogeneous space of
$(G, \pi_\sG)$ (resp. of $(G^*, \pi_{\sG^*})$). Moreover, for $d \in D$, identifying the $G$-orbit through $dQ \in D/Q$ with $G/H$, where
$H = G \cap dQd^{-1}$, the associated Drinfeld lagrangian Lie sub-algebra is $(\Ad_d \q^\perp) + \g \cap \Ad_d \q \subset \d$.
\end{prop}

\begin{numex}\label{ex:DQ-Lag}
{\rm
Suppose that the Lie algebra $\q$ of $Q$ is lagrangian with respect to $\lara$, and let $\O$ be any $G$-orbit in $D/Q$. Pick any $d \in \O$
and identify $\O = G/H$, where $H = G \cap dQd^{-1}$. By Proposition \ref{prop:GG-DQ}, the Drinfeld lagrangian Lie sub-algebra associated to $(G/H, \pi|_\sO)$ is
$\l = \Ad_d \q$. Taking $L = dQd^{-1}$, which is a closed subgroup of $D$ with Lie algebra  $\l$ and contains $H$, and applying Theorem \ref{thm:GH}
(and Theorem \ref{thm:hol} if we are in the holomorphic category), we then obtain
a symplectic groupoid integrating the Poisson manifold $(\O, \pi|_\sO)$.  We will give more
concrete examples in $\S$\ref{ss:semisimple}.
\hfill $\diamond$
}
\end{numex}

Consider now $(D/Q, \pi)$ as a Poisson homogeneous space of the Poisson Lie group $(D, \pi_\sD)$. In \cite{LB-S}, D. Li-Bland and P. \v{S}evera
constructed a symplectic groupoid of $(D/Q, \pi)$ as the moduli spaces of flat connections over certain quilted surfaces. We now compare our construction
with the one in \cite{LB-S}.

Consider the quadratic Lie algebra $(\d^{(2)}, \lara^{(2)})$, where $\d^{(2)} = \d \oplus \d$, and
\[
\langle a_1 + a_2, \, b_1 + b_2 \rangle^{(2)} = \langle a_1, \, b_1\rangle - \langle a_2, \, b_2\rangle, \hspace{.2in} a_1, a_2, b_1, b_2 \in \d.
\]
Then $\d^{(2)} = \d_{\rm diag} + \d^\prime$ is a lagrangian splitting of  $(\d^{(2)}, \lara^{(2)})$, where $\d_{\rm diag} = \{(a, a)|a \in \d\}$
and $\d^\prime = \g^* \oplus \g$. Identifying $D \cong D_{\rm diag}$, then  $(\d^{(2)}, \lara^{(2)})$ is the
Drinfeld double Lie algebra of the Poisson Lie group $(D, \pi_\sD)$, and $D \times D$  is
a Drinfeld double of $(D, \pi_\sD)$. The Drinfeld lagrangian Lie sub-algebra
 for $(D/Q, \pi)$, as a Poisson homogeneous space of $(D, \pi_\sD)$, is
$\l = \{(a_1, a_2)\in \q \oplus \q| a_1-a_2 \in \q^\perp\} \subset \d^{(2)}$.
The invariance of $\lara$ implies that $\q^\perp$ is an ideal of  $\q$.  Let $Q^\perp$ be the connected Lie subgroup of $Q$
with Lie algebra
$\q^\perp$. Then $Q^\perp$ is also a normal subgroup of $Q$. Let
\[
L = \{(q_1, q_2)\in Q \times Q \, | \, q_1q_2^{-1} \in Q^\perp\} \subset D \times D.
\]
Then $L$ is a Lie subgroup of $D \times D$ with Lie algebra $\l$ and $Q \cong Q_{\rm diag} = \{(q, q)|q \in Q\} \subset L$. Applying Theorem
\ref{thm:GH} (or Theorem \ref{thm:hol} if we are in the holomorphic category), we arrive at a symplectic groupoid $(\overline{\G}, \bar{\omega})$
integrating the Poisson manifold $(D/Q, \pi)$. More precisely, the groupoid $\overline{\G}$ has the total space
\begin{align*}
\overline{\G} &= \{((k, g), (d_1, d_1), (d_2, d_2), (q_1, q_2)) \in (G^* \times G) \times D_{\rm diag} \times D_{\rm diag} \times (Q \times Q)|\\
&\hspace{.25in}q_1q_2^{-1} \in Q^\perp, \;  (kd_1, gd_1) = (d_2q_1^{-1}, d_2q_2^{-1})\}/(Q \times Q),
\end{align*}
where the quotient is by the $(Q \times Q)$-action
\begin{align*}
&(x_1, x_2) \cdot ((k, g), (d_1, d_1), (d_2, d_2), (q_1, q_2)) \\
&=  ((k, g), (d_1x_1^{-1}, d_1x_1^{-1}), (d_2x_2^{-1}, d_2x_2^{-1}), (x_1q_1x_2^{-1}, x_1q_2x_2^{-1})).
\end{align*}
Here we write $g = \phi_\sG(g) \in D$ and $k = \phi_{\sG^*}(k)\in D$ for $g \in G$ and $k \in K$.
On the other hand,
the symplectic groupoid from \cite[Example 4]{LB-S} has the total space
\[
\hat{\G} = \{(g, \, k, \, d, \, p) \in G \times G^* \times D \times Q^\perp|\, gkdpd^{-1} = e\}/Q
\]
with the $Q$-action given by $q\cdot (g,  k,  d,  p) = (g,  k,  dq^{-1}, qpq^{-1})$.
It is straightforward to check that one has the Lie groupoid isomorphism
$\overline{\G} \to
\hat{\G}$ given by
\[
[(k, g), (d_1, d_1), (d_2, d_2), (q_1, q_2)] \mapsto [g, \, k^{-1}, \, d_2, \, q_1^{-1}q_2].
\]

%\hen{Since this section 7 is about examples of integrations, say a bit more about integration here?}

%%%%%%%%%%%%%%%%%%%%%%%%%%%%%%%%%%%%%%%%%%%%%%%%%%%%%%%%%%%%%%%%%%%%%%%%%%%%%%%%%%%%%%%%%%
\subsection{Examples related to semi-simple Lie groups}\label{ss:semisimple}
In this section, we recall how real or complex semi-simple Lie groups can be made into Poisson Lie groups.
Our constructions in $\S$\ref{s:explicit} then give concrete examples of (real smooth or holomorphic) pre-symplectic or symplectic groupoids, most of which
have not been considered in the literature.

Throughout this section, we fix a connected complex semisimple Lie group $G$ with Lie algebra $\g$. Fix also a
Cartan sub-algebra $\h$ of $\g$ and a choice of positive roots for $(\g, \h)$, so that
$\g = \n^- + \h + \n$,
where $\n^-$ and $\n$ are the nilpotent Lie sub-algebras of $\g$ respectively generated by the negative and positive root subspaces.
Set $\b = \h+\n$ and $\b^- = \b + \n^-$, and let $B$ and $B^-$ be the Borel subgroups of $G$ with respective Lie algebras $\b$ and $\b^-$.
Let $H = B \cap B^-$, a maximal torus of $G$, and let $N = \exp(\n) \subset B$ and $N^- = \exp(\n^-) \subset B^-$.
If $\tau: X \to X$ is an involution on a set $X$, let $X^\tau = \{x \in X|\tau(x) = x\}$.

\subsubsection{Real semi-simple Poisson Lie groups}\label{subss:realG}
Let $\lara_\g$ be
the Killing form
of $\g$, and let $\lara_{\g, 0}$ be the imaginary part of $\lara_\g$. Regarding
$\g$ as a Lie algebra over $\mathbb{R}$, then
$(\g, \lara_{\g, 0})$ is a real quadratic Lie algebra. When $\g$ is simple, (real)
lagrangian Lie sub-algebras of $(\g, \lara_{\g, 0})$ have been classified in
\cite{Karo}, and the geometry of the variety $\calL(\g, \lara_{\g, 0})$
 of all lagrangian Lie sub-algebras of $(\g, \lara_{\g, 0})$, such as its irreducible components and its $G$-orbits through the adjoint
%\hen{one of the referee reports complained about capital ``A'' in ``Adjoint''}
action of $G$, have been
studied in \cite{e-l:real}.

In particular, any real form $\g_0$ of $\g$
is a lagrangian Lie sub-algebra of $(\g, \lara_{\g, 0})$ because $\lara_\g$ takes real values on $\g_0$.
As shown in \cite{FL:flag}, one has lagrangian splittings
\begin{equation}\label{eq:ggl0}
\g = \g_0 + \l_0,
\end{equation}
where
$\l_0 \subset \b$.  Taking any $L_0 \subset B$ with Lie algebra $\l_0$,
one gets a pair of dual Poisson Lie groups $(G_0, \pi_{\sG_0})$ and
$(L_0, \pi_{\sL_0})$ by Remark \ref{re:Gs}. Any choice of a lagrangian Lie sub-algebra $\l$ of $(\g, \lara_{\g, 0})$
then defines a $(G_0, \pi_{\sG_0})$-affine (resp. $(L_0, \pi_{\sL_0})$-affine)
Dirac structures on $G_0$ (resp. on $L_0)$, integrations of which can then be constructed via Theorem \ref{thm:presymp} using
integrations $\phi_\sL: L \to G$ of $\l \to \g$.
When $G_0 = K$ is a compact real form, one has the global Iwasawa decomposition
$G = KAN$, and the  pair $(K, \pi_\sK)$ and
$(AN, \pi_{{\sAN}})$ of Poisson Lie groups are studied in detail in \cite{anton:jdg, GW, LuW}.

Turning to examples of Poisson homogeneous spaces of $(G_0, \pi_{\sG_0})$, take first $G_1 = G^{\tau_1}$ to be
any real form of $G$. As the Lie algebra
$\g_1 = \g^{\tau_1}$ is a lagrangian Lie sub-algebra of $(\g, \lara_{\g, 0})$, the (real) symmetric space $G/G_1$ of $G$
then has a Poisson structure $\pi$ defined by the lagrangian splitting in \eqref{eq:ggl0} (see Definition \ref{def:pi-DQ}).
By Proposition \ref{prop:GG-DQ}, every $G_0$-orbit in $G/G_1$, equipped with the restriction of $\pi$,
is a Poisson homogeneous space of $(G_0, \pi_{\sG_0})$. In particular, the symmetric space $G_0/H_0$ of $G_0$, where $H_0 = G_0^{\tau_1} =
G_0 \cap G_1$, becomes a Poisson homogeneous space, and
the quotient construction in Theorem \ref{thm:GH} then gives a symplectic groupoid $(\overline{\G}, \overline{\omega})$
integrating the Poisson manifold $(G_0/H_0, \pi)$. When $G_0 $ is a compact real form of $G$, by $\S$\ref{ss:complete}  and \cite[Example 5.14]{Lu-note},
the manifold $\overline{\G}$
can be identified with $G_0\times_{H_0}(G_1/H_0)$,
 which is a fiber bundle over the compact
Riemannian symmetric space $G_0/H_0$ with fibers diffeomorphic to the non-compact Riemannian symmetric space $G_1/H_0$.
 See  \cite{FL:sym} for more details.

As a second class of examples of Poisson homogeneous spaces of the real semi-simple Poisson Lie group $(G_0, \pi_{\sG_0})$, consider the
flag variety $G/B$ regarded as a real compact manifold. Since  the Lie algebra $\b$ of $B$ is coisotropic with respect to
$\lara_{\g, 0}$, the lagrangian splitting $\g = \g_0 + \l_0$ in \eqref{eq:ggl0} again defines a Poisson structure $\pi$ on $G/B$ (Definition \ref{def:pi-DQ}),
and by Proposition
\ref{prop:GG-DQ}, each of the (finitely many) $G_0$-orbits in $G/B$ is a Poisson homogeneous space of $(G_0, \pi_{\sG_0})$.
 Using the explicit description of the Drinfeld lagrangian Lie sub-algebra associated to $\O$ given in
 the proof of \cite[Theorem  4.1]{FL:flag}, the construction in Theorem \ref{thm:GH} then gives a symplectic groupoid
integrating  the Poisson manifold $(\O, \pi)$ for every $G_0$-orbit $\O$ in $\B$.

\subsubsection{Complex semi-simple Poisson Lie groups}\label{subss:cplxG}
Consider now the complex  Lie algebra $(\d, \lara_\d)$, where $\d=\g \oplus \g$ is the direct product Lie algebra and
\[
\la x_1+y_1, \, x_2+y_2\ra_\d= \la x_1, \, x_2\ra_\g - \la y_1, \, y_2\ra_\g, \hspace{.2in} x_1, x_2, y_1, y_2 \in \g.
\]
Lagrangian Lie sub-algebras of $(\d, \lara_\d)$ were classified in \cite{Karo:cplx}, and the geometry of
the variety $\calL(\d, \lara_\d)$ of all lagrangian Lie sub-algebras of $(\d, \lara_\d)$, such as its
irreducible components and its $G \times G$ and $G_{\rm diag}$-orbits, was studied in \cite{e-l:cplx}.

Let $\g_{\rm diag} = \{(x, x)|x \in \g\}$.
Lagrangian splittings of $(\d, \lara_\d)$ of the form $\d = \g_{\rm diag} + \l$  are called {\it Belavin-Drinfeld splittings}, as they
were classified, up to conjugation by elements of $G_{\rm diag}$, by Belavin and Drinfeld \cite{BD:r}.
We denote any such splitting by
$\d = \g_{\rm diag} + \l_{\sBD}$ and the corresponding multiplicative
holomorphic Poisson structure on $G \cong G_{\rm diag}$ by
$\pi_{\sBD}$.  In particular, one has the
{\it standard splitting} $\d = \g_{\rm diag} + \l_{\rm st}$, where
\[
\l_{\rm st}= \{(x_+ + x_0, \, -x_0 + x_-)|\; x_0 \in \h, \, x_+ \in \n, \, x_- \in \n^-\},
\]
and the corresponding complex Poisson Lie group $(G, \pist)$ is the semi-classical limit of the most studied quantum group $\C_q[G]$ (see \cite{Drinfeld}).
By Remark \ref{rem:LLL} (applied to the holomorphic setting),
any $\l \in \calL(\d, \lara_\d)$ defines a $(G, \pi_{\sBD})$-affine holomorphic Dirac structure $E$ on $G$,
and Theorem \ref{thm:hol} gives a canonical construction of a holomorphic pre-symplectic groupoid over $G$ integrating $E$.

Let $\theta \in {\rm Aut}(G)$ and let $G_\theta = \{(\theta(g), g)|g \in G\}$. Identify
\[
(G \times G)/G_\theta \to G, \;\; (g_1, g_2)G_\theta \mapsto g_1\theta(g_2^{-1}),
\]
so that the $G_{\rm diag}$-orbits in $(G \times G)/G_\theta$ become $\theta$-twisted conjugacy classes in $G$. If $\theta$ is involutive,
the $\theta$-twisted conjugacy class through $e \in G$ can then be identified with the symmetric space $G/G^\theta$.
As the Lie algebra $\g_\theta$ of $G_\theta$ is lagrangian
with respect to $\lara_\d$, any Belavin-Drinfeld splitting gives rise to a holomorphic Poisson structure $\pi_{\sBD}^\prime$ on
$(G \times G)/G_\theta \cong G$, with respect to which every $\theta$-twisted conjugacy class in $G$ becomes a Poisson homogeneous space of
$(G, \pi_{\sBD})$.
Theorem \ref{thm:hol} then gives a holomorphic symplectic groupoid integrating $(C, \pi_{\sBD}^\prime)$
for every $\theta$-twisted conjugacy class $C$.
The complex Poisson manifold $(G, \pi_{\rm st}^\prime)$ corresponding to the standard splitting has been studied in \cite[Proposition 3.18]{ABM}
and \cite{e-l:Grothendieck} for $\theta = {\rm Id}$ and in \cite{Lu:twisted} for general $\theta$.

Let now $P$ be a parabolic subgroup of $G$.
As the Lie algebra $\p$ of $P$ is coisotropic with
respect to $\lara_\g$, each Belavin-Drinfeld splitting of $\d$ defines a Poisson structure $\Pi_\sBD$ on $(G \times G)/(P \times P) \cong (G/P) \times (G/P)$,
with respect to
which every $G_{\rm diag}$-orbit is a Poisson sub-manifold. In particular,
the unique closed $G_{\rm diag}$-orbit, identified with the (partial) flag variety $G/P$ of $G$, carries the Poisson structure $\Pi_\sBD$ making it into
a Poisson homogeneous space of $(G, \pi_{\sBD})$. By Proposition \ref{prop:GG-DQ},
$L=\{(p_1, p_2) \in P \times P|p_1p_2^{-1} \in N_\sP\}$ is a Lie subgroup of $G \times G$ integrating the Drinfeld lagrangian Lie sub-algebra associated to
 $(G/P, \Pi_\sBD)$ as a Poisson homogeneous space of $(G, \pi_{\sBD})$, where $N_P$ is the uniradical of $P$.
Since $L \supset P_{\rm diag} \cong P$, the quotient construction in
Theorem \ref{thm:hol} gives a holomorphic symplectic groupoid integrating the complex Poisson manifold $(G/P, \Pi_{\sBD})$.

Finally, regarding $G$ as a homogeneous space of itself, any two Belavin-Drinfeld lagrangian splittings $\d = \g_{\rm diag} + \l_{\sBD} =
\g_{\rm diag} + \l_{\sBD}^\prime$ define $(G, \pi_{\sBD})$-affine Poisson structures on $G$, and Theorem \ref{thm:hol} (by taking $H$ to be trivial) then
gives a construction of their symplectic groupoids (see
Example \ref{ex:PLinteg} for more details).

%\hen{re-include version of the appendix?}

\appendix

%%%%%%%%%%%%%%%%%%%%%%%%%%%%%%%%%%%%%%%%%%%%%%%%%%%%%%%%%%%%%%%%%%
\section{Deriving the formula for the pre-symplectic form on $\G(L)$}\label{appA}

In this appendix, we explain how the formula for $\omega_\sL$ in Theorem~\ref{thm:presymp} can be derived from general
 properties of multiplicative pre-symplectic forms on $\G(L)$.

%%%%%%%%%%%%%%%%%%%%%%%%%%%%%%%%%%%%%%%%%%%%%%%%%%%%%%%%%%%%%%%%
\subsection{General properties of pre-symplectic groupoids}\label{subsec:generalmult}

Let $\G\rightrightarrows M$ be a Lie groupoid with Lie algebroid $A$, anchor $\rho$ and Lie bracket $[\cdot,\cdot]$ on $\Gamma(A)$.
%Let $\mu: A\to T^*M$ be a vector-bundle map (covering the identity) such that $(\rho,\mu):A\to TM\oplus T^*M$ is injective
%and its image
%\begin{equation}\label{eq:E}
%E := \{(\rho(a),\mu(a))\,:\, a\in A\}\subset TM\oplus T^*M
%\end{equation}
%is a Dirac structure.
For $\omega\in \Omega^2(\G)$, denote by $\mu_\omega: A\to T^*M$ the map in \eqref{eq:mu-omega},
\begin{equation}\label{eq:mu}
\mu_\omega: A\to T^*M, \quad \mu_\omega(a)=i_a\omega|_{\scriptscriptstyle TM},
\end{equation}
and consider $\delta_\omega=(\rho,\mu_\omega): A \to \mathbb{T}M$, as in \eqref{eq:delta-omega}.

%By the discussion in $\S$\ref{ss:kernel} after Definition~\ref{defn:pre}, we have the following result.
%\begin{prop}\label{prop:dirac}\
%Let $\omega \in \Omega^2(\G)$ be multiplicative and closed. Then
%$(\G,\omega)$ is a pre-symplectic groupoid integrating $E$ provided
%$\mu_\omega =\mu$. (If $\G$ is source-simply-connected, the converse holds: modulo isomorphisms, a pre-symplectic form integrating $E$ satisfies  $\mu_\omega =\mu$.)
%\end{prop}

Let $(\G, \omega)$ be a pre-symplectic groupoid (Definition~\ref{defn:pre}) integrating the Dirac structure $E=\delta_\omega(A)$.
As we now see, in some situations one can find explicit ways to express $\omega$ in terms of $\mu_\omega$ and $\rho$ (and hence in terms of $E$).

The fact that $\omega$ is multiplicative implies that (cf. \eqref{eq:id1})
\begin{equation}\label{eq:omegamuapp}
\omega(\dif \rt_g(a), Y)|_g = \langle \mu_\omega(a), \dif \t |_g(Y) \rangle,
\end{equation}
for all $a \in A|_{\t(g)}$, $Y\in T_g\G$.
Consider the distribution $\mathsf{V}:=\ker(\dif \s)$,
$
\mathsf{V}|_g = \{ \dif \rt_g (a) \,:\, a\in A|_{\t(g)}\},
$
and suppose that we are given a ``horizontal'' distribution $\mathsf{H}$ such that $T\G= \mathsf{H}\oplus \mathsf{V}$. Using this decomposition to write an arbitrary element in $T\G|_g$ as  $X+ \dif \rt_g(a)$, and recalling that $\dif \t |_g (\dif \rt_g (b)) = \rho(b)$,
we have
\begin{align}\label{eq:wform}
\omega|_g(X+ \dif \rt_g(a),Y+ \dif \rt_g(b)) = & \langle \mu_\omega(a), \rho(b) \rangle  + \langle \mu_\omega(a), \dif \t |_g(Y) \rangle  \\ \nonumber & -\langle \mu_\omega(b), \dif \t |_g(X)  \rangle
 + \omega|_g(X,Y).
\end{align}
This expression shows that the value of $\omega$ at any $g\in \G$ is almost entirely determined by $\mu_\omega$, with the exception of the last term. %When $\omega$ integrates a Dirac structure  $E$ as in \eqref{eq:E} with $\mu_\omega=\mu$, this formula expresses $\omega$ using direct information from $E$.

A simple example where a horizontal distribution is canonically given is when $\G$ is the action groupoid for the action of a Lie group $K$ on $M$. Then $\G = M\times K$,  $\mathsf{V}=TK$, and we can take $\mathsf{H}=TM$. In this case, it was shown in \cite[$\S$6.4]{BCWZ} that the last term in \eqref{eq:wform}
defines a map
$$
c_\omega: K\to \Omega^2(M), \qquad c_\omega(k)|_x(X,Y):=\omega|_{(x,k)}(X,Y),
$$
which is an $\Omega^2(M)$-valued 1-cocycle on $K$. Moreover, it results from the condition $\dif \omega=0$ that the corresponding Lie algebra 1-cocycle $c_\omega': \mathfrak{k}\to \Omega^2(M)$ is given by $u\mapsto \dif (\mu_\omega(u))$. So the problem of determining the last term of \eqref{eq:wform} (which leads to a formula for $\omega$ integrating $E$) amounts to finding an integration of $c_\omega'$ to a 1-cocycle on $K$.

We will see that we have a similar picture in our context.

%*** so any tangent vector to $\G$ can be uniquely written as a sum $a^r + X$, with $X\in \mathsf{H}$.

%We also recall the following important property, proven in \cite{BCWZ}.

%%%%%%%%%%%%%%%%%%%%%%%%%%%%%%%%%%%%%%%%%%%%%%%%%%%%%%%%%%%%%%%%%
\subsection{The horizontal distribution in $\G(L)$}

Consider the Lie groupoid $\G(L)\rightrightarrows G$ introduced in \eqref{eq:GL}; its Lie algebroid is the action Lie algebroid $A=\l\ltimes G$, with anchor
\begin{equation}\label{eq:rhoapp-2}
\rho: \l\times G \to TG,\;\;\; \rho(u+\xi, g)= \rho_\d(u+\xi)|_g  =-(u^\lt + \pi_{\sG}^\sharp(\xi^\lt))|_g,
\end{equation}
where $\rho_\d$ is the dressing action of $\d$ on $G$. Although $\G(L)$ is not an action Lie groupoid, $T\G(L)$ splits into horizontal and vertical sub-bundles through the trivialization
$$
\G(L)\times (\g \oplus \l) \to T(\G(L)),\;\;\; (\gamma, w, a ) \mapsto (\widetilde{w}+ \widetilde{a})|_\gamma,
$$
defined by the restriction of \eqref{eq:TGD-decomp}. Recall that, for $\gamma=(k, g_1, g_2, l)$ and $a=u +\xi \in \l$,
\begin{equation}\label{eq:z}
\widetilde{a}|_\gamma  = \left (- (\Ad^*_{{g}_1^{-1}}\xi)^\lt|_{k}, \rho_\d (a)|_{g_1},0, a^\rt|_{l}   \right ),
\end{equation}
which agrees with the right-translation of $a$, and such vectors span the vertical distribution $\mathsf{V} = \ker (\dif \s)$.
The horizontal distribution $\mathsf{H}$ at $\gamma\in\G(L)$ is spanned by
\begin{equation}\label{eq:w}
\widetilde{w}|_\gamma = \left ( -\pi^\sharp_{\sG^*}(w^\rt)|_k, (\Ad^*_{k} w)^\rt|_{g_1},w^\rt|_{g_2},0)  \right ), \quad \; w\in \g.
\end{equation}

%%%%%%%%%%%%%%%%%%%%%%%%%%%%%%%%%%%%
\subsection{Towards the formula}

For $\omega \in \Omega^2(\G(L))$, consider the map
$\mu_\omega$ in \eqref{eq:mu}.
By viewing $a \in \l$ as a constant section of $\l \times G$, we regard
 $\rho(a)$ as a vector field, and $\mu_\omega(a)$ as a 1-form, on $G$.
Using the horizontal distribution of $\G(L)$, we can now use \eqref{eq:wform} to express any multiplicative 2-form $\omega$. Given $w\in \g$ and $a=u+\xi\in \l$, consider the corresponding horizontal and vertical vector fields $\widetilde{w}$ and $\widetilde{a}$ on $\G(L)$ given in
$\eqref{eq:z}$ and $\eqref{eq:w}$. For $\gamma=(k,g_1,g_2,l)\in \G(L)$, it is clear from the groupoid structure on $\G(L)$ that
\begin{equation}\label{eq:dt}
\dif \t (\widetilde{w}|_\gamma) = (\Ad_k^*w)^\rt|_{g_1} , \qquad \dif \t (\widetilde{a}|_\gamma) =  \rho_\d(a)|_{g_1}.
\end{equation}
Since $\widetilde{a}$ is just the right-invariant vector field on $\G(L)$ determined by the (constant) section $a$ of $\l \times G$, we see that \eqref{eq:wform} takes the form
\begin{align}
\omega(\widetilde{w}_1 + \widetilde{a}_1,\widetilde{w}_2+ \widetilde{a}_2)|_{\gamma} = &\langle \mu_\omega(a_1),\rho(a_2)\rangle|_{g_1} + \langle \mu_\omega (a_1), (\Ad_k^*w_2)^\rt \rangle|_{g_1} \nonumber \\
&- \langle \mu_\omega (a_2),(\Ad_k^*w_1)^\rt\rangle|_{g_1}
+\omega|_\gamma (\widetilde{w}_1,\widetilde{w}_2), \label{eq:omegamu}
\end{align}
for $w_i\in \g$, $a_i=u_i+\xi_i\in \l$, $i=1,2$.

Consider the Dirac structure
$E={\bf e}(\l_\sG)$ (see \eqref{eq:I}),
\begin{equation}\label{eq:Eapp}
E|_g=\{(\rho_\d(a), -\xi^{\lt})|_g: a= u+\xi \in \l\subset \d \}.
\end{equation}
Setting
\begin{equation}\label{eq:muapp}
\mu: \l\times G \to T^*G,\qquad \mu(u+\xi, g) = -\xi^\lt|_g,
\end{equation}
we see that the map $(\rho,\mu): \l\times G \to TG\oplus T^*G$ is injective with image $E$
(injectivity follows since $\dim (\l)$ equals the rank of $E$).
By the discussion following Definition~\ref{defn:pre}, in order for a closed, multiplicative $\omega \in \Omega^2(\G(L))$  to integrate $E$,
one should have  $\mu_\omega = \mu$, so the formula \eqref{eq:omegamu} for such $\omega$ becomes
\begin{align}
\omega(\widetilde{w}_1 + \widetilde{a}_1,\widetilde{w}_2+ \widetilde{a}_2)|_{\gamma} = & \langle \xi_1 ,u_2\rangle - \pi_{\sG}|_{g_1}(\xi_1^\lt,\xi_2^\lt) - \langle \Ad^*_{g_1^{-1}}\xi_1, \Ad_k^*w_2 \rangle \nonumber\\& + \langle \Ad^*_{g_1^{-1}} \xi_2,\Ad_k^*w_1\rangle
+ \omega|_\gamma (\widetilde{w}_1,\widetilde{w}_2). \label{eq:omegamu2}
\end{align}

%\subsubsection*{The missing cocycle}

In analogy with the case of action Lie groupoids recalled in $\S$\ref{subsec:generalmult}, we now explain the sense in which the last term in \eqref{eq:omegamu2} is a groupoid 1-cocycle.

Consider the left representation of $\G(L)$ on $TG$ by
\begin{equation}\label{eq:act}
\gamma\cdot (w^\rt|_{g_2}) := \dif \t (\widetilde{w}|_\gamma)  =(\Ad^*_{k}w)^\rt|_{g_1}, \hspace{.2in}
\gamma = (k, g_1, g_2, l) \in \G(L).
\end{equation}
Correspondingly, one has the induced left representation of $\G(L)$ on $T^*G$, as well as on its exterior powers,
by $\gamma \cdot \alpha := (\gamma^{-1})^*\alpha$. Recall that a 1-cocycle on $\G(L)$ with coefficients in $\wedge^2T^*G$
is a map $c: \G(L)\to \s^*\wedge^2T^*G$ satisfying
\begin{equation}\label{eq:cocyc}
c(\gamma\gamma') = (\gamma')^{-1}.c(\gamma) + c(\gamma') =  (\gamma')^*c(\gamma) + c(\gamma')
\end{equation}
for all composable $\gamma,\gamma' \in \G(L)$.
For
$\omega \in \Omega^2(\G(L))$, define $c_\omega: \G(L)\to \s^*\wedge^2T^*G$  by
\begin{equation}\label{eq:c}
c_\omega(\gamma)(w_1^\rt|_{g_2},w_2^\rt|_{g_2}):= \omega(\widetilde{w}_1,\widetilde{w}_2)|_\gamma,
\hspace{.2in} \gamma = (k, g_1, g_2, l) \in \G(L).
\end{equation}

\medskip

\noindent{{\bf Claim 1:}} {\em If $\omega \in \Omega^2(\G(L))$ is multiplicative, then
 $c_\omega$ is a 1-cocycle on $\G(L)$ with coefficients in $\wedge^2T^*G$.}

\medskip
To prove Claim 1, we consider the tangent map
\[
Tm: T(\G(L)^{(2)})=T\G(L)\times_{TM}T\G(L)\to T\G(L),
\]
which defines a groupoid multiplication on $T\G(L)$, denoted by $\circ$.
If $(X_k,Y_{g_1},w^\rt|_{g_2},a^\rt|_l)$ is  tangent to $\G(L)$ at $\gamma=(k,g_1,g_2,l)$, then
$$
Y_{g_1} = (\Ad^*_{k}w)^\rt|_{g_1} +  \rho_\d (a)|_{g_1}.
$$
Let $\gamma=(k,g_1,g_2,l)$ and $\gamma'=(k', g_1^\prime, g_2^\prime, l^\prime)$ be composable, so that $g_2={g}_1'$.
If
$$
(\widetilde{w}+\widetilde{a})|_\gamma=(X_k,Y_{g_1},w^\rt|_{g_2},a^\rt|_l)\;\; \mbox{ and }\;\;
(\widetilde{w}'+\widetilde{a}')|_{\gamma'} = ({X}'_{{k}'},{Y}'_{{g}_1'},{w'}^\rt|_{{g}_2^\prime},{a'}^\rt|_{{l'}})
$$
are composable, then
$w^\rt|_{g_2} = {Y}_{{g}_1'}' = (\Ad^*_{{k'}}{w'})^\rt|_{{g}_1'} +  \rho_\d ({a'})|_{{g}_1'}$, and one has
\begin{align}
%(X_v,Y_{g_1},w^\rt|_{g_2},z^r|_l)\circ ({X}_{{v'}}',{Y}_{{g}_1'}',{w'}^\rt|_{{g}_2'},{z'}^\rt|_{{l'}})
(\widetilde{w}+\widetilde{a})|_\gamma \circ (\widetilde{w}'+\widetilde{a}')|_{\gamma'}= &
(Tm_{\sG^*}({X}_{{k'}}',X_k), Y_{g_1}, {w'}^\rt|_{{g}_2'},Tm_{\sL}(a^\rt_l,{a'}^\rt_{{l'}}))\nonumber \\
=& (Tm_{\sG^*}({X}_{{k'}}',X_k), Y_{g_1}, {w'}^\rt|_{{g}_2'},(a + \Ad_{l}{a'})^\rt|_{l{l'}})\nonumber \\
=& (\widetilde{w}' + \widetilde{(a + \Ad_{l}{a'})})|_{\gamma\gamma'},\label{eq:prod}
\end{align}
where $m_{\sG^*}$ and $m_{\sL}$ denote the respective multiplications on $G^*$ and $L$, and we used
that $Tm_{\sL}(a^\rt|_l,{a'}^\rt|_{{l'}}) = (a + \Ad_{l}{a'})^\rt|_{l{l'}}$.
Note that $\widetilde{w}_{\gamma\gamma'} = \widetilde{(\Ad^*_{k'} w)}_\gamma \circ \widetilde{w}_{\gamma'}$,
so the multiplicativity of $\omega$ implies that
\begin{align*}
\omega_{\gamma\gamma'}(\widetilde{w}_1,\widetilde{w}_2) &=
\omega_{\gamma\gamma'}(\widetilde{(\Ad^*_{k'} w_1)}_\gamma \circ (\widetilde{w}_1)_{\gamma'},
\widetilde{(\Ad^*_{k'} w_2)}_\gamma \circ(\widetilde{w}_2)_{\gamma'})\\
& = \omega_\gamma(\widetilde{(\Ad^*_{k'} w_1)},\widetilde{(\Ad^*_{k'} w_2)}) +
\omega_{\gamma'}(\widetilde{w}_1,\widetilde{w}_2),
\end{align*}
which is exactly \eqref{eq:cocyc}. This proves Claim 1.

\subsection{Finding the missing cocycle}
%\medskip
 A groupoid 1-cocycle $c: \G(L)\to \s^*\wedge^2T^*G$ has an infinitesimal counterpart (see \cite{crainic}), which is a Lie algebroid 1-cocycle:
$$
c': \l\ltimes G \to \wedge^2T^*G, \;\;\; c'(a, g) = \frac{d}{d\epsilon}\Big|_{\epsilon=0} c(\alpha_\epsilon),
$$
where $\alpha_\epsilon$ is a curve on  $\s^{-1}(g)$, passing through $1_g\in \G(L)$ with tangent vector $\widetilde{a}|_{1_g}$ at $\epsilon=0$.
For $a \in \l$, we let $c'(a) \in \Omega^2(G)$ be such that $c'(a)|_g = c'(a, g)$.
(If the Lie groupoid is source connected, then a groupoid 1-cocycle is uniquely determined by the corresponding infinitesimal 1-cocycle.)

\medskip

\noindent{{\bf Claim 2:}} {\em A multiplicative $\omega \in \Omega^2(\G(L))$ satisfies
$\dif \omega(\widetilde{a},\widetilde{w}_1,\widetilde{w}_2) =0$ for all
$a\in \l$, $w_1,w_2\in \g$ if and only if
the groupoid 1-cocycle $c_\omega$ in \eqref{eq:c} satisfies $c_\omega^\prime(a,g) = \dif \mu_\omega(a)|_g$ for all $g \in G$ and  $a \in \l$.}

\medskip
To prove Claim 2, note that horizontal and vertical vector fields satisfy
$$
[\widetilde{w}_1,\widetilde{w}_2]=\widetilde{[w_1,w_2]},\;\;\; [\widetilde{w},\widetilde{a}]=0
$$
for all $a\in \l$ and $w, w_1,w_2\in \g$.
Using the second equation, we see that
$$
\dif \omega(\widetilde{a},\widetilde{w}_1,\widetilde{w}_2)  =  \Lie_{\widetilde{a}}(\omega(\widetilde{w}_1,\widetilde{w}_2)) -\Lie_{\widetilde{w}_1}(\omega(\widetilde{a},\widetilde{w}_2)) + \Lie_{\widetilde{w}_2}(\omega(\widetilde{a},\widetilde{w}_1))+\omega(\widetilde{a},[\widetilde{w}_1,\widetilde{w}_2]),
$$
so the vanishing of $\dif \omega(\widetilde{a},\widetilde{w}_1,\widetilde{w}_2)$ is equivalent to the identity
\begin{align}
\Lie_{\widetilde{a}}(\omega(\widetilde{w}_1,\widetilde{w}_2))  &= \Lie_{\widetilde{w}_1}(\omega(\widetilde{a},\widetilde{w}_2)) - \Lie_{\widetilde{w}_2}(\omega(\widetilde{a},\widetilde{w}_1))-\omega(\widetilde{a},[\widetilde{w}_1,\widetilde{w}_2])\nonumber\\
&= \Lie_{\widetilde{w}_1}\langle \t^*\mu_\omega(a),\widetilde{w}_2\rangle - \Lie_{\widetilde{w}_2}\langle \t^*\mu_\omega({a}),\widetilde{w}_1\rangle- \langle \t^*\mu_\omega({a}), [\widetilde{w}_1,\widetilde{w}_2] \rangle\nonumber\\
&=\t^*(\dif \mu_\omega(a))(\widetilde{w}_1,\widetilde{w}_2).\label{eq:case3}
\end{align}
%Here $z\in \l$ is seen as a constant section, so $\mu(z)$ is a 1-form on $G$.

On the other hand, for $\gamma = (k,g_1,g_2,l) \in \G(L)$, let
$\gamma_\epsilon$ be a curve in $\s^{-1}(g_2)$ passing through $\gamma$ with tangent vector $\widetilde{a}|_\gamma$ at $\epsilon =0$, and let $\alpha_\epsilon = \gamma_\epsilon \gamma^{-1}$. Then
\begin{align*}
\Lie_{\widetilde{a}}(\omega(\widetilde{w}_1,\widetilde{w}_2))|_\gamma & = \frac{d}{d\epsilon}\Big|_{\epsilon =0}
c_\omega(\gamma_\epsilon)(w_1^\rt|_{g_2},w_2^\rt|_{g_2}) = \frac{d}{d\epsilon}\Big|_{\epsilon =0}
 c_\omega(\alpha_\epsilon\gamma)(w_1^\rt|_{g_2},w_2^\rt|_{g_2})\\
&= \frac{d}{d\epsilon}\Big|_{\epsilon =0} (\gamma^*c_\omega(\alpha_\epsilon) + c(\gamma))(w_1^\rt|_{g_2},w_2^\rt|_{g_2})
=c_\omega'(\gamma)(\gamma\cdot(w_1^\rt|_{g_2}),\gamma\cdot (w_2^\rt|_{g_2}))\\
&=\t^*(c_\omega'(a))(\widetilde{w}_1,\widetilde{w}_2)|_{\gamma},
\end{align*}
where the last equality uses \eqref{eq:act}.
By comparison, we see that \eqref{eq:case3} holds if and only if $c_\omega'(a) = \dif \mu_\omega(a)$, proving Claim 2.

\medskip

Assuming that $\omega$ is a pre-symplectic form integrating $E = {\mathbf e}(\l_\sG)$ such that $\mu_\omega(u+\xi) = -\xi^\lt$, by Claim 2, the  groupoid 1-cocycle
$c_\omega: \G(L)\to \s^*\wedge^2TG$ must satisfy $c_\omega^\prime(u+\xi)=-\dif \xi^\lt$ for all $u + \xi \in \l$.
We now give such a groupoid 1-cocycle.

\medskip

\noindent {{\bf Claim 3}}: {\em The map $c: \G(L)\to \s^*\wedge^2T^*G$ given by
\begin{equation}\label{eq:cpi}
c(\gamma)(w_1^\rt|_{g_2},w_2^\rt|_{g_2}):= -(\rt_{k^{-1}}\pi_{\sG^*}|_k)(w_1,w_2)  \;\;\; \mbox{for}\;\; \gamma=(k,g_1,g_2,l)
\end{equation}
is a 1-cocycle on $\G(L)$ satisfying $c'(u+\xi)=-\dif \xi^\lt$ for all $u + \xi \in \l$.}

\medskip
Indeed, the fact that $c$ is a 1-cocycle follows from the multiplicativity of $\pi_{\sG^*}$.
%: for $\gamma=(k,g_1,g_2,l)$ and $\gamma'=(k',g_1',g_2',l')$, we have
%\begin{align*}
%c(\gamma\gamma')(w_1^\rt|_{g_2'},w_2^\rt|_{g_2'}) &= - \rt_{(k'k)^{-1}}\pi_{\sG^*}|_{k'k}(w_1,w_2)\! =\! -\rt_{(k')^{-1}}\rt_{k^{-1}}(\lt_{k'}\pi_{\sG^*}|_{k} + \rt_{k}\pi_{\sG^*}|_{k'})(w_1,w_2)\\
%&= -(\Ad_{k'}\rt_{k^{-1}}\pi_{\sG^*}|_{k} + \rt_{(k')^{-1}} \pi_{\sG^*}|_{k'})(w_1,w_2)\\
%&= c(\gamma)(\gamma'\cdot (w_1^\rt|_{g_2'}),\gamma'\cdot (w_2^\rt|_{g_2'})) + c(\gamma')(w_1|_{g_2'}^\rt,w_2|_{g_2'}^\rt)\\
%&= ((\gamma')^*c(\gamma) + c(\gamma'))(w_1^\rt|_{g_2'},w_2^\rt|_{g_2'}).
%\end{align*}
To calculate $c'$, let $a= u + \xi \in \l$ and consider a curve $\gamma_\epsilon= (k(\epsilon),g_1(\epsilon),$ $g,l(\epsilon))$ passing through $1_g=(1,g,g,1)$ at
$\epsilon=0$ with tangent vector $\widetilde{a}|_{1_g}$. Then
\begin{align*}
c'(z)(w_1^\rt|_g,w_2^\rt|_g) &
%= \frac{d}{d\epsilon}\Big|_{\epsilon=0}c(\sigma_\epsilon)(w_1^\rt|_g,w_2^\rt|_g)
=-\frac{d}{d\epsilon}\Big|_{\epsilon=0}(\rt_{k_\epsilon^{-1}}\pi_{\sG^*}|_{k_\epsilon})(w_1,w_2) = \left (\frac{d}{d\epsilon}\Big|_{\epsilon=0}k_\epsilon \right ) ([w_1,w_2]) \\
&= -\Ad^*_{g^{-1}}\xi([w_1,w_2]) = -\xi([\Ad_{g^{-1}}w_1,\Ad_{g^{-1}}w_2])\\
&= \xi^\lt([(\Ad_{g^{-1}}w_1)^\lt,(\Ad_{g^{-1}}w_2)^\lt])|_g=-\dif \xi^\lt (w_1^\rt|_g,w_2^\rt|_g).
\end{align*}
This proves Claim 3.

Combining with \eqref{eq:omegamu2}, we are led to the expression
%conclude that when $\G(L)$ is source fiber connected, any multiplicative pre-symplectic
%form $\omega$ on $\G(L)$ integrating $E = I(\l_\sG)$ must be given by
\begin{align}\label{eq:omegaapp}
\omega(\widetilde{w}_1 + \widetilde{a}_1,\widetilde{w}_2+ \widetilde{a}_2)|_{\gamma} =
& \langle \xi_1,u_2\rangle - \pi_{\sG}|_{g_1}(\xi_1^\lt,\xi_2^\lt) - \langle \Ad^*_{g^{-1}_1}\xi_1,
Ad_k^*w_2 \rangle \nonumber\\& + \langle \Ad^*_{g^{-1}_1} \xi_2,\Ad_k^*w_1\rangle
- \pi_{\sG^*}|_k({w}^\rt_1,{w}^\rt_2),
\end{align}
which agrees with \eqref{eq:omegaL-explicit}.
If $\G(L)$ is source simply-connected,  then, by the uniqueness of the integrations of Dirac structures and Lie-algebroid cocycles, we see that a multiplicative pre-symplectic form $\omega$ on $\G(L)$ integrating $E = \mathbf{e}(\l_\sG)$ must be (isomorphic to one) given by \eqref{eq:omegaapp}.

%%%%%%%%%%%%%%%%%%%%%%%%%%%%%%%%%%%%%%%%%%%%%%%%%%%%%%

\small{

}
\end{document}